\documentclass{amsart}

\usepackage{amscd}
\usepackage{amssymb}
\usepackage[utf8]{inputenc}
\usepackage[arrow,curve,matrix]{xy}
\usepackage{times}
\usepackage{graphicx}
\usepackage{color}
\usepackage{comment}
\usepackage{enumitem}
\usepackage{varioref}
\usepackage{xstring}
\usepackage{hyperref}
\usepackage{xspace}

\setitemize[1]{leftmargin=*,parsep=0em,itemsep=0.125em,topsep=0.125em}

\DeclareFontFamily{OMS}{rsfs}{\skewchar\font'60}
\DeclareFontShape{OMS}{rsfs}{m}{n}{<-5>rsfs5 <5-7>rsfs7 <7->rsfs10 }{}
\DeclareSymbolFont{rsfs}{OMS}{rsfs}{m}{n}
\DeclareSymbolFontAlphabet{\scr}{rsfs}

\renewcommand{\P}{\mathbb{P}} \newcommand{\C}{\mathbb{C}}
\renewcommand{\O}{\sO}

\newcommand\resto[1]{\hbox{\hbox{$\big\vert{}_{_{#1}}$}}}

\newcommand{\rdown}[1]{\left\lfloor{#1}\right\rfloor}

\newcommand{\wlr}{log resolution\xspace}
\newcommand{\slr}{strong log resolution\xspace}
\newcommand{\Wlr}{Log resolution\xspace}
\newcommand{\Slr}{Strong log resolution\xspace}
\newcommand{\wlrs}{log resolutions\xspace}
\newcommand{\slrs}{strong log resolutions\xspace}
\newcommand{\Wlrs}{Log resolutions\xspace}
\newcommand{\Slrs}{Strong log resolutions\xspace}

\newenvironment{sequation}{%
\numberwithin{equation}{section}%
\begin{equation}%
}{%
\end{equation}%
\numberwithin{equation}{thm}%
\addtocounter{thm}{1}%
}

\newcommand{\sA}{\scr{A}}
\newcommand{\sB}{\scr{B}}
\newcommand{\sC}{\scr{C}}

\newcommand{\sF}{\scr{F}}
\newcommand{\sG}{\scr{G}}
\newcommand{\sH}{\scr{H}}

\newcommand{\sL}{\scr{L}}

\newcommand{\sO}{\scr{O}}

\newcommand{\sS}{\scr{S}}
\newcommand{\sT}{\scr{T}}

\newcommand{\bC}{\mathbb{C}}

\newcommand{\bH}{\mathbb{H}}

\newcommand{\bN}{\mathbb{N}}

\newcommand{\bQ}{\mathbb{Q}}

\DeclareMathOperator{\Diff}{Diff}
\DeclareMathOperator{\Exc}{Exc}

\DeclareMathOperator{\codim}{codim}
\DeclareMathOperator{\coker}{coker}

\DeclareMathOperator{\Image}{Im} 

\DeclareMathOperator{\ord}{ord}
\DeclareMathOperator{\Pic}{Pic}

\DeclareMathOperator{\red}{red}
\DeclareMathOperator{\reg}{reg}
\DeclareMathOperator{\sing}{sing}
\DeclareMathOperator{\Sing}{Sing}
\DeclareMathOperator{\Sym}{Sym}
\DeclareMathOperator{\supp}{supp}

\newcommand{\into}{\hookrightarrow}

\newcommand{\ol}{\overline}
\newcommand{\ul}{\underline}
\newcommand{\wt}{\widetilde}

\newcommand{\wtilde}{\widetilde}
\newcommand{\what}{\widehat}

\newcommand{\DuBois}[1]{{\uline \Omega^0_{#1}}}
\newcommand{\FullDuBois}[1]{{\uline \Omega^{\updot}_{#1}}}
\newcommand{\uline}{\underline}
\newcommand{\updot}{
  \begin{picture}(0,0)(-.05,-.1)
    \circle*{.055}
  \end{picture}}

\newcommand{\leteq}{\colon\!\!\!=}

\newcommand{\Om}{\underline{\Omega}}
\newcommand{\Ox}[1]{\underline{\Omega}_X^{\,#1}}

\newcounter{thisthm}
\newcommand{\ilabel}[1]{\newcounter{#1}\setcounter{thisthm}{\value{thm}}\setcounter{#1}{\value{enumi}}}
\newcommand{\iref}[1]{(\thesection.\the\value{thisthm}.\the\value{#1})}

\newcommand{\jlabel}[1]{%
}
\newcommand{\jref}[1]{%
}

\theoremstyle{plain}
\newtheorem{thm}{Theorem}[section]
\newtheorem{defn}[thm]{Definition}

\newtheorem{setup}[thm]{Setup}
\numberwithin{equation}{thm}
\theoremstyle{plain}
\newtheorem{cor}[thm]{Corollary}
\newtheorem{lem}[thm]{Lemma}

\theoremstyle{plain}

\newtheorem{fact}[thm]{Fact}
\newtheorem{num}[thm]{}
\newtheorem{prop}[thm]{Proposition}
\newtheorem{proclaim-special}[thm]{\specialthmname}

\theoremstyle{remark}
\newtheorem{rem}[thm]{Remark}
\newtheorem{q}[thm]{Question}
\newtheorem{obs}[thm]{Observation}
\newtheorem{explanation}[thm]{Explanation}

\newtheorem{subrem}[equation]{Remark}
\newtheorem{claim}[thm]{Claim} 
\newtheorem*{claim*}{Claim}
\newtheorem{ex}[thm]{Example}
\newtheorem{notation}[thm]{Notation}
\newtheorem{cons}[thm]{Observation} 
\newtheorem{conv}[thm]{Convention}

\newtheorem{application-idea}[thm]{Idea of Application}
\newtheorem{awlog}[thm]{Additional Assumption} 

\newtheoremstyle{bozont-remark}{3pt}{3pt}%
     {}
     {}
     {\sc}
     {.}
     {.5em}
     {\thmname{#1}\thmnumber{ #2}} 
\theoremstyle{bozont-remark}
\newtheorem{rem-special}[thm]{\specialthmname}

\def\factor#1.#2.{\left. \raise 2pt\hbox{$#1$} \right/\hskip -2pt\raise
  -2pt\hbox{$#2$}}

\setenumerate[1]{leftmargin=*,parsep=0em,itemsep=0.125em,topsep=0.125em}
\newlength{\swidth}
\newenvironment{enumerate-p}{
  \begin{enumerate}}
  {\setcounter{equation}{\value{enumi}}\end{enumerate}}
\newenvironment{enumerate-cont}{
  \begin{enumerate}
    {\setcounter{enumi}{\value{equation}}}}
  {\setcounter{equation}{\value{enumi}}
  \end{enumerate}}

\date{\today}

\author{Daniel Greb}
\author{Stefan Kebekus}
\author{S\'andor J Kov\'acs}
\author{Thomas Peternell}

\thanks{Daniel Greb was supported in part by an MSRI postdoctoral fellowship
  during the 2009 special semester in Algebraic Geometry. Stefan Kebekus and
  Thomas Peternell were supported in part by the DFG-Forschergruppe
  ``Classification of Algebraic Surfaces and Compact Complex Manifolds''.
  S\'andor Kov\'acs was supported in part by NSF Grants DMS-0554697 and
  DMS-0856185, and the Craig McKibben and Sarah Merner Endowed Professorship in
  Mathematics.\\[1mm]
  A shortened version of this paper without color figures has appeared in \emph{Publications
    mathématiques de l'IHÉS},
  \href{http://dx.doi.org/10.1007/s10240-011-0036-0}{DOI:10.1007/s10240-011-0036-0}.
  The final publication is available at \url{www.springerlink.com}.}

\address{Daniel Greb, Mathematisches Institut, Albert-Ludwigs-Universit\"at Freiburg,
  Eckerstraße 1, 79104 Freiburg im Breisgau, Germany}
\email{\href{mailto:daniel.greb@math.uni-freiburg.de}{daniel.greb@math.uni-freiburg.de}}
\urladdr{\href{http://home.mathematik.uni-freiburg.de/dgreb}{http://home.mathematik.uni-freiburg.de/dgreb}}

\address{Stefan Kebekus, Mathematisches Institut, Albert-Ludwigs-Universit\"at
  Freiburg, Eckerstraße 1, 79104 Freiburg im Breisgau, Germany}
\email{\href{mailto:stefan.kebekus@math.uni-freiburg.de}{stefan.kebekus@math.uni-freiburg.de}}
\urladdr{\href{http://home.mathematik.uni-freiburg.de/kebekus}{http://home.mathematik.uni-freiburg.de/kebekus}}

\address{S\'andor Kov\'acs, University of Washington, Department of Mathematics, Box
  354350, Seattle, WA 98195, U.S.A.}
\email{\href{mailto:kovacs@math.washington.edu}{kovacs@math.washington.edu}}
\urladdr{\href{http://www.math.washington.edu/~kovacs}{http://www.math.washington.edu/$\sim$kovacs}}

\address{Thomas Peternell, Institut f\"ur Mathematik, Universit\"at Bayreuth,
  95440~Bayreuth, Germany}
\email{\href{mailto:thomas.peternell@uni-bayreuth.de}{thomas.peternell@uni-bayreuth.de}}
\urladdr{\href{http://btm8x5.mat.uni-bayreuth.de/mathe1}{http://btm8x5.mat.uni-bayreuth.de/mathe1}}

\dedicatory{In memory of Eckart Viehweg}

\definecolor{tomato}{RGB}{180,62,39}
\definecolor{forrest}{RGB}{81,133,49}
\definecolor{lighttomato}{RGB}{253,65,65}
\definecolor{lightforrest}{RGB}{145,237,87}
\definecolor{mygreen}{RGB}{40,104,69}
\definecolor{mygreen2}{RGB}{3,149,39}
\definecolor{darkolivegreen}{RGB}{102,118,75}
\definecolor{cranegreen}{RGB}{102,118,75}
\definecolor{mydarkblue}{RGB}{10,92,153}
\definecolor{myblue}{RGB}{57,222,186}
\definecolor{pinkish}{RGB}{213,83,222}
\definecolor{colD}{RGB}{213,83,222}
\definecolor{defb}{RGB}{213,83,222}
\definecolor{goldenrod}{RGB}{225,115,69}
\definecolor{mauve}{RGB}{224, 176, 255}
\definecolor{fuchsia}{RGB}{255, 0, 255}
\definecolor{lavender}{RGB}{230, 230, 250}
\definecolor{gold}{RGB}{255, 215, 0}
\definecolor{orange}{RGB}{255, 127, 0}
\definecolor{maroon}{RGB}{123, 17, 19}
\definecolor{brightmaroon}{RGB}{195, 33, 72}
\definecolor{richmaroon}{RGB}{176, 48, 96}

\definecolor{green}{RGB}{3,149,39}

\newcommand{\PreprintAndPublication}[2]{#1}

\newcommand{\change}[1]{#1}

\title{Differential forms on log canonical spaces}

\begin{document}

\begin{abstract}
  The present paper is concerned with differential forms on log canonical
  varieties. It is shown that any $p$-form defined on the smooth locus of a
  variety with canonical or klt singularities extends regularly to any
  resolution of singularities. In fact, a much more general theorem for log
  canonical pairs is established. The proof relies on vanishing theorems for
  log canonical varieties and on methods of the minimal model program. In
  addition, a theory of differential forms on dlt pairs is developed. It is
  shown that many of the fundamental theorems and techniques known for sheaves
  of logarithmic differentials on smooth varieties also hold in the dlt
  setting.

  Immediate applications include the existence of a pull-back map for
  reflexive differentials, generalisations of Bogomolov-Sommese type vanishing
  results, and a positive answer to the Lipman-Zariski conjecture for klt
  spaces.
\end{abstract}

\maketitle
\tableofcontents

\part{INTRODUCTION}

\section{Introduction}

Differential forms play an essential role in the study of algebraic
varieties. On a smooth complex variety $X$ of dimension $n$ the sheaf
$\omega_X = \Omega^n_X$ of $n$-forms is of particular importance as it appears
both in Serre duality and in the Kodaira vanishing theorem. As observed by
Grauert and Riemenschneider, these two roles do not generalise the same way to
the singular case. If $X$ is singular, there are several possible definitions
for the sheaf of $n$-forms, depending on which of the properties one would
like to keep. In general, there is one definition that preserves the role of
differentials in duality theory and another one suitable for vanishing
theorems.

\subsubsection*{A simple case}

Consider the case when $X$ is normal and Gorenstein. In this setting the
dualising sheaf $\omega_X$ is locally free, and Serre duality holds the same
way as in the smooth case. In contrast, the Kodaira vanishing theorem fails in
general. There exist a Gorenstein variety $X$ with ample line bundle $\sL \in
\Pic X$ such that $H^1\bigl( X,\, \omega_X \otimes \sL \bigr) \not = 0$,
\cite[Sect.~3.3]{GR70}.  However, when $\pi : \wtilde X \to X$ is a resolution
of singularities and $\wtilde \omega_X := \pi_*\omega_{\wtilde X}$, then there
exists an inclusion $\wtilde \omega_X \subseteq \omega_X$, the subsheaf
$\wtilde \omega_X$ is independent of the resolution, and Kodaira vanishing
holds for $\wtilde \omega_X$ by \cite[Thm.~2.1]{GR70}.  Consequently, there
are two sheaves on $X$ that generalise the notion of the canonical line bundle
of a smooth variety: $\omega_X$ works for duality, $\wtilde \omega_X$ for
vanishing.

Given the importance of duality and vanishing theorems in complex algebraic
geometry, the following question seems natural in this context.

\begin{q}\label{q:q1}
  Given a normal Gorenstein variety $X$, when do the sheaves $\omega_X$ and
  $\wtilde \omega_X$ agree?
\end{q}

To answer this question, recall that $\omega_X$ is locally free and therefore
reflexive. If $U \subseteq X$ is any open subset, to give a section $\tau \in
\omega_X(U)$, it is therefore equivalent to give an $n$-form on the smooth
locus of $U$. In other words, to give a section $\tau \in \omega_X(U)$, it is
equivalent to give an $n$-form $\tau' \in \omega_{\wtilde X} \bigl(\pi^{-1}(U)
\setminus E \bigr)$, where $E \subset \wtilde X$ is the exceptional locus of
the resolution map $\pi$. In contrast, a section $\sigma \in \wtilde
\omega_X(U)$ is, by definition, an $n$-form $\sigma' \in \omega_{\wtilde
  X}\bigl(\pi^{-1}(U) \bigr)$.

In summary, we obtain the following equivalent reformulation of
Question~\ref*{q:q1}.

\begin{q}\label{q:q2}
  When is it true that any $n$-form, defined on an open set of the form
  $\pi^{-1}(U) \setminus E \subset \wtilde X$ extends across $E$, to give a
  form on $\pi^{-1}(U)$?
\end{q}

The answer to Question~\ref{q:q2} is almost a tautology: it follows directly
from the definition that $X$ has canonical singularities if and only if any
$n$-form $\pi^{-1}(U) \setminus E$ extends across $E$.  The fact that spaces
with canonical singularities have a single sheaf that works for both duality
and vanishing is one of the reasons for their importance in higher dimensional
algebraic geometry.

\subsubsection*{Main result of this paper}

This paper aims to answer Question~\ref{q:q2} for differential forms of degree
$p$, where $p \leq n$ and where $X$ is not necessarily Gorenstein. The main
results, formulated in Theorems~\ref{thm:take-one} and \ref{thm:main} below,
assert that if $X$ is log terminal, then any $p$-form will extend.  Our
results also hold in the logarithmic setup, for log canonical pairs. Immediate
applications concern vanishing theorems and other properties of differential
forms on log canonical varieties.

\subsubsection*{Formulation using reflexive sheaves}

Extension properties of differential forms can be expressed in terms of
reflexivity of push-forward sheaves. Although perhaps not quite intuitive at
first sight, this language is technically convenient. The following
observation relates reflexivity and extension properties and will be used
throughout the paper.

\begin{obs}\label{obs:13}
  Let $X$ be a normal variety, and $\pi: \wtilde X \to X$ a resolution of
  singularities, with exceptional set $E \subset \wtilde X$. If $\sA$ is any
  locally free sheaf on $\wtilde X$, then $\pi_* \sA$ is torsion free, but not
  necessarily reflexive. Using that $\codim_X \pi(E) \geq 2$, observe that
  $\pi_* \sA$ reflexive if and only if any section of $\pi_* \sA|_{X \setminus
    \pi(E)}$ extends to $X$. Equivalently, $\pi_* \sA$ is reflexive if and
  only if any section of $\sA$, defined on an open set of the form
  $\pi^{-1}(U) \setminus E$ extends to $\pi^{-1}(U)$.
\end{obs}

\subsection{Main results}\label{ssec:main}

The main result of this paper gives necessary and sufficient conditions that
guarantee reflexivity of $\pi_* \Omega^p_{\wtilde X}$ for all $p \leq \dim
X$. Equivalently, the main result gives necessary and sufficient conditions to
guarantee that any differential $p$-form on $\wtilde X$, defined away from the
exceptional set $E$ extends across $E$. The simplest form of our main result
is the following.

\begin{thm}[Extension theorem for differential forms on klt varieties]\label{thm:take-one}
  Let $X$ be a complex quasi-projective variety with at most klt (Kawamata log
  terminal) singularities and $\pi: \wtilde X \to X$ a \wlr. Then $\pi_*
  \Omega^{p}_{\wtilde X}$ is reflexive for all $p \leq \dim X$.
\end{thm}

\begin{subrem}
  Gorenstein klt varieties have canonical singularities. The statement of
  Theorem~\ref{thm:take-one} therefore includes the results discussed in the
  introduction.
\end{subrem}

In fact, we prove much more. Our main result works in the category of log
canonical (lc) pairs.

\begin{thm}[Extension theorem for differential forms on lc pairs]\label{thm:main}
  Let $X$ be a complex quasi-projective variety of dimension $n$ and let $D$
  be a $\mathbb Q$-divisor on $X$ such that the pair $(X,D)$ is log canonical.
  Let $\pi: \wtilde X \to X$ be a \wlr with $\pi$-exceptional set $E$
  and
  $$
  \wtilde D := \text{largest reduced divisor contained in } \supp
  \pi^{-1}(\text{non-klt locus}),
  $$
  where the non-klt locus is the smallest closed subset $W \subset X$ such
  that $(X,D)$ is klt away from $W$. Then the sheaves $\pi_* \Omega^p_{\wtilde
    X}(\log \wtilde D)$ are reflexive, for all $p \leq n$.
\end{thm}

\begin{subrem}
  In Section~\ref{sec:examples} we gathered a number of examples to illustrate
  Theorem~\ref{thm:main} and to show that its statement is sharp.
\end{subrem}

\begin{subrem}\label{rem:whyextension}
  The name ``extension theorem'' is justified by Observation~\ref{obs:13},
  which asserts that the sheaf $\pi_* \Omega^p_{\wtilde X}(\log \wtilde D)$ is
  reflexive if and only if for any open set $U \subseteq X$ and any number
  $p$, the restriction morphism
  \begin{equation*}
    H^0\bigl(U,\, \pi_*\Omega^p_{\wtilde X}(\log \wtilde D)\bigr) \to
    H^0\bigl(U \setminus \pi(E),\, \Omega^p_{X}(\log \lfloor D\rfloor)\bigr)
  \end{equation*}
  is surjective. In other words, logarithmic $p$-forms defined on the
  non-singular part of $X$ can be extended to any resolution of singularities.
\end{subrem}

\begin{rem}
  A pair is log canonical if its sheaf of logarithmic $n$-forms satisfies
  certain conditions, closely related to extension properties. For such pairs,
  Theorem~\ref{thm:main} asserts that analogous extension properties hold for
  forms of arbitrary degrees. This matches the philosophy that the geometry of
  a variety is governed by the behaviour of its $n$-forms.
\end{rem}

\subsection{Previous results}

The extension problem has been studied in the literature, mostly asking
extension only for special values of $p$. For a variety $X$ with only isolated
singularities, reflexivity of $\pi_* \Omega^{p}_{\wtilde X}$ was shown by
Steenbrink and van Straten for $p \leq \dim X-2$ without any further
assumption on the nature of the singularities, \cite[Thm.~1.3]{SS85}.  Flenner
extended these results to normal varieties, subject to the condition that $p
\leq \codim X_{\sing} - 2$, \cite{Flenner88}. Namikawa proved reflexivity for
$p \in \{1, 2\}$, in case $X$ has canonical Gorenstein singularities,
\cite[Thm.~4]{Namikawa01}. In the case of finite quotient singularities
similar results were obtained in \cite{deJongStarr}. For a log canonical pair
with reduced boundary divisor, the cases $p \in \{1, \dim X-1, \dim X\}$ were
settled in \cite[Thm.~1.1]{GKK08}.

A related setup where the pair $(X,D)$ is snc, and where $\pi : \wtilde X \to
X$ is the composition of a finite Galois covering and a subsequent resolution
of singularities has been studied by Esnault and Viehweg. In \cite{RevI} they
obtain in their special setting a result similar to Theorem~\ref{thm:main} and
additionally prove vanishing of higher direct image sheaves.

\change{We would also like to mention the paper \cite{Bar78} where
  differential forms are discussed even in non-normal settings.}

\subsection{Applications}

In order to keep the length of this article reasonable, we only give a few
applications. These include the existence of a pull-back map for reflexive
differentials, rational connectivity of klt spaces, the
Lipman-Zariski-conjecture, and Bogomolov-Sommese type results. Many more
applications, e.g., to rational connectivity, Kodaira-Akizuki-Nakano vanishing
type results and varieties with trivial canonical classes, will be published
separately.

\subsection{Further results of this paper}

Apart from the extension results, we develop a theory of differential forms on
dlt pairs, showing that many of the fundamental theorems and techniques known
for sheaves of logarithmic differentials on smooth varieties also hold in the
dlt setting. In particular, there is a satisfactory theory of relative
differentials and a residue theory. A detailed introduction is given in
Section~\vref{sec:reflxDLTintro}.

We believe that these results are of independent interest. Sheaves of
reflexive differentials on singular spaces appear naturally when one uses
minimal model theory to study compactifications of moduli spaces, where
differentials can often be constructed using Hodge-theoretic methods,
cf.~\cite{VZ02, Viehweg06}. For a concrete example, we refer to \cite{KK08c}
where a study of reflexive differentials on dlt spaces was an important
ingredient in a generalisation of Shafarevich hyperbolicity.

\subsection{Outline of the paper}

The proof of our main theorem is given in two steps.  We first extend up to
logarithmic poles and then we prove the stronger extension result.  This is
done in Parts~\ref{part:4} and \ref{part:5}, respectively.

After a preliminary section, mainly devoted to setting up the basic notation,
we first give in Part~\ref{part:1} some applications of the Extension
Theorem~\ref{thm:main}.  Parts~\ref{part:2} and \ref{part:3} consist of
indispensable technical preparations which might, however, merit attention on
their own.  In particular, Part~\ref{part:2} presents a systematic treatment
of reflexive differential on dlt pairs. Part~\ref{part:3} presents two
vanishing theorems for direct image sheaves on log canonical pairs, one of
them generalising and expanding Steenbrink's vanishing theorem. A technical
vanishing theorem for cohomology with support is also included. In the
Appendix~\ref{app:A} and \ref{app:B}, we present several important facts that
are likely known to experts, but for which we were unable to find complete
references.

\subsection*{Acknowledgements}

The main ideas that led to this paper were perceived when all four authors visited
the MSRI special program in algebraic geometry in the spring of 2009. We would like
to thank the institute for support and for the excellent working conditions. The work
on this paper benefited from discussions with V.~Alexeev, C.~Birkar, H.~Esnault,
T.~de Fernex, G.M.~Greuel, Y.~Kawamata, J.~Kollár, J.~McKernan, M.~Reid,
O.~Riemenschneider, and W.~Soergel. \change{The authors want to thank the referee for very
  valuable remarks and suggestions.}

\section{Notation, conventions and standard facts}
\label{sec:notation}

The results of this paper are formulated and proven using the language of
higher dimensional algebraic geometry. While most of the definitions and much
of the notation we use is fairly standard in the field, we are aware of
several instances where definitions have evolved with time and are not always
coherently used in the literature. To minimise the potential for confusion, we
have chosen to prepend this paper with the present section that collects
standard facts and specifies notation wherever misunderstandings seem
likely. We quote standard references where possible.

\subsection{Base field, Kähler differentials}
\label{ssect:standard-assumptions}

Throughout the paper, we will work over the field of complex numbers.  For a
point on a scheme or complex analytic space, $p\in X$, the residue field of
$p$ will be denoted by $\kappa(p)$.

The central objects in this paper are differential forms on singular spaces.
Traditionally that means (logarithmic) Kähler differentials: If $X$ is a
scheme or complex space and $D$ a reduced Weil divisor on $X$ then we denote
the sheaves of Kähler differentials (resp.\ logarithmic Kähler differentials)
by $\Omega^1_X$ (resp.\ $\Omega^1_X(\log D)$). For a $p\in\bN$ we let
$\Omega^p_X = \bigwedge^p \Omega^1_X$ and $\Omega^p_X(\log D) = \bigwedge^p
\Omega^1_X(\log D)$. In particular, $\Omega^0_X=\Omega_X^0(\log D)=\sO_X$.

\begin{rem}
  The sheaves of Kähler differentials do not behave well near singular
  points. It is often more advantageous to work with their reflexive
  hulls. See Subsection~\ref{ssec:reflexive} for definitions and remarks
  regarding reflexive differential forms.
\end{rem}

\subsection{Pairs}

The main results of this paper concern pairs of algebraic varieties and
effective divisors, which have long been central objects in higher dimensional
algebraic geometry. In our discussion of pairs, we follow the language and
notational conventions of the book \cite{KM98}. We recall the most important
conventions for the reader's convenience.

\begin{defn}[Pairs and reduced pairs]\label{def:everythinglog}
  A \emph{pair} (or \emph{log variety}) $(X,D)$ consists of a normal
  quasi-projective variety $X$ and a \emph{boundary}, i.e., an effective
  $\mathbb Q$-Weil divisor $D=\sum d_iD_i$ on $X$ such that $D_i$ are reduced
  effective (integral) Weil-divisors and $d_i\in [0,1]\cap\bQ$.  A
  \emph{reduced pair} is a pair $(X,D)$ such that $D$ is reduced, that is, $D=
  \lfloor D \rfloor$, or equivalently all components of $D$ appear with
  coefficient $1$.
\end{defn}

\begin{notation}[Singularities of pairs]
  Given a pair $(X,D)$, we will use the notions lc (log canonical), klt, dlt
  without further explanation or comment and simply refer to
  \cite[Sect~2.3]{KM98} for a discussion and for their precise definitions.
\end{notation}

\begin{defn}[\protect{Snc pairs \cite[0.4(8)]{KM98}}]\label{def:everythinglog2a}
  Let $(X, D)$ be a pair, and $x \in X$ a point. We say that $(X, D)$ is
  \emph{snc at $x$} if there exists a Zariski-open neighbourhood $U$ of $x$
  such that $U$ is smooth and such that $\supp(D) \cap U$ is either empty, or
  a divisor with simple normal crossings.  The pair $(X, D)$ is called
  \emph{snc} if it is snc at every point of $X$.

  Given a pair $(X,D)$, let $(X,D)_{\reg}$ be the maximal open set of $X$
  where $(X,D)$ is snc, and let $(X,D)_{\sing}$ be its complement, with the
  induced reduced subscheme structure.
\end{defn}

\begin{rem}
  If $(X,D)$ is a pair, then by definition $X$ is normal. Furthermore, near a
  general point of $D$, both $X$ and $D$ are smooth. In particular,
  $\codim_X(X,D)_{\sing}\geq 2$.
\end{rem}

\begin{ex}
  In Definition~\ref{def:everythinglog2a}, it is important that we work in the
  Zariski topology.  If $X = \mathbb P^2$ and $D \subset X$ is a nodal cubic curve
  with singular point $x \in D$, then $(X,D)$ is \emph{not} snc. In particular,
  $(X,D)_{\reg} = X \setminus \{x\}$.
\end{ex}

While snc pairs are the logarithmic analogues of smooth spaces, snc morphisms, which
we discuss next, are the analogues of smooth maps. Although \emph{relatively snc
  divisors} have long been used in the literature, cf.~\cite[Sect.~3]{Deligne70}, we
are not aware of a good reference that discusses them in detail, so that we include a
full definition here.

\begin{notation}[Intersection of boundary components]\label{not:DI}
  Let $(X,D)$ be a pair, where the boundary divisor $D$ is written as a sum of its
  irreducible components $D = \alpha_1 D_1 + \ldots + \alpha_n D_n$.  If $I \subseteq
  \{1, \ldots, n\}$ is any non-empty subset, we consider the scheme-theoretic
  intersection $D_I := \cap_{i \in I} D_i$. If $I$ is empty, set $D_I := X$.
\end{notation}

\begin{rem}[Description of snc pairs]\label{rem:descrSNC}
  In the setup of Notation~\ref{not:DI}, it is clear that the pair $(X,D)$ is snc if
  and only if all $D_I$ are \change{smooth and of codimension equal to the number of defining
    equations: $\codim_XD_I=|I|$ for all $I$ where $D_I \not = \emptyset$.}
\end{rem}

\begin{defn}[\protect{Snc morphism, relatively snc divisor,
    \cite[Def.~2.1]{VZ02}}]\label{def:sncMorphism}
  If $(X,D)$ is an snc pair and $\phi: X \to T$ a surjective morphism to a smooth
  variety, we say that $D$ is \emph{relatively snc}, or that $\phi$ is \emph{an snc
    morphism of the pair} $(X,D)$ if \change{for any set $I$ with $D_I \not =
    \emptyset$} all restricted morphisms $\phi|_{D_I} : D_I \to
  T$ are smooth \change{of relative dimension $\dim X-\dim T -|I|$}.
\end{defn}

\begin{rem}[Fibers of an snc morphisms]\label{rem:fiberSNC}
  If $(X,D)$ is an snc pair and $\phi: X \to T$ is any surjective snc morphism
  of $(X,D)$, it is clear from Remark~\ref{rem:descrSNC} that if $t \in T$ is
  any point, with preimages $X_t := \phi^{-1}(t)$ and $D_t := D \cap X_t$ then
  the pair $(X_t, D_t)$ is again snc.
\end{rem}

\begin{rem}[All morphisms are generically snc]\label{rem:genSNC}
  If $(X,D)$ is an snc pair and $\phi: X \to T$ is any surjective morphism, it
  is clear from generic smoothness that there exists a dense open set $T^\circ
  \subseteq T$, such that $D \cap \phi^{-1}(T^\circ)$ is relatively snc over
  $T^\circ$.
\end{rem}

\subsection{\Slrs}

Resolutions of singularities have been in constant use in algebraic geometry
ever since Hironaka's seminal work \cite{Hir62}.  There are several
incompatible definitions of ``\wlrs'' used in the literature, all serving
different purposes.  In this paper, we use two variations of the resolution
theme, called ``\wlr'' and ``\slr'', respectively. We refer to
\cite[p.~3]{KM98} for further explanations concerning these notions.

\begin{defn}[\protect{{\Wlr} and \slr
    \cite[0.4(10)]{KM98}}]\label{def:everythinglog2b}
  A \emph{\wlr of} a pair $(X, D)$ is a surjective birational morphism $\pi: \wtilde
  X \to X$ such that
  \begin{enumerate-p}
    \setcounter{enumi}{\value{equation}}
  \item\ilabel{il:lrp1} the space $\wtilde X$ is smooth,
  \item\ilabel{il:lrp2} the $\pi$-exceptional set $\Exc(\pi)$ is of pure
    codimension one, and
  \item the set $\pi^{-1}(\supp D) \cup \Exc(\pi)$ is a divisor with simple
    normal crossings.
  \end{enumerate-p}
  A \wlr $\pi$ is called a \emph{\slr of $(X, D)$} if the following property
  holds in addition.
  \begin{enumerate-cont}
  \item\ilabel{il:lrp4} The rational map $\pi^{-1}$ is a well-defined
    isomorphism over the open set $(X, D)_{\reg}$.
  \end{enumerate-cont}

\end{defn}

\begin{fact}[Hironaka's theorem on resolutions, cf.~\cite{Kollar07}]
  \Wlrs and \slrs exist.
\end{fact}

\begin{rem}\label{rem:logresforsmalldiv}
  Let $(X, D)$ be a pair, and $\pi: \wtilde X \to X$ a \slr. If $D' \subseteq
  D$ is a subdivisor, it is not generally true that $\pi$ is also a
  \change{strong} log resolution of the pair $(X, D')$.
  \PreprintAndPublication{

    For an example, let $X = \mathbb P^2$ , let $D \subset \mathbb P^2$ be a
    cuspidal plane cubic, and $D' = \emptyset$. Let $\pi: \wtilde X \to X$ be
    a \slr of the pair $(X,D)$. Since $(X, D)$ is not snc, the morphism $\pi$
    is not isomorphic. On the other hand, since $(X, D')$ is snc, the
    property~\iref{il:lrp4} of Definition~\ref{def:everythinglog2b} asserts
    that any \slr of $(X, D')$ must in fact be isomorphic.}{}
\end{rem}

The following elementary lemma shows that the property~\iref{il:lrp4} is the only
property that possibly fails when one replaces $D$ by a smaller divisor.
\PreprintAndPublication{}{A complete proof is found in the extended version of this
  paper, \cite{GKKP10}.}

\begin{lem}\label{lem:logresforsmalldiv}
  Let $(X, D)$ be a pair, and $\pi: \wtilde X \to X$ a \wlr $(X, D)$.  If $D'
  \subseteq D$ is an effective sub-$\bQ$-divisor, then $\pi$ is a \wlr of $(X,
  D')$.\PreprintAndPublication{}{\qed}
\end{lem}
\PreprintAndPublication{
\begin{proof}
  Properties~\iref{il:lrp1} and \iref{il:lrp2} being clear, it remains to show
  that $\pi^{-1}(\supp D') \cup \Exc(\pi)$ is a divisor with simple normal
  crossings.  Since every subdivisor of an snc divisor is again an snc
  divisor, it suffices to show that the set $\pi^{-1}(\supp D') \cup
  \Exc(\pi)$ is of pure codimension one.  Accordingly, there is nothing to
  show if either $\supp D' = \supp D$, or if $\supp D' = \emptyset$. We may
  thus assume without loss of generality that $\supp D' \not = \emptyset$, and
  that $\supp D' \varsubsetneq \supp D$.

  We decompose the preimage of $\supp D'$ into a divisorial and a small part,
  $$
  \pi^{-1}(\supp D') = \wtilde D'_{\rm div} \cup \wtilde D'_{\rm small}
  $$
  where $\wtilde D'_{\rm div}$ has pure codimension one, and $\codim_{\wtilde
    X} \wtilde D'_{\rm small} > 1$. Since $\supp D'$ is of pure codimension
  one, it is clear that $\pi$ cannot be isomorphic at general points of
  $\wtilde D'_{\rm small}$, so that $\wtilde D'_{\rm small} \subseteq
  \Exc(\pi)$. It follows that
  \begin{equation}\label{eq:lnsfp}
    \pi^{-1}(\supp D') \cup \Exc(\pi) = \wtilde D_{\rm div} \cup
    \Exc(\pi).
  \end{equation}
  Equation~\eqref{eq:lnsfp} immediately shows that $\pi^{-1}(\supp D') \cup
  \Exc(\pi)$ has pure codimension $1$, as claimed. This completes the proof.
\end{proof}}{}

\PreprintAndPublication{}{\change{
  \subsection{Effective linear combinations of exceptional divisors}\label{app:A}

The following ``Negativity Lemma'' is well-known to experts. Variants are found in
the literature, for instance in \cite[3.39]{KM98}, \cite[Lem.~3.6.2]{BCHM06},
\cite[Lem.~5.23]{Hacon-Kovacs10}. A detailed proof is included in the expanded
version of this paper, \cite[Appendix~A]{GKKP10}.

\begin{lem}[\protect{Negativity Lemma for exceptional divisors}]\label{lem:excdiv}
  Let $\pi: \wtilde X \to X$ be a birational, projective and surjective
  morphism between irreducible and normal quasi-projective varieties.
  \begin{enumerate}
  \item\label{el:excdiv2} If $X$ is $\mathbb Q$-factorial, then there exists
    an effective and $\pi$-anti-ample Cartier divisor $D$ on $\wtilde X$ with
    $\supp(D) = E$. In particular, the $\pi$-exceptional set is of pure
    codimension one in $\wtilde X$.

  \item\label{el:excdiv1} If $D \subset \wtilde X$ is any non-trivial
    effective $\mathbb Q$-Cartier divisor with $\supp(D) \subseteq E$, then
    $D$ is not $\pi$-nef. \hfill \qed
  \end{enumerate}
\end{lem}
}}

\subsection{Reflexive sheaves and their tensor operations}
\label{ssec:reflexive}

The main theme of this paper being reflexive sheaves of differentials on
singular spaces, we constantly need to discuss sheaves that are not
necessarily locally free.  For this, we frequently use square brackets to
indicate taking the reflexive hull.

\begin{notation}[Reflexive tensor operations]\label{not:relfxive}
  Let $X$ be a normal variety, \change{$D$ a reduced Weil divisor}, and $\sA$ a
  coherent sheaf of $\O_X$-modules. For $n\in \bN$, set $\sA^{[n]} := (\sA^{\otimes
    n})^{**}$ and if $\pi: X' \to X$ is a morphism of normal varieties, set
  $\pi^{[*]}(\sA) := \bigl( \pi^*\sA \bigr)^{**}$.  In a similar vein, let
  $\Omega^{[p]}_X := \bigl( \Omega^p_X \bigr)^{**}$ and $\Omega^{[p]}_X(\log D) :=
  \bigl( \Omega^p_X (\log D) \bigr)^{**}$ For the definition of $\Omega^p_X$ and
  $\Omega^p_X(\log D)$ see \ref{ssect:standard-assumptions}.

  Observe that if $(X,D)$ is a pair and $\iota:U=(X,D)_{\reg}\into X$ is the
  embedding of the regular part of $(X,D)$ in to $X$, then
  $\Omega^{[p]}_X(\log D) \simeq \iota_*\bigl(\Omega^p_U(\log D\resto U)
  \bigr)$.
\end{notation}

\begin{notation}[Reflexive differential forms]
  A section in $\Omega^{[p]}_X$ or $\Omega^{[p]}_X(\log D)$ will be called a
  \emph{reflexive form} or a \emph{reflexive logarithmic form}, respectively.
\end{notation}

Generalising the vanishing theorem of Bogomolov-Sommese to singular spaces, we
need to discuss the Kodaira-Iitaka dimension of reflexive sheaves. Since this
is perhaps not quite standard, we recall the definition here.

\begin{defn}[Kodaira-Iitaka dimension of a sheaf]\label{def:KIdim}
  Let $X$ be a normal projective variety and $\sA$ a reflexive sheaf of rank
  one on $Z$.  If $h^0\bigl(X,\, \sA^{[n]}\bigr) = 0$ for all $n \in \mathbb
  N$, then we say that $\sA$ has Kodaira-Iitaka dimension $\kappa(\sA) :=
  -\infty$.  Otherwise, set
  $$
  M := \bigl\{ n\in \mathbb N \,|\, h^0\bigl(X,\, \sA^{[n]}\bigr)>0\bigr\},
  $$
  recall that the restriction of $\sA$ to the smooth locus of $X$ is locally
  free and consider the natural rational mapping
  $$
  \phi_n : X \dasharrow \mathbb P\bigl(H^0\bigl(X,\, \sA^{[n]}\bigr)^*\bigr)
  \quad \text{ for each } n \in M.
  $$
  The Kodaira-Iitaka dimension of $\sA$ is then defined as
  $$
  \kappa(\sA) := \max_{n \in M} \bigl(\dim \overline{\phi_n(X)}\bigr).
  $$
\end{defn}

\change{\begin{defn}\label{def:QCartier} Let $X$ be a normal algebraic variety.
    A reflexive sheaf $\sF$ of rank one is called \emph{$\mathbb{Q}$-Cartier} if
    there exists an $m\in \mathbb{N}^{>0}$ such that $\sF^{[m]}$ is locally free.
\end{defn}
\begin{rem}
  In the setup of Definition~\ref{def:QCartier}, there exists a reduced Weil divisor
  $D$ on $X$ such that $\sF = \sO_X(D)$, see for example \cite[Appendix to
  \S1]{Reid79}. Then, $\sF$ is $\mathbb{Q}$-Cartier if and only if there exists an
  $m\in \mathbb{N}^{>0}$ such that $\sO_X(mD)$ is locally free.
\end{rem}}

\subsection{Cutting down}

An important technical property of canonical, terminal, klt, dlt and lc
singularities is their stability under general hyperplane sections. This is
particularly useful in inductive proofs, as we will see, e.g., in
Section~\ref{sect:dltlocalstructure}. We gather the relevant facts here for
later reference.

\begin{notation}
  For a line bundle $\sL\in\Pic X$, the associated linear system of effective
  Cartier divisors will be denoted by $|\sL|$.
\end{notation}

\begin{lem}[Cutting down pairs I]\label{lem:cuttingDown}
  Let $(X, D)$ be a pair, \change{$\dim(X)\geq 2$}, and let $H \in |\sL|$ be a
  general element of an ample basepoint-free linear system corresponding to $\sL \in
  \Pic X$. Consider the cycle-theoretic intersection $D_H := D\cap H$.  Then the
  following holds.
  \begin{enumerate}
  \item\label{il:cdA1} The divisor $H$ is irreducible and normal.
  \item\label{il:cdA2} If $D = \sum a_i D_i$ is the decomposition of $D$ into
    irreducible components, then the intersections $D_i\cap H$ are distinct,
    irreducible and reduced divisors in $H$, and $D_H = \sum a_i (D_i\cap H)$.
  \item\label{il:cdA3} The tuple $(H, D_H)$ is a pair in the sense of
    Definition~\ref{def:everythinglog}, and rounding-down $D$ commutes with
    restriction to $H$, i.e., $\supp (\lfloor D_H \rfloor) = H\cap \supp
    (\lfloor D \rfloor)$.
  \item\label{il:cdB} If $H$ is smooth, then $X$ is smooth along $H$.
  \item\label{il:cdC} If $(H, D_H)$ is snc, then $(X,D)$ is snc along $H$.
  \end{enumerate}
\end{lem}
\begin{proof}
  Assertion (\ref{lem:cuttingDown}.\ref{il:cdA1}) is a known generalisation of
  Seidenberg's Theorem, see \cite[Thm.~1.7.1]{BS95} and \cite[Thm.~1]{Seidenberg50}.
  Assertion~(\ref{lem:cuttingDown}.\ref{il:cdA2}) is a well-known consequence of
  Bertini's theorem, (\ref{lem:cuttingDown}.\ref{il:cdA3}) follows from
  (\ref{lem:cuttingDown}.\ref{il:cdA1}) and (\ref{lem:cuttingDown}.\ref{il:cdA2}).
  Statements~(\ref{lem:cuttingDown}.\ref{il:cdB})--(\ref{lem:cuttingDown}.\ref{il:cdC})
  are consequences of the fact that a space is smooth along a Cartier divisor if the
  divisor itself is smooth.
\end{proof}

\begin{lem}[Cutting down \slrs]\label{lem:cuttingDownRes}
  Let $(X,D)$ be a pair, \change{$\dim X \geq 2$}, and let $\pi : \wtilde X \to
  X$ a \slr (resp.~a \wlr). Let $H \in |\sL|$ be a general element of an ample
  basepoint-free linear system on $X$ corresponding to $\sL \in \Pic X$. Set $\wtilde
  H := \pi^{-1}(H)$. Then the restricted morphism $\pi|_{\wtilde H}: \wtilde H \to H$
  is a \slr (resp.~a \wlr) of the pair $(H, D\cap H)$, with exceptional set
  $\Exc(\pi|_{\wtilde H}) = \Exc(\pi) \cap \wtilde H$.
\end{lem}
\PreprintAndPublication{\begin{proof}
  First consider the case when $\pi$ is a \wlr.  Zariski's Main Theorem
  \cite[V~Thm.~5.2]{Ha77} implies that since $X$ is normal, a point $\wtilde x
  \in \wtilde X$ is contained in the $\pi$-exceptional set $\Exc(\pi)$ if and
  only if the fibre through $\wtilde X$ is positive dimensional. Since $H$ is
  normal by (\ref{lem:cuttingDown}.\ref{il:cdA1}), the same holds for the
  restriction $\pi|_{\wtilde H}$; for all points $\wtilde x \in \wtilde H$, we
  have $\wtilde x \in \Exc(\pi|_{\wtilde H})$ if and only if the $\pi$-fibre
  through $\wtilde x$ is positive dimensional. It follows that
  \begin{align}
    \label{eq:Eres1} \Exc(\pi|_{\wtilde H}) & = \wtilde H \cap \Exc(\pi) & \text{and}\\
    \label{eq:Eres2} (\pi|_{\wtilde H})^{-1}(D \cap H) \cup \Exc(\pi|_{\wtilde
      H}) & = \wtilde H \cap \bigl( \underbrace{\pi^{-1}(D) \cup
      \Exc(\pi)}_{\text{snc divisor by assumption}} \bigr).
  \end{align}
  Since $\pi$ has connected fibres the linear systems $|\sL|$ and $|\pi^*\sL|$ can be
  canonically identified. In particular, $\wtilde H$ is a general element of a
  basepoint-free linear system, and it follows immediately from Bertini's Theorem
  that $\wtilde H$ is smooth. The equality in \eqref{eq:Eres1} shows that
  $\Exc(\pi|_{\wtilde H})$ is of pure codimension one in $\wtilde H$.  The equality
  in \eqref{eq:Eres2} and Bertini's Theorem then give that the set $(\pi|_{\wtilde
    H})^{-1}(D \cap H) \cup \Exc(\pi|_{\wtilde H})$ is a divisor with simple normal
  crossings. It follows that the restricted map $\pi|_{\wtilde H}$ is a \wlr of the
  pair $(H, D \cap H)$.

  Now assume that $\pi$ is a \slr of $(X,D)$. We aim to show that then
  $\pi|_{\wtilde H}$ is a \slr of the pair $(H, D \cap H)$. To this end, let
  $x \in H$ be any point where the pair $(H, D \cap H)$ is snc. By
  (\ref{lem:cuttingDown}.\ref{il:cdC}) the pair $(X,D)$ is then snc in a
  neighbourhood of $x$, and the \slr $\pi$ is isomorphic near $x$. The
  equality in \eqref{eq:Eres1} then shows that the restriction $\pi|_{\wtilde
    H}$ is likewise isomorphic near $x$ showing that $\pi|_{\wtilde H}$ is a
  \slr indeed.
\end{proof}}{\change{\noindent A proof of Lemma~\ref{lem:cuttingDownRes} can be found in the
  preprint version \cite{GKKP10} of this paper.}}

\begin{lem}[Cutting down pairs II]\label{lem:cuttingDown2}
  Let $(X, D)$ be a pair and let $H \in |\sL|$ be a general element of an ample
  basepoint-free linear system corresponding to $\sL \in \Pic X$. Consider the
  cycle-theoretic intersection $D_H := D\cap H$. If $(X, D)$ is dlt (resp.~canonical,
  klt, lc), then $(H, D_H)$ is dlt (resp.~canonical, klt, lc) as well.
\end{lem}
\begin{proof}
  To prove Lemma~\ref{lem:cuttingDown2} for dlt pairs, recall Szabó's
  characterisation of ``dlt'' \cite{Szabo95}, \cite[Thm.~2.44]{KM98} which asserts
  that a pair is dlt if and only if there exists a log resolution $\pi: \wtilde X \to
  X$ where all exceptional divisors have discrepancy greater than $-1$.  Choose one
  such resolution and set $\wtilde H := \pi^{-1}(H)$.  Lemma~\ref{lem:cuttingDownRes}
  then asserts that $\pi_{\wtilde H}: \wtilde H \to H$ is a \slr of the pair $(H,
  D_H)$, and it follows from the adjunction formula that the discrepancy of any
  $\pi_{\wtilde H}$-exceptional divisor is likewise greater than $-1$. A second
  application of the characterisation of dlt pairs then yields the claim in case
  $(X,D)$ is dlt.

  For canonical, klt, or lc pairs, Lemma~\ref{lem:cuttingDown2} follows from a
  computation of discrepancies, \cite[Lem.~5.17]{KM98}.
\end{proof}

\subsection{Projection to subvarieties}

Let $X$ be a normal variety such that $X_{\sing}$ is irreducible and of
dimension $1$.  One may study the singularities of $X$ near general points of
$X_{\sing}$ by looking at a family of sufficiently general hyperplane sections
$(H_t)_{t \in T}$, and by studying the singularities of the hyperplanes
$H_t$. Near the general point of $X_{\sing}$ the $H_t$ define a morphism, and
it is often notationally convenient to discuss the family $(H_t)_{t \in T}$ as
being fibres of that morphism.

This idea is not new. We include the following proposition to fix notation,
and to specify a precise framework for later use.

\begin{prop}[Projection to a subvariety]\label{prop:projection}
  Let $X$ be quasi-projective variety and $T \subseteq X$ an irreducible
  subvariety. Then there exists a Zariski-open subset $X^\circ \subseteq X$
  such that $T^\circ := T \cap X^\circ$ is not empty, and such that there
  exists \change{a diagram}
  $$
  \xymatrix{ %
    Z^\circ \ar[rr]^(.45){\gamma}_(.45){\text{finite, étale}} \ar[d]_{\phi} &&
    X^\circ \\
    S^\circ }
  $$
  with the property that the restriction of $\phi$ to any connected component
  of $\,\wtilde T^\circ := \gamma^{-1}(T^\circ)$ is an isomorphism.
\end{prop}
\begin{proof}
  Let $X^\circ_0 \subseteq X$ be an affine open set that intersects $T$
  non-trivially.  An application of the Noether normalisation theorem,
  \cite[I.~Thm.~10]{Shaf94}, to the affine variety $T^\circ_0 := T\cap
  X^\circ_0 \subseteq X^\circ_0$ yields a projection to an affine space,
  $\phi_0: X^\circ_0 \to S^\circ_0$, whose restriction to $T^\circ_0$ is
  generically finite. Shrinking $X^\circ_0$ and $S^\circ_0$ further, if
  necessary, we may assume that the restriction $\phi_0|_{T^\circ_0}$ is
  finite and étale, say $n$-to-$1$. Next, we will construct a commutative
  diagram of morphisms,
  \begin{equation}\label{eq:diagX}
    \xymatrix{
      X^\circ_d \ar[d]_{\phi_d} \ar[rr]^{\gamma_d}_{\text{étale}} &&
      \cdots \ar[rr]^{\gamma_2}_{\text{étale}} &&
      X^\circ_1 \ar[d]^{\phi_1} \ar[rr]^{\gamma_1}_{\text{étale}} &&
      X^\circ_0 \ar[d]^{\phi_0} \\
      S^\circ_d \ar[rr]_{\text{étale}} &&
      \cdots \ar[rr]_{\text{étale}} &&
      S^\circ_1 \ar[rr]_{\text{étale}} &&
      S^\circ_0 }
  \end{equation}
  such that
  \begin{enumerate-cont}
  \item for any index $k$, the restriction of $\phi_k$ to $T^\circ_k :=
    (\gamma_1 \circ \cdots \gamma_k)^{-1}(T^\circ_0)$ is étale, and

  \item the restriction of $\phi_d$ to any component of $T^\circ_d$ is
    isomorphic.
  \end{enumerate-cont}
  Once the diagram is constructed, the proof is finished by setting $Z^\circ
  := X^\circ_d$, $S^\circ := S^\circ_d$ and $\phi := \phi_d$.

  To construct a diagram as in~\eqref{eq:diagX}, we proceed inductively as
  follows.  Assume $\phi_k : X^\circ_k \to S^\circ_k$ have already been
  constructed. If the restriction of $\phi_k$ to any component of $T^\circ_k$
  is an isomorphism then we stop.  Otherwise, let $S^\circ_{k+1} \subseteq
  T^\circ_k$ be any component where $\phi_k|_{S^\circ_{k+1}}$ is not
  isomorphic, and set $X^\circ_{k+1} := X^\circ_k \times_{S^\circ_k}
  S^\circ_{k+1}$. Since étale morphisms are stable under base change,
  \cite[I~Prop.~4.6]{SGA1}, it follows that the projection $\gamma_{k+1} :
  X^\circ_{k+1} \to X^\circ_k$ and the restriction
  $\phi_{k+1}|_{T^\circ_{k+1}}$ are both étale.

  We need to show that the inductive process terminates. For that, observe
  that all restrictions $\phi_k|_{T^\circ_k} : T^\circ_k \to S^\circ_k$ are
  finite, étale and $n$-to-$1$. Additionally, it follows inductively from the
  fibre product construction that the restriction $\phi_k|_{T^\circ_k}$ admits
  at least $k$ sections. It is then immediate that the process terminates
  after no more than $n$ steps.
\end{proof}

\begin{ex}\label{ex:projection}
  To illustrate how projections to subvarieties will be used, consider a dlt
  pair $(X,D)$ whose singular locus $T := (X,D)_{\sing}$ is irreducible and of
  codimension $\codim_X T = 2$. We are often interested in showing properties
  of the pair $(X,D)$ that can be checked on the étale cover $Z^\circ$
  constructed in \eqref{prop:projection}. 
  Examples for such properties include the following.
  \begin{enumerate}
  \item\label{il:balthasar} The space $X$ is analytically $\mathbb
    Q$-factorial away from a set of codimension 3.
  \item\label{il:caspar} Near the general point of $T$, the space $X$ has only
    quotient singularities.
  \item\label{il:melchior} For any \slr $\pi: \wtilde X \to X$, the sheaf
    $\pi_* \Omega^p_{\wtilde X}$ is reflexive at the general point of $T$.
  \end{enumerate}
  Setting $\Delta^\circ := \gamma^{*}(D)$ and considering general fibres
  $$
  Z^\circ_t := \phi^{-1}(t) \quad\text{and}\quad \Delta^\circ_t :=
  \Delta^\circ \cap Z^\circ_t,
  $$
  it follows from the Cutting-Down Lemma~\ref{lem:cuttingDown} that the fibre
  pairs $(Z^\circ_t, \Delta^\circ_t)$ are dlt surfaces, where the property in
  question may often be checked easily. Once it is known that the fibres of
  $\phi$ have the desired property, it is often possible to prove that the
  property also holds for the total space $(Z^\circ, \Delta^\circ)$ of the
  family, and hence for $(X, D)$.
\end{ex}

\section{Examples}
\label{sec:examples}

In this section we discuss a number of examples that show to what extent the
main result of this paper, the Extension Theorem~\ref{thm:main}, is optimal

\subsection{Non-log canonical singularities}

The next example shows that log canonicity of $(X,D)$ is necessary to obtain any
extension result allowing no worse than log poles along the exceptional divisor. This
example is discussed in greater detail in \cite[Ex.~6.3]{GKK08}.

\begin{ex}\label{ex:non-lc}
  Let $X$ be the affine cone over a smooth curve $C$ of degree $4$ in
  $\P^2$. Observe that $X$ is a normal hyperplane singularity. In particular,
  $X$ is Gorenstein.  Let $\wtilde X$ be the total space of the line bundle
  $\sO_C(-1)$. Then, the contraction of the zero section $E$ of $\wtilde X$
  yields a \slr $\pi: \wtilde X \to X$. An elementary computation shows that
  the discrepancy of $E$ with respect to $X$ is equal to $-2$
  cf.~\cite[p.~351, Ex.~(1)]{Reid87}. Hence, $X$ has worse than log canonical
  singularities. If $\tau$ is a local generator of the locally free sheaf
  $\Omega^{[2]}_X$ near the vertex $P \in X$, the discrepancy computation
  implies that $\tau$ acquires poles of order $2$ when pulled back to
  $\widetilde X$.  By abusing notation we denote the rational form obtained on
  $\wtilde X$ by $\pi^*\tau$.

  Next, let $\xi$ be the vector field induced by the natural $\C^*$-action on
  $\wtilde X$ coming from the cone structure. By contracting $\pi^* \tau$ by
  $\xi$ we obtain a regular $1$-form on $\wtilde X \setminus E$ that does not
  extend to an element of $H^0 \bigl(\wtilde X,\, \Omega^1_{\wtilde X}(\log E)
  \bigr)$.
\end{ex}

Hence, in the non-log canonical case there is in general no extension result
for differential forms, not even for special values of $p$.

\subsection{Non-klt locus and discrepancies}

It follows from the definition of \emph{discrepancy} that for a given
reflexive logarithmic $n$-form $\sigma$ on a reduced pair $(X, D)$ of
dimension $n$ with log canonical singularities, the pull-back $\pi^*\sigma$
acquires additional poles only along those exceptional divisors $E_i$ with
discrepancy $a_i = -1$, see \cite[Sect.~5]{GKK08}. It hence extends without
poles even over those divisors $E_i$ with discrepancy $a_i > -1$ that map to
the non-klt locus of $(X, D)$. In the setup of Theorem~\ref{thm:main}, it is
therefore natural to ask whether it is necessary to include the full-preimage
of the non-klt locus in $\widetilde D$ in order to obtain an extension result
or if it suffices to include the non-klt places, that is, those divisor with
discrepancy $-1$. The next example shows that this does not work in general
for extending $p$-forms, when $p<n$.

\begin{ex}\label{ex:non-klt}
  Let $X =\{uw-v^ 2 \} \subset \mathbb{C}^3_{u,v,w}$ be the quadric cone, and
  let $D = \{v =0\}\cap X$ be the union of two rays through the vertex. The
  pair $(X, D)$ is log canonical.  Let $\widetilde X \subset
  \mathrm{Bl}_{(0,0,0)}(\C^3) \subset \mathbb{C}^3_{u,v,w} \times
  \mathbb{P}^2_{[y_1:y_2:y_3]}$ be the strict transform of $X$ in the blow-up
  of $\mathbb{C}^ 3$ at $(0,0,0)$ and $\pi: \widetilde X \to X$ the
  corresponding resolution. The intersection $U$ of $\widetilde X$ with
  $\{y_1\neq 0 \}$ is isomorphic to $\mathbb{C}^2$ and choosing coordinates
  $x,z$ on this $\mathbb{C}^2$, the blow-up is given by $\varphi: (x, z)
  \mapsto (z, xz, x^2z)$.  In these coordinates the exceptional divisor $E$ is
  defined by the equation $\{z=0\}$. The form $d\log v:= \frac{1}{v}dv $
  defines an element in $H^0\bigl(X,\ \Omega^{[1]}_X (\log D) \bigr)$. Pulling
  back we obtain
  $$
  \varphi^*(d\log v) = d\log x + d\log z.
  $$
  which has log-poles along the exceptional divisor. If $f: \widetilde X' \to
  \widetilde X$ is the blow up at a point $p \in E \setminus \pi_*^{-1}(D)$, we
  obtain a further resolution $\pi' = \pi \circ f$ of $X$. This resolution has an
  additional exceptional divisor $E' \subset \widetilde X '$ with discrepancy $0$.
  Note however that the pull-back of $d\log v$ via $\pi'$ has logarithmic poles along
  $E'$. To be explicit we compute on $f^{-1}(U)$: we have
  $$
  f^*\varphi^*(d\log v) = d\log (f^*x) + d\log (f^*z),
  $$
  and we note that $f^*z$ vanishes along $E'$ since we have blown up a point
  in $E = \{z = 0\}$.
\end{ex}

\subsection{Other tensor powers}

The statement of Theorem~\ref{thm:main} does not hold for arbitrary reflexive
tensor powers of $\Omega^1_X$. We refer to \cite[Ex.~3.1.3]{GKK08} for an
example where the analogue of the Extension Theorem~\ref{thm:main} fails for
$\Sym^{[2]} \Omega^1_X$, even when $X$ is canonical.

\part{APPLICATIONS OF THE EXTENSION THEOREM}
\label{part:1}

\section{Pull-back morphisms for reflexive differentials}

Kähler differentials are characterised by a number of universal properties,
one of the most important being the existence of a pull-back map: if $\gamma :
Z \to X$ is any morphism of algebraic varieties and if $p\in\bN$, then there
exists a canonically defined sheaf morphism
\begin{sequation}\label{eq:pbKd}
  d\gamma : \gamma^* \Omega^p_X  \to \Omega^p_Z.
\end{sequation}
The following example illustrates that for sheaves of reflexive differentials on
normal spaces, a pull-back map does not exist in general.

\begin{ex}[Pull-back morphism for dualising sheaves]
  Let $X$ be a \change{normal Gorenstein variety} of dimension $n$, and let $\gamma : Z \to
  X$ be any resolution of singularities.  Observing that the sheaf of reflexive
  $n$-forms is precisely the dualising sheaf, $\Omega^{[n]}_X \simeq \omega_X$, it
  follows directly from the definition of canonical singularities that $X$ has
  canonical singularities if and only if a pull-back morphism $d\gamma : \gamma^*
  \Omega^{[n]}_X \to \Omega^{n}_Z$ exists.
\end{ex}

An important consequence of the Extension Theorem~\ref{thm:main} is \change{the}
existence of a pull-back map for reflexive differentials of arbitrary degree,
whenever $\gamma: Z \to X$ is a morphism where the target is klt. The
pull-back map exists also in the logarithmic setup and ---in a slightly
generalised form--- in cases where the target is only lc.

\begin{thm}[Pull-back map for reflexive differentials on lc pairs]\label{thm:generalpullback}
  Let $(X, D)$ be an lc pair, and let $\gamma : Z \to X$ be a morphism from a
  normal variety $Z$ such that the image of $Z$ is not contained in the
  reduced boundary or in the singular locus, i.e.,
  $$
  \gamma(Z) \not\subseteq (X,D)_{\sing} \cup \supp \lfloor D \rfloor.
  $$
  If $1 \leq p \leq \change{\dim X}$ is any index and
  $$
  \Delta := \text{largest reduced Weil divisor contained in }
  \gamma^{-1}\bigl(\text{non-klt locus}\bigr),
  $$
  then there exists a sheaf morphism,
  $$
  d\gamma : \gamma^* \Omega^{[p]}_X(\log \lfloor D \rfloor) \to
  \Omega^{[p]}_Z(\log \Delta),
  $$
  that agrees with the usual pull-back morphism \eqref{eq:pbKd} of Kähler
  differentials at all points $p \in Z$ where $\gamma(p) \not \in
  (X,D)_{\sing} \cup \supp \lfloor D \rfloor$.
\end{thm}

Before proving Theorem~\ref{thm:generalpullback} below, we illustrate the
statement with one example and add a remark concerning possible
generalisations.

\begin{ex}[Restriction as a special case of
  Theorem~\ref{thm:generalpullback}]\label{rem:kltpullback}
  \PreprintAndPublication{
  \begin{figure}
      \centering

      \begin{picture}(8,6)(0,0)
        \put( 0.0, 0.0){\includegraphics[height=6cm]{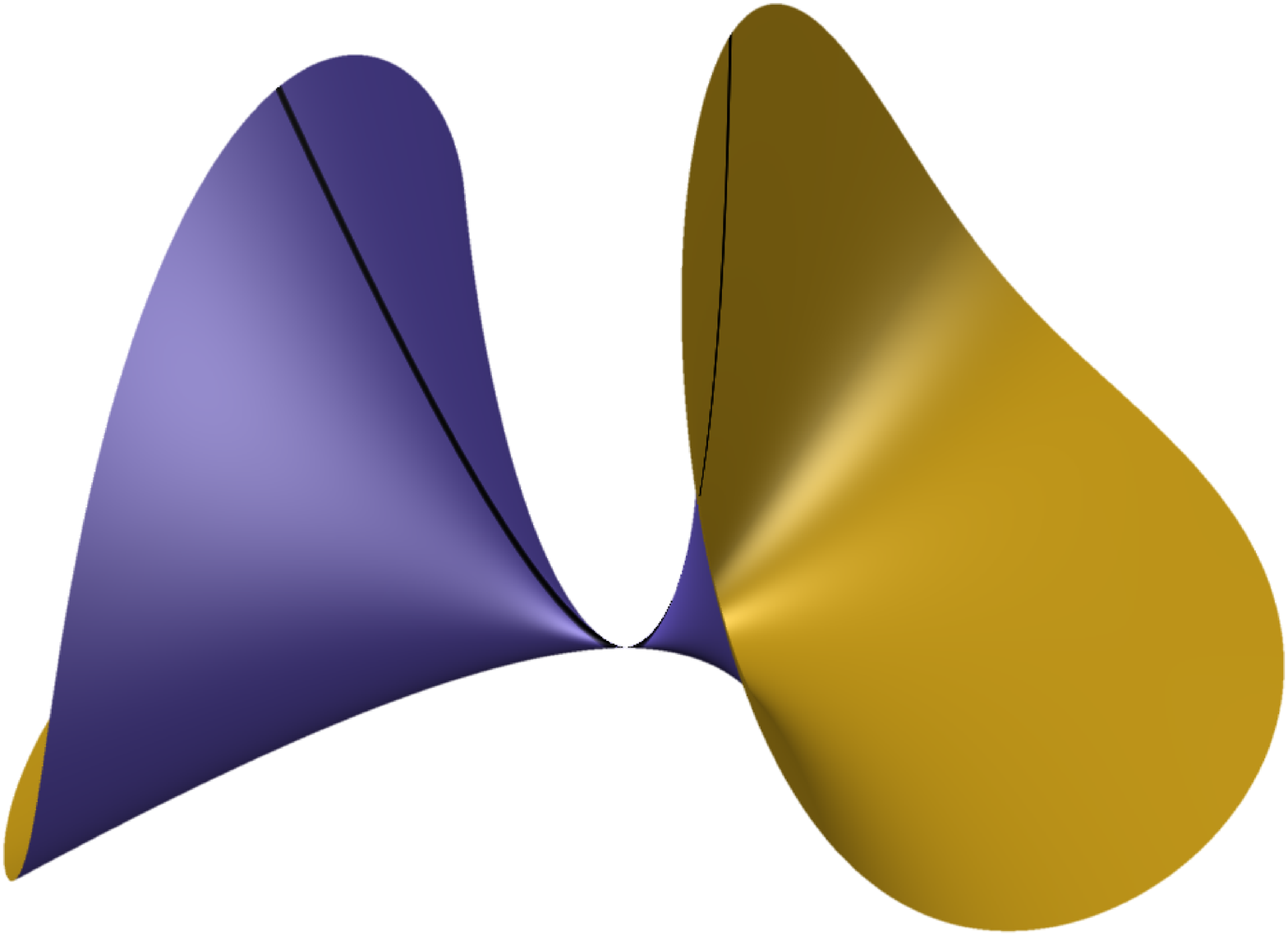}}
        \put(-1.5, 5.0){smooth curve $Z$}
        \put( 0.8, 5.0){\vector(2, -1){1.3}}
        \put( 6.5, 4.7){$X$ (singular surface)}
      \end{picture}

      \caption{A special case of the Pull-back theorem~\ref*{thm:generalpullback}.}
      \label{fig:sct}
    \end{figure}
  }{} %
  For a special case of Theorem~\ref{thm:generalpullback}, consider the case
  \PreprintAndPublication{sketched in Figure~\vref{fig:sct},}{} where $(X,
  \emptyset)$ is klt and $Z \subset X$ is a smooth subvariety that intersects
  $X_{\reg}$ non-trivially with inclusion map $\gamma: Z \to X$. Under these
  assumptions, Theorem~\ref{thm:generalpullback} asserts that any reflexive
  differential form $\sigma \in H^0\bigl(X, \Omega_X^{[p]} \bigr)$ restricts
  to a regular form on $Z$.
\end{ex}

\begin{rem}[Pull-back map when the image is contained in the boundary]
  In the setup of Theorem~\ref{thm:generalpullback}, if we assume additionally
  that the pair $(X, D)$ is dlt, then one may use the residue sequence
  (\ref{thm:relativereflexiveresidue}.\ref{il:RelReflResidue}) of
  Theorem~\ref{thm:relativereflexiveresidue} to define a pull-back map even in
  a setting where the image of $\gamma$ is contained in the boundary $\lfloor
  D \rfloor$. Details will be published in a forthcoming paper.
\end{rem}

The proof of Theorem~\ref{thm:generalpullback} uses the following notation.

\begin{notation}
  Let $(X,D)$ and $(Z, \Delta)$ be two pairs, and $\gamma: Z \to X$ a morphism
  such that $\gamma(Z) \not\subseteq (X,D)_{\sing} \cup \supp \lfloor D
  \rfloor$. If $\sigma$ is a rational section in $\Omega_X^{[p]}(\log \lfloor
  D \rfloor)$, then one may use the standard pull-back map for Kähler
  differentials to pull $\sigma$ back to a rational section of
  $\Omega_Z^{[p]}(\log \lfloor \Delta \rfloor)$, which we denote by
  $\gamma^*(\sigma)$.
\end{notation}

\begin{proof}[Proof of Theorem~\ref{thm:generalpullback}]\label{pf:generalpullback}
  Notice that to prove Theorem~\ref{thm:generalpullback}, it suffices to show
  that for every open subset $V \subseteq X$ the following holds:
  \begin{equation}\label{eq:nnasw}
    \gamma^*(\sigma) \in H^0\bigl( \gamma^{-1}(V),\, \Omega_Z^{[p]}(\log \Delta) \bigr)
    \text{ for all}\,\, \sigma \in H^0\bigl(V,\, \Omega_X^{[p]}(\log \lfloor D
    \rfloor) \bigr).
  \end{equation}
  Indeed, for every point $p \in Z$ and every germ $s \in \bigl(\gamma^*
  \Omega_X^{[p]}(\log \lfloor D \rfloor )\bigr)_p$ there exists an open
  neighbourhood $U$ of $p$ in $Z$, an open neighbourhood $V$ of $\gamma(p)$ in
  $X$ such that $\gamma(U) \subseteq V$, and such that $s$ is represented by a
  sum $\sum g_j\cdot \gamma^*\sigma_j$, where $g_j \in \sO_Z(U)$ and $\sigma_j
  \in H^0\bigl(V, \Omega_X^{[p]}(\log \lfloor D \rfloor)\bigr)$.

  To prove~\eqref{eq:nnasw}, let $\sigma \in H^0\bigl(V,\, \Omega_X^{[p]}(\log
  \lfloor D \rfloor) \bigr)$ be any reflexive form. To simplify notation, we
  may assume without loss of generality that $V = X$ and
  $\gamma^{-1}(V)=Z$. Let $\pi: \widetilde X \to X$ be any \change{strong resolution} of the pair
  $(X, D)$ and consider the following commutative diagram of varieties,
  $$
  \begin{xymatrix}{ %
      \wtilde Z \ar@/_3mm/[drrrr]_{\wtilde \pi \text{, birational \quad}}
      \ar@/^5mm/[rrrrrr]^{g} \ar[rrrr]_{\psi \text{, \wlr}} &&&& Y
      \ar[d]^{\pi_Z \text{, birational}} \ar[rr]_{\pi_{\wtilde X}} &&
      \wtilde X \ar[d]^{\pi \text{, \wlr of }(X, D)}\\
      &&&& Z \ar[rr]_{\gamma} && X,}
  \end{xymatrix}
  $$
  where $Y$ is the normalisation of the unique component of $Z \times_X
  \wtilde X$ that dominates $Z$, and where $\psi$ is a \wlr of the pair
  $\bigl( Y, (\pi \circ \pi_{\wtilde X})^*D \bigr)$. Furthermore, set
  \begin{align*}
    \wtilde D & := \text{largest reduced divisor in } \supp\,
    \pi^{-1} \bigl(\text{non-klt locus of }(X,D)\bigr), \\
    \wtilde \Delta & := \text{largest reduced divisor in } \supp\, (\pi \circ
    g)^{-1}\bigl(\text{non-klt locus of }(X,D)\bigr).
  \end{align*}
  By definition, we immediately obtain two relations\footnote{Note that the
    inclusion in \eqref{eq:inc1} might be strict. This can happen when
    $\pi^{-1} \bigl(\text{non-klt locus of }(X,D)\bigr)$ contains components
    of high codimension whose preimages under $g$ become divisors.} involving
  cycle-theoretic pull-back and push-forward,
  \begin{align}
    \label{eq:inc1} \supp g^*\wtilde D & \subseteq \supp \wtilde \Delta,\\
    \label{eq:inc2} \supp \wtilde \pi_*\wtilde \Delta & = \supp \Delta.
  \end{align}
  It is then clear from \eqref{eq:inc2} that~\eqref{eq:nnasw} holds once we
  show that
  \begin{equation}\label{eq:nnasw1}
    \wtilde \pi^* \bigl( \gamma^*(\sigma) \bigr) =
    g^* \bigl( \pi^*(\sigma) \bigr) \in H^0\bigl( \wtilde Z,\
    \Omega_{\wtilde Z}^{[p]}(\log \wtilde \Delta) \bigr).
  \end{equation}
  The Extension Theorem~\ref{thm:main} states that the pull-back
  $\pi^{*}(\sigma)$ is a regular logarithmic form in $H^0\bigl(\widetilde X,
  \Omega_{\widetilde X}^p(\log \widetilde D) \bigr)$, for all reflexive forms
  $\sigma$. Using \eqref{eq:inc1} and the standard pull-back map for
  logarithmic forms on snc pairs to pull back $\pi^{*}\sigma$ via the map $g$,
  the desired inclusion in \eqref{eq:nnasw1} follows.  This completes the
  proof.
\end{proof}

\section{Reflexive differentials on rationally chain connected spaces}

Rationally chain connected manifolds are rationally connected, and do not
carry differential forms. Building on work of Hacon and McKernan,
\cite{HMcK07}, we show that the same holds for reflexive forms on klt pairs.

\begin{thm}[Reflexive differentials on rationally chain connected spaces]\label{thm:kltRCC}
  Let $(X,D)$ be a klt pair. If $X$ is rationally chain connected, then $X$ is
  rationally connected, and $H^0 \bigl( X,\, \Omega^{[p]}_X \bigr) = 0 $ for
  all $p\in\bN, 1 \leq p \leq \dim X$.
\end{thm}
\begin{proof}
  Choose a \slr $\pi: \wtilde X \to X$ of the pair $(X,D)$. Since klt pairs
  are also dlt, a theorem of Hacon-McKernan, \cite[Cor.~1.5(2)]{HMcK07},
  applies to show that $X$ and $\wtilde X$ are both rationally connected. In
  particular, it follows that $H^0 \bigl( \wtilde X,\, \Omega^p_{\wtilde X}) =
  0$ for all $p > 0$ by \cite[IV.~Cor.~3.8]{K96}.

  Since $(X,D)$ is klt, Theorem~\ref{thm:generalpullback} asserts that there
  exists a pull-back morphism $d\pi : \pi^* \Omega^{[p]}_X \to
  \Omega^p_{\wtilde X}$.  As $\pi$ is birational, $d\pi$ is generically
  injective and since $\Omega^{[p]}_X$ is torsion-free, this means that the
  induced morphism on the level of sections is injective:
  $$
  \pi^*: H^0 \bigl( X,\, \Omega^{[p]}_X \bigr) \to H^0 \bigl( \wtilde X,\,
  \Omega^p_{\wtilde X} \bigr) = 0.
  $$
  The claim then follows.
\end{proof}

In this section, Theorem~\ref{thm:kltRCC} is presented as a consequence of the
Extension Theorem~\ref{thm:main}. As a matter of fact, the proof of the
Extension Theorem~\ref{thm:main}, which we give in Part~\ref{part:5} of the
paper, involves a proof of Theorem~\ref{thm:kltRCC} as part of the induction
process. This explains why the statement of Theorem~\ref{thm:kltRCC} appears
essentially unchanged as Proposition~\ref{prop:kltRCC-n} in Part~\ref{part:5},
where the Extension Theorem~\ref{thm:main} is proven.

In order to avoid confusion about the logic of this paper, we have chosen to
present an independent statement and an independent proof here.

\section{The Lipman-Zariski Conjecture for klt spaces}

The Lipman-Zariski Conjecture asserts that a variety $X$ with a locally free tangent
sheaf $\sT_X$ is necessarily smooth, \cite{MR0186672}.  The conjecture has been shown
in special cases; for hypersurfaces or homogeneous complete intersections
\cite{MR0360585, MR0306172}, for isolated singularities in higher-dimensional
varieties \cite[Sect.~1.6]{SS85}, and more generally, for varieties whose singular
locus has codimension at least $3$ \cite{Flenner88}. In this section we use the
Extension Theorem~\ref{thm:main} to prove the Lipman-Zariski Conjecture for klt
spaces. Notice that klt spaces in general have singularities in codimension $2$. The
proof follows an argument that goes back at least as far as \cite{SS85}. \change{It uses
  the notion of \emph{logarithmic tangent sheaf}, which we quickly recall: if $Z$ is
  a smooth algebraic variety and $\Delta$ is an snc divisor on $Z$, then the
  logarithmic tangent sheaf $\sT_Z(-\log \Delta)$ is defined to be the dual of
  $\Omega_Z^1(\log \Delta)$. A local computation shows that $\sT_Z(-\log \Delta)$ can
  be identified with the subsheaf of $\sT_Z$ containing those vector fields that are
  tangent to $\Delta$ at smooth points of $\Delta$.}

\begin{thm}[Lipman-Zariski Conjecture for klt spaces]
  Let $X$ be a klt space such that the tangent sheaf $\sT_X$ is locally
  free. Then $X$ is smooth.
\end{thm}
\begin{proof}
  We argue by contradiction and assume that $X$ is not smooth. Recall that
  there exists a uniquely defined \slr $\pi : \widetilde X \to X$ of the pair
  $(X,\emptyset)$, called the ``functorial'' resolution, that is universal in
  the sense that it commutes with smooth morphisms, see \cite[Thms.~3.35 and
  3.45]{Kollar07}. The $\pi$-exceptional set $E$ will then be a non-empty
  divisor in $\wtilde X$, with snc support.

  Next, let $\theta_1$, \ldots, $\theta_n$ be sections in $\sT_X$ that freely
  generate $\sT_X$ in a neighbourhood $U$ of a given point $x \in X$. For
  simplicity of notation, we assume in the following that $U=X$. Given that
  $\pi$ is the \emph{functorial} resolution, and that the singular set
  $X_{\sing}$ is invariant under any automorphism, it follows from
  \cite[Cor.~4.7]{GKK08} that we may lift each $\theta_j$ to a logarithmic
  vector field on $\wtilde X$,
  \begin{equation}\label{eq:ral}
    \wtilde \theta_j \in H^0\bigl(\widetilde X, \sT_{\wtilde X}(-\log E) \bigr)
    \subseteq H^0\bigl(\widetilde X, \sT_{\wtilde X} \bigr).
  \end{equation}
  Notice that away from $E$, the vector fields $\wtilde \theta_j$ are linearly
  independent. Choosing the dual basis, we will therefore obtain a set of
  differential forms
  $$
  \omega_1, \dots, \omega_n \in H^0\bigl(\wtilde X \setminus E,\,
  \Omega^1_{\wtilde X} \bigr) \quad\text{such that}\quad \forall i,j: \omega_i
  \bigl(\wtilde \theta_j|_{\wtilde X \setminus E} \bigr) = \delta_{ij}\cdot
  {\boldsymbol 1}_{\wtilde X \setminus E},
  $$
  where ${\boldsymbol 1}_{\wtilde X \setminus E}$ is the constant function on
  $\wtilde X \setminus E$ with value $1$.  By the Extension
  Theorem~\ref{thm:main} and Remark~\ref{rem:whyextension}, the $\omega_i$
  extend to differential forms that are defined on all of $\wtilde X$,
  \begin{equation}\label{eq:dualform}
    \wtilde \omega_1, \dots, \wtilde \omega_n \in H^0\bigl(\wtilde X,\,
    \Omega^1_{\wtilde X} \bigr) \quad\text{such that}\quad
    \forall i,j: \wtilde \omega_i \bigl(\wtilde \theta_j \bigr)
    = \delta_{ij} \cdot {\boldsymbol 1}_{\wtilde X}.
  \end{equation}
  Now, if we evaluate the vector fields $\wtilde \theta_j \in
  H^0\bigl(\widetilde X, \sT_{\wtilde X} \bigr)$ at any smooth point $p$ of
  $E$, the inclusion in \eqref{eq:ral} shows that the tangent vectors
  obtained,
  $$
  \theta_1(p), \dots, \theta_n(p) \in \sT_{\wtilde X}\otimes \kappa(p)
  $$
  actually lie in $\sT_E\otimes \kappa(p)$. In particular, the tangent vectors
  $\theta_i(p)$ are linearly dependent. This contradicts \eqref{eq:dualform}
  and completes the proof.
\end{proof}

\section{Bogomolov-Sommese type results on log canonical spaces}
\label{sec:BogomolovSommese}

\subsection{Introduction and statement of the result}

In this section, we use the Extension Theorem~\ref{thm:main} to generalise the
Bogomolov-Sommese vanishing theorem to the log canonical setting and to Campana's
``geometric orbifolds''. In its standard version, \cite[Cor.~6.9]{EV92}, the theorem
limits positivity of invertible sheaves of differentials, asserting that for any \change{reduced}
snc pair $(X,D)$, any invertible sheaf of $p$-forms has Kodaira-Iitaka dimension no
more than $p$, i.e.,
\begin{sequation}\label{eq:BSV}
  \forall \text{ invertible }\sA \subseteq \Omega^p_X(\log D) \,\, : \,\,
  \kappa(\sA) \leq p,
\end{sequation}
Theorem~\ref{thm:BS}, the main result of this section, asserts that the inequality
\eqref{eq:BSV} also holds in the log canonical setting, for arbitrary $\mathbb
Q$-Cartier sheaves of rank one \change{(in the sense of Definition~\ref{def:QCartier})}.

For three-dimensional reduced pairs $(X,D)$ this was proven in
\cite[Thm.~1.4]{GKK08}. This three-dimensional case was an important
ingredient in the generalisation of Shafarevich hyperbolicity to families over
two-- and three-dimensional base manifolds, \cite{KK07b, KK08c}. There is hope
that Theorem~\ref{thm:BS} will allow to generalise Shafarevich hyperbolicity
to families over base manifolds of arbitrary dimension.

\begin{thm}[Bogomolov-Sommese vanishing for lc pairs]\label{thm:BS}
  Let $(X, D)$ be an lc pair, where $X$ is projective. If $\sA \subseteq
  \Omega^{[p]}_X(\log \lfloor D \rfloor)$ is a $\mathbb Q$-Cartier reflexive
  subsheaf of rank one, then $\kappa(\sA) \leq p$.
\end{thm}

\begin{subrem}
  The number $\kappa(\sA)$ appearing in the statement of Theorem~\ref{thm:BS}
  is the generalised Kodaira-Iitaka dimension introduced in
  Definition~\vref{def:KIdim}.
\end{subrem}

A proof of Theorem~\ref{thm:BS} is given in Section~\vref{ssec:PoTBS}.

\subsection{Bogomolov-Sommese vanishing in the orbifold setting}

In \cite{Cam04}, Campana introduced the category of ``geometric
orbifolds''. These are pairs $(X,D)$ where all coefficients of the boundary
divisor $D$ are of special form. Geometric orbifolds can in many ways be seen
as interpolating between the compact and the logarithmic setup. As the word
``geometric orbifold'' is perhaps not universally accepted in this context, we
prefer to call $(X,D)$ a ``$\mathcal C$-pair'' in this paper.  A brief
overview and precise definitions for all notions that are relevant to our
discussion are found in \cite[Part~I]{JKSpecialBaseManifolds}.

Essentially all notions used in the compact or logarithmic setup can be
generalised to $\mathcal C$-pairs. Examples include the following.
\begin{itemize}
\item Given $p,q\in\bN$, there exist reflexive sheaves of $\mathcal
  C$-differentials $\Sym^{[q]}_{\mathcal C} \Omega^p_X(\log D)$,
  \cite[Sect.~3.5]{JKSpecialBaseManifolds}, with inclusions
  $$
  \Sym^{[q]} \Omega^{[p]}_X(\log \lfloor D \rfloor) \subseteq
  \Sym^{[q]}_{\mathcal C} \Omega^p_X(\log D) \subseteq \Sym^{[q]}
  \Omega^{[p]}_X(\log \lceil D \rceil).
  $$
  In case $q = 1$ one has the equality $\Sym^{[1]}_{\mathcal C}
  \Omega^p_X(\log D) = \Omega^{[p]}_X(\log \lfloor D \rfloor)$.

\item Given a reflexive subsheaf $\sA \subseteq \Sym^{[1]}_{\mathcal C}
  \Omega^p_X(\log D)$ of rank one, there exists a notion of a $\mathcal
  C$-Kodaira dimension, denoted by $\kappa_{\mathcal C}(\sA)$ that takes
  fractional parts of $D$ into account,
  \cite[Def.~4.3]{JKSpecialBaseManifolds}. In general, one has
  $\kappa_{\mathcal C}(\sA) \geq \kappa(\sA)$.
\end{itemize}

Sheaves of $\mathcal C$-differentials seem particularly suitable for the
discussion of positivity on moduli spaces, cf.~\cite{JK09}. In this context,
the following strengthening of Theorem~\ref{thm:BS} promises to be of great
importance.

\begin{thm}[Bogomolov-Sommese vanishing for lc $\mathcal C$-pairs]\label{prop:BSvan}
  Let $(X, D)$ be a $\mathcal C$-pair.  Assume that $X$ is projective and
  $\bQ$-factorial, that $\dim X \leq 3$, and that the pair $(X,D)$ is lc. If $1 \leq
  p \leq \dim X$ is any number and if $\sA \subseteq \Sym^{[1]}_{\mathcal C}
  \Omega^p_X(\log D)$ is a reflexive sheaf of rank one, then $\kappa_{\mathcal
    C}(\sA) \leq p$.
\end{thm}

\begin{subrem}
  The important point in Theorem~\ref{prop:BSvan} is the use of the $\mathcal
  C$-Kodaira dimension $\kappa_{\mathcal C}(\sA)$ instead of the usual Kodaira
  dimension of $\sA$.
\end{subrem}

\begin{proof}[Proof of Theorem~\ref{prop:BSvan}]
  Using the Bogomolov-Sommese vanishing theorem for lc pairs,
  Theorem~\ref{thm:BS} instead of the weaker version \cite[Thm.~1.4]{GKK08},
  the proof from \cite[Sect.~7]{JKSpecialBaseManifolds} applies verbatim.
\end{proof}

\subsection{Proof of Theorem~\ref*{thm:BS}}\label{ssec:PoTBS}

We argue by contradiction and assume that there exists a reflexive subsheaf
$\sA \subseteq \Omega^{[p]}_X(\log \lfloor D \rfloor)$ with Kodaira-Iitaka
dimension $\kappa(\sA) > p$. Let $\pi: \widetilde X \to X$ be a \slr of the
pair $(X, D)$. We consider the following reduced snc divisors on $\wtilde X$,
$$
\begin{array}{ll}
  E  & := \text{$\pi$-exceptional set}, \\
  E' & := \supp \bigl( \pi^{-1}_*D + E \bigr), \\
  \widetilde D & := \text{largest reduced divisor in } \pi^{-1}\bigl(\text{non-klt locus of }(X,D)\bigr)
\end{array}
$$
Since $\widetilde D \subseteq E'$, the Pull-Back
Theorem~\ref{thm:generalpullback} for reflexive differentials implies that
there exists an embedding $\pi^{[*]}\sA \into \Omega^p_{\wtilde X}(\log
E')$. Let $\sC \subseteq \Omega^p_{\wtilde X} \bigl( \log E' \bigr)$ be the
saturation of the image, which is reflexive by \cite[Lem.~1.1.16 on
p.~158]{OSS}, and in fact invertible by \cite[Lem.~1.1.15 on
p.~154]{OSS}. Further observe that for any $k \in \bN$, the subsheaf
$\sC^{\otimes k} \subseteq \Sym^k \Omega^p_{\wtilde X}(\log E' )$ is likewise
saturated. To prove Theorem~\ref{thm:BS} it suffices to show that
\begin{equation}\label{eq:kappaupstairs}
  \kappa(\sC) \geq \kappa(\sA) > p
\end{equation}
which contradicts the standard Bogomolov-Sommese Vanishing Theorem for snc
pairs, \cite[Cor.~6.9]{EV92}.

\subsubsection*{Choosing a basis of sections}

Choose a number $m$ such that $\dim \overline{\phi_m(X)} = \kappa(\sA) =:
\kappa $, where $\phi_m$ is the rational map used in the definition of Kodaira
dimension, Definition~\vref{def:KIdim}. Let $B := \{\sigma_1, \dots,
\sigma_\kappa\}$ be a a basis of $H^0\bigl(X,\, \sA^{[m]} \bigr)$. If $\sigma
\in B$ is any element, consider the pull-back $\pi^*(\sigma)$, which is a
rational section in $\sC^{\otimes m}$, possibly with poles along the
exceptional set $E$. To show~\eqref{eq:kappaupstairs}, it suffices to prove
that $\pi^*(\sigma)$ does not have any poles as a section in $\sC^{\otimes
  m}$, i.e., that
\begin{equation}\label{eq:nopole}
  \pi^*(\sigma) \in H^0\bigl(\wtilde X, \sC^{\otimes m} \bigr)
  \quad \forall \sigma \in B.
\end{equation}

\noindent
Since $\sC^{\otimes m}$ is saturated in $\Sym^m \Omega^p_{\wtilde X}(\log E')$, to
show~\eqref{eq:nopole}, it suffices to show that the $\pi^*(\sigma)$ do not have any
poles as sections in the sheaf of symmetric differentials, i.e., that
\begin{equation}\label{eq:nopole2}
  \pi^*(\sigma) \in H^0\bigl( \wtilde X,\, \Sym^m \Omega^p_{\wtilde X} (\log
  E') \bigr) \quad \forall \sigma \in B.
\end{equation}

\subsubsection*{Taking an index-one-cover}

The statement of \eqref{eq:nopole2} is local on $X$, hence we may shrink $X$ and
assume that a suitable reflexive tensor power of $\sA$ is trivial, say $\sA^{[r]}
\simeq \sO_X$. Let $\gamma : Z \to X$ be the associated index-one-cover, \change{cf.\
  \cite[Def.~2.52]{KM98}, \cite[Sect.~2.D]{Hacon-Kovacs10}}.
 \change{Let $D=\sum_i d_iD_i$ where $D_i$ are reduced irreducible divisors and
  $d_i\in\bQ^{>0}$. Given any index $i$, let $\Delta_i := \gamma^{-1}(D_i)$ be the
  reduced irreducible divisor supported on $\gamma^{-1}(\supp D_i)$, and set $\Delta
  := \sum_i d_i\Delta_i$.  Since $\gamma$ is \'etale in codimension $1$ by
  construction, it follows that $K_Z+\Delta=\gamma^*(K_X+D)$ and hence the pair $(Z,
  \Delta)$ is again lc by \cite[Prop.~5.20]{KM98}. Furthermore, the sheaf $\sB :=
  \gamma^{[*]}(\sA)$ is a locally free subsheaf of $\Omega^{[p]}_Z(\log \lfloor
  \Delta \rfloor)$, with section
$$
\sigma_Z := \gamma^{[*]}(\sigma) \in H^0\bigl(Z,\, \sB^{\otimes m} \bigr).
$$
}

\subsubsection*{A partial resolution of $Z$}

Next, consider the commutative diagram
$$
\xymatrix{ %
  \wtilde Z \ar[rrr]^{\txt{\scriptsize $\wtilde \gamma$, finite}}
  \ar[d]_{\pi_Z} &&&
  \wtilde X \ar[d]^{\pi} \\
  Z \ar[rrr]_{\txt{\scriptsize $\gamma$, finite}} &&& X, }
$$
where $\wtilde Z$ is the normalisation of the fibre product $Z \times_X
\change{{\wtilde X}}$. We consider the following reduced divisors on $\wtilde Z$,
$$
\begin{array}{lll}
  E_Z & := \text{$\pi_Z$-exceptional set} &= \supp \wtilde \gamma^*E,\\
  E'_Z & := \supp \bigl( (\pi \circ \wtilde
  \gamma)_*^{-1}(D) + \widetilde E \bigr) &= \supp \wtilde \gamma^*E',\\
  \wtilde \Delta & := \text{largest reduced divisor in }
  \pi^{-1}\bigl(\text{non-klt locus of } (Z,\Delta) \bigr)
\end{array}
$$
The inclusion $\wtilde \Delta \subseteq E'_Z$ and 
Theorem~\ref{thm:generalpullback} gives an embedding $\pi^*_Z \, \sB \into
\Omega^{[p]}_{\wtilde Z}(\log E'_Z)$. In fact, since $\sB$ is locally free, we also
obtain an embedding of tensor powers,
$$
\iota_m : \pi^*_Z \, \sB^{\otimes m} \into \Sym^{[m]} \Omega^{[p]}_{\wtilde Z}
(\log E'_Z).
$$

\subsubsection*{Completion of proof}

Since the index-one-cover $\gamma$ is étale away from the singularities of
$X$, the morphism $\wtilde \gamma$ is étale outside of $E \subseteq E'$. In
particular, the standard pull-back morphism of logarithmic differentials,
defined on the smooth locus of $\wtilde Z$, gives an isomorphism
$$
\wtilde \gamma^{[*]} \left( \Sym^m \Omega^{p}_{\wtilde X}(\log E') \right)
\simeq \Sym^{[m]} \Omega^{[p]}_{\wtilde Z} \bigl( \log E'_Z ).
$$
This isomorphism implies that in order to prove~\eqref{eq:nopole2}, it suffices to
show that
\begin{equation}\label{eq:nopole3}
  \wtilde \gamma^{[*]} \bigl( \pi^*(\sigma) \bigr) \in
  H^0\bigl(\wtilde Z,\, \Sym^{[m]} \Omega^{[p]}_{\wtilde Z}(\log E'_Z) \bigr).
\end{equation}
The inclusion in \eqref{eq:nopole3}, however, follows when we observe that the
rational section $\wtilde \gamma^{[*]} \bigl( \pi^*(\sigma) \bigr)$ of
$\Sym^{[m]} \Omega^{[p]}_{\wtilde Z}(\log E'_Z)$ and the regular section
$\iota_m (\sigma_Z) = \wtilde \pi^{[*]} (\sigma_Z)$ agree on the open set
$\wtilde Z \setminus \supp E_Z$. This finishes the proof
Theorem~\ref{thm:BS}. \qed

\part{REFLEXIVE FORMS ON DLT PAIRS}
\label{part:2}

\section{Overview and main results of Part~\ref*{part:2}}
\label{sec:reflxDLTintro}

\subsection{Introduction}

Logarithmic Kähler differentials on snc pairs are canonically defined. They
are characterised by strong universal properties and appear accordingly in a
number of important sequences, filtered complexes and other
constructions. First examples include the following:%

\addtocounter{thm}{1}
\begin{enumerate}
\item\ilabel{il:u1} the pull-back property of differentials under arbitrary
  morphisms,

\item\ilabel{il:u2} relative differential sequences for smooth morphisms,

\item\ilabel{il:u4} residue sequences associated with snc pairs, and

\item\ilabel{il:u5} the description of Chern classes as the extension classes
  of the first residue sequence.
\end{enumerate}

On singular spaces, Kähler differentials enjoy similar universal properties,
but the sheaves of Kähler differentials are hardly ever normal, often contain
torsion parts and are notoriously hard to deal with. For one example of the
problems arising with Kähler differentials, observe that $\Omega^p_X$ is
generally not pure in the sense of \cite[Def.~1.1.2]{HL97}, so that no
Harder-Narasimhan filtration ever exists.

Many of these problems can be overcome by using the sheaves $\Omega^{[p]}_X$
of reflexive differentials. For instance, Harder-Narasimhan filtrations exist
for $\Omega^{[p]}_X$, sheaves of reflexive differentials enjoy good
push-forward properties, \cite[Lem.~5.2]{KK08c}, and reflexive differential
can be constructed using Hodge-theoretic methods in a number of settings that
are of interest for moduli theory, see for instance \cite[Thm.~1.4]{VZ02} and
the application in \cite[Thm.~5.3]{KK08c}.

Reflexive differentials do in general not enjoy the same universal properties
as Kähler differentials. However, we have seen in
Section~\ref{thm:generalpullback} as one consequence of the Extension Theorem
that reflexive differentials do have good pull-back properties if we are
working with dlt pairs, and that an analogue of the property~\iref{il:u1}
holds. In the present Part~\ref{part:2} of this paper, we would like to make
the point that each of the Properties~\iref{il:u2}--\iref{il:u5} has a very
good analogue for reflexive differentials if we are working with dlt
pairs. This makes reflexive differential extremely useful in practise. In a
sense, it seems fair to say that ``reflexive differentials and dlt pairs are
made for one another''.

\subsection{Overview of Part~\ref*{part:2}}

We recall the precise statements of the properties~\iref{il:u2}--\iref{il:u5},
formulate and prove generalisations to singular spaces in
Sections~\ref{sec:relReflSeq}--\ref{sec:part2-last} below.

Unlike the property~\iref{il:u1}, whose generalisation to singular spaces is given
in Theorem~\ref{thm:generalpullback} as a corollary of our main result, the
results of this section do not depend on the Extension Theorem~\ref{thm:main},
but follow from a detailed analysis of the local analytic codimension $2$
structure of dlt pairs. We have therefore included a preparatory
Section~\ref{sect:dltlocalstructure} devoted to the discussion of dlt pairs.

\section{The local structure of dlt pairs in codimension {$2$}}
\label{sect:dltlocalstructure}

The proofs of the results announced in the previous Section~\ref{sec:reflxDLTintro}
will be given in Sections~\ref{sec:relReflSeq}--\ref{sec:part2-last} below. To
prepare for the proofs, this section contains a detailed analysis of singularities
that appear in the minimal model program. Since we are concerned with reflexive
differentials and their restrictions to boundary components, we are mostly interested
in structure results that hold in codimension $2$.

Although the statements proven in this section are probably known to experts,
to the best of our knowledge, proofs of these are not available in the
literature. Since our arguments in other parts of the paper crucially depend
on the detailed knowledge about the structure of dlt pairs as presented in
this section, we have therefore chosen to include proofs of all statements
required later, also for the reader's convenience.

\subsection{$\pmb{\mathbb Q}$-factoriality of dlt pairs in codimension \pmb{$2$}}

If $(X,D)$ is a dlt surface pair, it is well-understood that $X$ is automatically
$\mathbb Q$-factorial, \cite[Prop.~4.11]{KM98}. This remains true even if $(X,D)$ is
only assumed to be numerically dlt and $K_X+D$ is not assumed to be $\mathbb
Q$-Cartier.  \change{A higher dimensional dlt pair is not necessarily $\mathbb{Q}$-factorial,
but the underlying space of a dlt pair is always $\mathbb{Q}$-factorial in codimension
$2$ regardless of its dimension.}

\begin{prop}
  [$\mathbb Q$-factoriality of dlt pairs in $\codim = 2$]\label{prop:QfactofDLT} Let
  $(X, D)$ be a dlt pair. Then there exists a closed subset $Z \subset X$ with
  $\codim_X Z \geq 3$ such that $X \setminus Z$ is $\mathbb Q$-factorial.
\end{prop}
\PreprintAndPublication{
\begin{proof}
  Since every dlt pair is a limit of klt pairs, \cite[Prop.~2.43]{KM98}, there
  exists a $\mathbb Q$-divisor $D'$ on $X$ such that $(X,D')$ is klt.  We may
  therefore assume from the beginning without loss of generality that $(X,D)$
  is klt.

  Applying \cite[Cor.~1.4.3]{BCHM06} with $\Delta_0 = \Delta$ and
  $\mathfrak{E} = \emptyset$, we obtain a log-terminal model of $X$, i.e., a
  small birational morphism $p: Y\to X$ from a $\mathbb{Q}$-factorial variety
  $Y$ to $X$. Since $p$ is small, there exists a closed subset $Z \subset X$
  with codimension $\codim_X Z \geq 3$ such that $p^{-1}: X \setminus Z \to Y
  \setminus p^{-1}(Z)$ is well-defined and an isomorphism. This finishes the
  proof.
\end{proof}

\begin{rem}
  Instead of using the full force of minimal model theory, it is certainly
  possible to give an elementary (though lengthy) proof that follows the
  arguments of \cite[Lem.~4.10]{KM98} after using repeated hyperplane sections
  to reduce to the surface case.
\end{rem}

The reader who is willing to use the classification of
  dlt pairs over arbitrary, not necessarily closed, fields of characteristic
  zero might prefer the following argument, suggested by János Kollár.

  \begin{proof}[Alternate proof of Proposition~\ref{prop:QfactofDLT} using localization]
    Let $U \subseteq X$ be the maximal open set which is locally $\mathbb
    Q$-factorial, and set $Z = X \setminus U$. To prove
    Proposition~\ref{prop:QfactofDLT}, it suffices to show that $\codim_X Z
    \geq 3$. If not, let $Z' \subseteq Z$ be a component of codimension
    $\codim_X Z' = 2$. Localisation at the generic point of $Z'$ then gives a
    2-dimensional local dlt pair, which is defined over a field of
    characteristic zero but which is not $\mathbb Q$-factorial. This
    contradicts \cite[Prop.~4.11]{KM98}.
  \end{proof}
}{\change{A detailed proof of Proposition~\ref{prop:QfactofDLT} can be found in the preprint
    version \cite{GKKP10} of this paper.}}

\subsection{The local structure of canonical pairs in codimension \pmb{$2$}}

If $(X,D)$ is a canonical (or log canonical) pair and $x \in X$ a point, then the
discrepancy of $(X,D)$ at $x$ is small if either $X$ is very singular at $x$ or if
$D$ has high multiplicity at $x$. Conversely, it is true that $X$ cannot be very
singular wherever the multiplicity of $D$ is large. This principle leads to the
following description of canonical pairs along reduced components of the boundary
divisor $D$.

\begin{prop}[Codimension $2$ structure of canonical pairs along the
  boundary]\label{prop:reducedCanonical}
  Let $(X,D)$ be a canonical pair with $\lfloor D \rfloor \not = 0$. Then
  there exists a closed subset $Z \subset \supp \lfloor D \rfloor$ with
  $\codim_X Z \geq 3$ such that for any point $z \in (\supp \lfloor D \rfloor)
  \setminus Z$,
  \begin{enumerate}
  \item\label{il:ackX} the pair $(X,D)$ is snc at $z$, and
  \item\label{il:ackY} the subvariety $\supp D$ is smooth at $z$.
  \end{enumerate}
\end{prop}
\begin{proof}
  Consider general hyperplanes $H_1, \ldots, H_{\dim X-2} \subseteq X$ and set
  $$
  \bigl( H,\, D_H \bigr) := \bigl( H_1 \cap \cdots \cap H_{\dim
    X-2},\, D \cap H_1 \cap \cdots \cap H_{\dim X-2} \bigr).
  $$
  Lemma~\ref{lem:cuttingDown2} then asserts that $( H,\, D_H )$ is a canonical
  surface pair. The classification of these pairs, \cite[Thm.~4.5(2)]{KM98},
  therefore applies to show that both $H$ and $\supp D_H$ are smooth along
  $\supp \lfloor D_H \rfloor$. The Cutting-Down Lemma~\ref{lem:cuttingDown}
  then gives that
  \begin{itemize}
  \item the properties~(\ref{prop:reducedCanonical}.\ref{il:ackX}) and
    (\ref{prop:reducedCanonical}.\ref{il:ackY}) hold for all points $z \in
    \supp \lfloor D_H \rfloor$, and
  \item we have an equality of sets, $\supp \lfloor D_H \rfloor = \supp \bigl(
    \lfloor D \rfloor \bigr) \cap H$.
  \end{itemize}
  The claim then follows because the hyperplanes $H_i$ are generic.
\end{proof}

\subsection{The local structure of klt pairs in codimension \pmb{$2$}}
\label{ssec:kltLocal}

We show that the underlying space of a klt pair has quotient singularities in
codimension $2$. This result is used in
Sections~\ref{sec:relReflSeq}--\ref{sec:part2-last}, where we reduce the study of
reflexive differentials on singular spaces to the study of $G$-invariant
differentials on suitable local Galois coverings with Galois group $G$.

\begin{prop}[Klt spaces have quotient singularities in codimension
  $2$]\label{prop:kltarequot}
  Let $(X,D)$ be a klt pair. Then there exists a closed subset $Z \subset X$
  with $\codim_X Z \geq 3$ such that $X \setminus Z$ has quotient
  singularities.

  More precisely, every point $x \in X \setminus Z$ has an analytic
  neighbourhood that is biholomorphic to an analytic neighbourhood of the
  origin in a variety of the form $\C^{\dim X}/ G$, where $G$ is a finite
  subgroup of $GL_{\dim X}(\C)$ that does not contain any
  quasi-reflections. The quotient map is a finite Galois map, totally branched
  over the singular locus and étale outside of the singular set.
\end{prop}

\begin{subrem}\label{rem:kawamata}
  For families of du~Val singularities, similar statements appear in the
  literature, e.g.~in \cite[Cor.~1.14]{Reid79} or \cite[Proof of
  Prop.~1]{Namikawa01}, but with little or no indication of proof. Our proof
  of Proposition~\ref{prop:kltarequot} employs Grauert's miniversal
  deformation space for analytic germs of isolated singularities, tautness of
  dlt surface singularities and Teissier's ``economy of the miniversal
  deformation'', \cite{Teissier}.  We would like to thank Yujiro Kawamata and
  Gert-Martin Greuel for discussions on the subject.
\end{subrem}

The remainder of the present Section~\ref{ssec:kltLocal} is devoted to a proof
of Proposition~\ref{prop:kltarequot}. We subdivide the proof into a number of
relatively independent steps.

\subsubsection{Proof of Proposition~\ref*{prop:kltarequot}: projection to the singular set}

Observe that the assertion of Proposition~\ref{prop:kltarequot} is local on
$X$. Recalling from Proposition~\ref{prop:QfactofDLT} that $X$ is $\mathbb
Q$-factorial in codimension $2$, observe that it suffices to prove
Proposition~\ref{prop:kltarequot} under the following additional assumption.

\begin{awlog}
  The variety $X$ is $\mathbb Q$-factorial. In particular, we assume that the
  pair $(X, \emptyset)$ is klt, cf.~\cite[Cor.~2.39]{KM98}. The singular locus
  $T := X_{\sing}$ is irreducible and of codimension $\codim_X T = 2$.
\end{awlog}

Recall from Proposition~\ref{prop:projection} that there exists an open set
$X^\circ \subseteq X$ such that $T^\circ := T \cap X^\circ$ is not empty, and
a diagram
$$
\xymatrix{%
  Z^\circ \ar[rr]^{\gamma}_{\text{finite, étale}} \ar[d]_{\phi} && X^\circ \\
  S^\circ }
$$
such that the restriction of $\phi$ to any connected component of
$\gamma^{-1}(T^\circ)$ is an isomorphism.  It is clear that $X$ is smooth at
all points of $X \setminus (X^\circ \cup T)$, and that $\codim_X (T \setminus
T^\circ) \geq 3$. Consequently, it suffices to prove
Proposition~\ref{prop:kltarequot} for points contained in $X^\circ$. Better
still, since the assertion of Proposition~\ref{prop:kltarequot} is local in
the analytic topology, it suffices to prove Proposition~\ref{prop:kltarequot}
for the variety $Z^\circ$, even after removing all but one component of
$\gamma^{-1}(T^\circ)$. We may therefore assume the following.

\begin{awlog}\label{awlog:x1}
  There exists a surjective morphism $\phi : X \to S$ with connected fibres
  whose restriction $\phi |_T : T \to S$ is isomorphic.
\end{awlog}

\begin{obs}\label{obs:mayshrinkT}
  Let $U \subseteq S$ be any Zariski-open subset. As in the previous
  paragraph, observe that $X$ is smooth at all points of $X \setminus (\phi
  ^{-1}(U) \cup T)$, and that $\codim_X T\setminus \phi^{-1}(U) \geq 3$. As
  above, we see that to prove Proposition~\ref{prop:kltarequot}, it suffices
  to consider the open set $\phi ^{-1}(U) \subseteq X$ only.
\end{obs}

Observation~\ref{obs:mayshrinkT}, together with the Generic Flatness Lemma,
\cite[Thm.~5.12]{FGIKN}, and the Cutting-Down Lemma~\ref{lem:cuttingDown}
allows to assume the following.

\begin{awlog}\label{awlog:x2}
  The morphism $\phi$ is flat. Given any point $s \in S$, the preimage $X_s :=
  \phi^{-1}(s)$ is a normal klt surface\footnote{More precisely, we assume
    that the pair $(X_s, \emptyset)$ is klt.}.  If $t_s \in T$ is the unique
  point with $\phi(t_s) = s$, then $X_s$ is smooth away from $t_s$.
\end{awlog}

\subsubsection{Proof of Proposition~\ref*{prop:kltarequot}: simultaneous resolution
  of singularities}

In this subsection, we aim to show that, possibly shrinking $S$ further, there
exists a simultaneous minimal resolution of the surface singularities
$(X_s)_{s \in S}$.

\begin{claim}\label{claim:gensimres}
  There exists a dense smooth open set $S^\circ \subseteq S$ with preimage
  $X^\circ := \phi^{-1}(S^\circ)$, and a resolution of singularities $\pi:
  \wtilde X^\circ \to X^\circ$ such that the composition $\psi := \phi
  \circ \pi$ is smooth over $S^\circ$, and such that the fibre $\wtilde X_s :=
  \psi^{-1}(s)$ is a minimal resolution of the klt surface $X_s$, for every $s
  \in S^\circ$.
\end{claim}
\begin{proof}
  To start, let $\pi: \wtilde X \to X$ be any resolution of singularities. If
  $s \in S$ is general, it is then clear that $\wtilde X_s$ is smooth. We may
  thus choose $S^\circ$ such that all scheme-theoretic fibres $(\wtilde
  X_s)_{s \in S^\circ}$ are smooth. Set $\wtilde X^\circ :=
  \wtilde\phi^{-1}(T^\circ)$.

  Now, if $K_{\wtilde X^\circ/X^\circ}$ is nef, then none of the surfaces $\wtilde
  X_s$ contains a 
  $(-1)$-curve, $\pi$ is a simultaneous minimal resolution of the surface
  singularities $(X_s)_{s\in S^\circ}$, and the proof is complete.

  If $K_{\wtilde X^\circ/X^\circ}$ is not nef, then the Relative Cone Theorem,
  \cite[Thm.~3.25]{KM98} asserts that there exists a factorisation of $\pi$
  via a birational, $\pi$-relative contraction of an extremal ray,
  $$
  \xymatrix{ \wtilde X^\circ \ar@/^0.3cm/[rr]^{\pi} \ar[r]_{\pi_1} & \what
    X^\circ \ar[r]_{\pi_2} & X^\circ \ar[r]^{\phi } & S^\circ.  }
  $$
  If $\pi_1$ is a divisorial contraction, then $\what X^\circ$ is terminal,
  \cite[Cor.~3.43.(3)]{KM98}, and $\codim_{\what X^\circ} \what X^\circ_{\sing} \geq
  3$, \cite[Cor.~5.18]{KM98}. If $\pi_1$ is a small contraction, it is likewise clear
  that $\codim_{\what X^\circ} \what X^\circ_{\sing} \geq 3$. In either case, the
  singular set $\what X^\circ_{\sing}$ does not dominate $S^\circ$. Replacing
  $S^\circ$ by a suitable subset, we may assume that $\pi_2 : \what X^\circ \to
  X^\circ$ is a resolution of singularities with relative Picard-number $\rho(\what
  X^\circ/X^\circ) < \rho(\wtilde X^\circ/X^\circ)$. Replacing $\wtilde X^\circ$ by
  $\what X^\circ$ and repeating the process finitely many times, we will end up with
  a resolution where $K_{\wtilde X^\circ/X^\circ}$ is nef.
  Claim~\ref{claim:gensimres} is thus shown.
\end{proof}

Claim~\ref{claim:gensimres} and Observation~\ref{obs:mayshrinkT} together
allow to assume the following.

\begin{awlog}
  There exists a resolution of singularities $\pi: \wtilde X \to X$ such that
  the composition $\psi := \phi \circ \pi$ is smooth, and such that for
  any $s \in S$, the fibre $\wtilde X_s := \psi^{-1}(s)$ is a minimal
  resolution of the klt surface singularity $X_s$.
\end{awlog}

\subsubsection{Proof of Proposition~\ref*{prop:kltarequot}: the isomorphism type of
  the surface germs $X_s$}

Given a point $s \in S$, we consider the germ of the pointed surface $X_s$ at the
point $t_s \in T$, the unique point of $T$ that satisfies $\phi(t_s) = s$. We use the
symbol $(\mathbf{X}_s \ni t_s)$ to denote this germ.

\begin{claim}\label{claim:germsisomorphic}
  There exists a dense Zariski-open subset $S^\circ \subseteq S$ such that for
  any two points $s_1, s_2 \in S^\circ$, the associated germs of the pointed
  surfaces are isomorphic, $(\mathbf{X}_{s_1} \ni t_{s_1}) \simeq
  (\mathbf{X}_{s_2} \ni t_{s_2})$.
\end{claim}
\begin{proof}
  By 
  \cite[Cor.~5.1]{Verdier76}, there exists a Zariski dense open subset $S^\circ
  \subseteq S$ with preimage $\wtilde X^\circ := \psi^{-1}(S^\circ)$ such that
  $\psi|_{\wtilde X^\circ}: \wtilde X^\circ \to S^\circ$ is a topological fibre
  bundle (in the analytic topology). As a consequence of the classification of
  log-terminal surface singularities, cf.\ \cite[Thm.~9.6]{Kawamata88}, the analytic
  isomorphism type of any such singularity is uniquely determined by the resolution
  graph (labelled with self-intersection numbers) of its minimal resolution. In other
  words, log terminal surface singularities are taut in the sense of Laufer
  \cite[Def.~1.1]{Laufer73}.  Since $\psi|_{\wtilde X^\circ}$ is a fibre bundle,
  Claim~\ref{claim:germsisomorphic} follows.
\end{proof}

Again, Observation~\ref{obs:mayshrinkT} allows to shrink $S$ and assume the
following.

\begin{awlog}\label{awlog:x3}
  For any two points $s_1, s_2 \in S$, we have an isomorphism
  $(\mathbf{X}_{s_1} \ni t_{s_1}) \simeq (\mathbf{X}_{s_2}\ni t_{s_2})$.
\end{awlog}

\subsubsection{Proof of Proposition~\ref*{prop:kltarequot}: the completion of the
  proof}

Let now $t \in T = X_{\sing}$ be any point, with image $s := \phi(t)$. Note
that by Assumption~\ref{awlog:x2}, the point $t$ is the unique singular point
in the klt surface $X_s$. Since $(\mathbf{X}_s \ni t)$ is the germ of an
isolated singularity, a theorem of Grauert, \cite{Grauert72}, asserts the
existence of a miniversal deformation space $(\mathbf{U} \ni 0)$ for
$(\mathbf{X}_s \ni t)$, which is itself a germ of a pointed complex space; we
refer to \cite[Sect.~II.1]{MR2290112} for these matters. Since $\phi: X \to S$
is flat, we obtain a holomorphic map of pointed space germs, say $\eta:
(\mathbf{S} \ni s) \to (\mathbf{U} \ni 0)$.  Since all fibres of $\phi $ give
isomorphic space germs by Assumption~\ref{awlog:x3}, it follows from the
``economy of the miniversal deformation'', \cite[Cor.~2]{HauserMueller},
\cite[Thm.~4.8.4]{Teissier} that $\eta$ is the constant map which maps the
germ $(\mathbf{T} \ni t)$ to $0 \in \mathbf{U}$. The universal property of the
miniversal deformation space then gives an isomorphism of germs
$$
(\mathbf{X} \ni t) \simeq \bigl(\mathbf{S} \times \mathbf{X}_s \ni (s, t)
\bigr).
$$

Since $T$ and $S$ are smooth, there exists a neighbourhood $U$ of $t$ in $X$
such that $U$ is biholomorphic to $B^{\dim X-2} \times (X_s \cap U)$, where
$B^k$ denotes the unit ball in $\C^k$. It follows from the classification of
log terminal surface singularities and from the general description of
quotient singularities, cf.~\cite[Thm.~9.6]{Kawamata88} and \cite{MR0210944},
that the exits a finite group $G \subset GL_2(\C)$ without quasi-reflections
such that a neighbourhood of $t \in X_s$ is biholomorphic to a neighbourhood
of the origin in $\C^2 /G$. The quotient map is totally branched over the
origin and étale elsewhere.  Hence, $t \in X$ possesses a neighbourhood $U'
\subseteq U$ that is biholomorphic to a complex space of the form $(B^{\dim
  X-2} \times B^2)/G$, where $G$ is a finite group acting linearly and without
quasi-reflections on the second factor, and where the quotient map is totally
branched over the singular set and étale elsewhere.  \qed

\subsection{The local structure of dlt pairs in codimension \pmb{$2$}}
\label{ssec:9D}

We conclude the present Section~\ref{sect:dltlocalstructure} by describing the
codimension $2$ structure of dlt pairs along the reduced components of the
boundary, similarly to Proposition~\ref{prop:reducedCanonical} above. Since
dlt pairs are klt away from the reduced components of the boundary,
\cite[Prop.~2.41]{KM98}, Propositions~\ref{prop:kltarequot} and
\ref{prop:reducedDLT} together give a full account of the structure of dlt
pairs in codimension $2$. These results are summarised in
Corollary~\ref{cor:dltisQuot} at the end of this section.

\begin{prop}[Codimension $2$ structure of dlt pairs along the reduced boundary]\label{prop:reducedDLT}
  Let $(X, D)$ be a dlt pair with $\lfloor D \rfloor \not = 0$. Then there
  exists a closed subset $Z \subset X$ with $\codim_X Z \geq 3$ such that $X
  \setminus Z$ is $\mathbb Q$-factorial, and such that for every point $x \in
  \bigl( \supp \lfloor D \rfloor \bigr) \setminus Z$ one of the following two
  conditions holds.
  \begin{enumerate}
  \item\label{il:Camillo} The pair $(X, D)$ is snc at $x$, and the point $x$
    is contained in precisely two components of $D$. These components have
    coefficient one in $D$ and and intersect transversely at $x$.

  \item\label{il:Peppone} The divisor $\lfloor D \rfloor$ is smooth at $x$ and
    the pair $(X,D)$ is plt at $x$.
  \end{enumerate}
\end{prop}

As with Proposition~\ref{prop:reducedCanonical}, the proof of
Proposition~\ref{prop:reducedDLT} relies on cutting-down and on classification
results for surface pairs. Before starting with the proof, we recall the
relevant classification of dlt surface pairs for the reader's convenience.

\begin{fact}[\protect{Classification of dlt surface pairs, \cite[Cor.~5.55]{KM98}}]\label{fact:classDLTsurf}
  Let $\bigl(X,\, D \bigr)$ be a dlt surface pair, and let $x \in \supp
  \lfloor D \rfloor$ be any point. Then either one of the following
  holds.
  \begin{enumerate}
  \item\label{il:bonny} The pair $(X, D)$ is snc at $x$, and $x$ is
    contained in precisely two components of $D$. These components have
    coefficient one and intersect transversely at $x$.

  \item\label{il:clyde} The divisor $\lfloor D \rfloor$ is smooth at
    $x$. \qed
  \end{enumerate}
\end{fact}

With Fact~\ref{fact:classDLTsurf} at hand, the proof of
Proposition~\ref{prop:reducedDLT} becomes rather straightforward.

\begin{proof}[Proof of Proposition~\ref{prop:reducedDLT}]
  To start, recall from Proposition~\ref{prop:QfactofDLT} that $X$ is $\mathbb
  Q$-factorial in codimension $2$. Removing a suitable small subset, we may
  therefore assume without loss of generality that $X$ is $\mathbb
  Q$-factorial

  Consider general hyperplanes $H_1, \ldots, H_{\dim X-2} \subseteq X$, and
  set
  $$
  \bigl( H,\, D_H \bigr) := \bigl( H_1 \cap \cdots \cap H_{\dim X-2},\, D \cap
  H_1 \cap \cdots \cap H_{\dim X-2} \bigr).
  $$
  Then (\ref{lem:cuttingDown}.\ref{il:cdA3}) of the Cutting-Down
  Lemma~\ref{lem:cuttingDown} asserts that $\supp \bigl( \lfloor D_H \rfloor \bigr) =
  H \cap \supp \bigl( \lfloor D \rfloor \bigr)$. By general choice of the $H_i$, it
  suffices to show that the properties~(\ref{prop:reducedDLT}.\ref{il:Camillo}) or
  (\ref{prop:reducedDLT}.\ref{il:Peppone}) hold for all points $x \in \supp \bigl(
  \lfloor D_H \rfloor \bigr)$. Fix one such point for the remainder of the proof.

  By Lemma~\ref{lem:cuttingDown2}, the surface pair $\bigl( H,\, D_H \bigr)$
  is dlt, so that the classification stated in Fact~\ref{fact:classDLTsurf}
  applies. If we are in case~(\ref{fact:classDLTsurf}.\ref{il:bonny}), it
  follows from (\ref{lem:cuttingDown}.\ref{il:cdC}) and
  (\ref{lem:cuttingDown}.\ref{il:cdA2}) of Lemma~\ref{lem:cuttingDown} that
  the pair $(X, D)$ is snc at $x$, and that near $x$ the pair $(X, D)$ is of
  the form stated in (\ref{prop:reducedDLT}.\ref{il:Camillo}).

  We may thus assume that we are in case~(\ref{fact:classDLTsurf}.\ref{il:clyde}),
  where smoothness of $\lfloor D \rfloor$ at $x$ follows from
  (\ref{lem:cuttingDown}.\ref{il:cdB}). The fact that pair $(X,D)$ is plt at $x$
  follows from \cite[Prop.~5.51]{KM98}.
\end{proof}

\begin{cor}\label{cor:dltisQuot}
  Let $(X, D)$ be a dlt pair. Then there exists a closed subset $Z \subset X$
  with $\codim_X Z \geq 3$ such that $X^\circ := X \setminus Z$ is $\mathbb
  Q$-factorial, and such that there exists a covering of $X^\circ$ by subsets
  $(U_\alpha)_{\alpha \in A}$ that are open in the analytic topology, and
  admit covering maps
  $$
  \gamma_\alpha : V_\alpha \to U_\alpha \quad \text{finite Galois cover, étale
    in codimension one}
  $$
  such that the pairs $\bigl( V_\alpha, \gamma_\alpha^*\lfloor D \rfloor
  \bigr)$ are snc for all indices $\alpha \in A$.  Furthermore, the covering
  may be chosen to satisfy the following additional conditions.
  \begin{enumerate}
  \item\label{il:xxz}Only finitely many of the open sets, say $U_{\alpha_1}, \ldots,
    U_{\alpha_k}$, intersect $\supp \lfloor D \rfloor$. The sets $U_{\alpha_i}$ are
    open in the Zariski topology, and the covering maps $\gamma_{\alpha_i}$ are
    algebraic morphisms of quasi-projective varieties.

  \item\label{il:xxy} For any index $\alpha$ with $U_\alpha \cap \supp \lfloor
    D \rfloor = \emptyset$, the covering map $\gamma_\alpha$ is totally
    branched over the singular set, and étale elsewhere.
  \end{enumerate}
\end{cor}
\begin{subrem}
  Since the $\gamma_\alpha$ are étale in codimension one, round-down of
  divisors commutes with pulling-back. That is, we have equalities
  $\gamma_\alpha^* \lfloor D \rfloor = \lfloor \gamma_\alpha^* D \rfloor$ for
  all $\alpha \in A$.
\end{subrem}

\subsubsection{Proof of Corollary~\ref*{cor:dltisQuot}, setup of notation}

Removing a subset of codimension $3$, Proposition~\ref{prop:reducedDLT}
allows to assume that the variety $X$ is $\mathbb Q$-factorial. In particular,
we assume that the pair $(X, \lfloor D \rfloor)$ is likewise dlt,
\cite[Cor.~2.39]{KM98}.  We may therefore assume that $D$ is reduced, i.e.,
that $D = \lfloor D \rfloor$. Finally, consider the open set $X' := X
\setminus \supp D$ and observe that the pair $(X', \emptyset)$ is klt,
\cite[Prop.~2.41]{KM98}.

The open cover $(U_\alpha)_{\alpha \in A}$ will be constructed in two steps,
first covering $\supp D$ with finitely many Zariski-open sets, and then
covering $X'$ by (possibly infinitely many) sets that are open only in the
analytic topology. In each step, we might need to remove from $X$ finitely
many further sets of codimension $3$.

\subsubsection{Proof of Corollary~\ref*{cor:dltisQuot}, covering $\supp D$}

Assuming that $D \not = 0$ and removing a suitable subset of
codimension $3$, we may assume that for all points $x \in \supp D$ either
Condition~(\ref{prop:reducedDLT}.\ref{il:Camillo}) or Condition~(\ref{prop:reducedDLT}.\ref{il:Peppone}) of Proposition~\ref{prop:reducedDLT}
holds.

We start the construction setting $U_1 := (X,D)_{\reg}$, and taking for
$\gamma_1$ the identity map. Observing that $(X,D)$ is plt at all points of
$\supp D \setminus U_1$, we can cover $\supp D \setminus U_1$ by finitely many
affine Zariski-open subsets $U_2, \ldots, U_k$ such that the following holds
for all indices $i$,
\begin{itemize}
\item the pairs $(U_i, D)$ are plt, and
\item there are numbers $m_i > 0$ and isomorphisms $\sO_{U_i}\bigl(
  m_i(K_X+ D)\bigr) \simeq \sO_{U_i}$.
\end{itemize}
Let $\gamma_i : V_i \to U_i$ be the associated index-one covers, which are finite
cyclic Galois covers that are étale in codimension one. Set $\Delta_i :=
\gamma_i^*D$. Since discrepancies do not increase under this kind of covers, see
\cite[Prop. 5.20(3)]{KM98}, the pairs $(V_i, \Delta_i)$ are again plt, so
\change{the}
discrepancies \change{of all exceptional divisors} are greater tha\change{n} $-1$.
Better still, since the log-canonical divisors $K_{V_i} + \Delta_i$ are Cartier by
construction, \change{these} discrepancies are integral, and therefore non-negative.
The reduced pairs $(V_i, \Delta_i)$ are thus canonical. In this setup,
Proposition~\ref{prop:reducedCanonical} applies to show that there exists a subset
$Z' \subset X$ of $\codim_X Z' \geq 3$ such that all pairs $\bigl( V_i \setminus
\gamma^{-1}(Z'),\, \Delta_i \setminus \gamma^{-1}(Z') \bigr)$ are snc.  Removing the
subset $Z'$ from $X$, we obtain the desired covering.

\subsubsection{Proof of Corollary~\ref*{cor:dltisQuot}, covering most of $X^\circ$}

Let $Z'' \subset X'$ be the subset of codimension $3$ that is discussed in
Proposition~\ref{prop:kltarequot}. Removing from $X$ the closure of $Z''$, the
existence of the covering follows from the assertion that $X'$ has quotient
singularities of the form described in Proposition~\ref{prop:kltarequot} and
therefore $\gamma_{\alpha}$ is totally branched over the singular set.

\section{Relative differential sequences on dlt pairs}
\label{sec:relReflSeq}

\noindent
In this section we start the systematic study of sheaves of reflexive differentials
on dlt pairs. Specifically we construct a standard exact sequence for forms of degree
$1$ with respect to a morphism $\phi: X \to T$ and study the induced filtration for
forms of degree $p \geq 2$.

\subsection{The relative differential sequence for snc pairs}
\label{sec:relReflSeqSNC}

Here we recall the generalisation of the standard sequence for relative
differentials, \cite[Prop.~II.8.11]{Ha77}, to the logarithmic setup. Let
$(X,D)$ be a reduced snc pair, and $\phi: X \to T$ an snc morphism of $(X,D)$,
as introduced in Definition~\ref{def:sncMorphism}. In this setting, the
standard pull-back morphism of 1-forms extends to yield the following exact
sequence of locally free sheaves on $X$,
\begin{sequation}\label{eq:relDiff}
  0 \to \phi^*\Omega^1_T \to \Omega^1_X(\log D) \to \Omega^1_{X/T}(\log D) \to 0,
\end{sequation}
called the ``relative differential sequence for logarithmic differentials''.
We refer to \cite[Sect.~4.1]{EV90} or \cite[Sect.~3.3]{Deligne70} for a more
detailed explanation. For forms of higher degrees, the sequence
\eqref{eq:relDiff} induces filtrations
\begin{sequation}\label{eq:relDiffFilt}
  \Omega^{p}_{X}(\log D) = \sF^{0}(\log) \supseteq \sF^1(\log) \supseteq \dots
  \supseteq \sF^{p}(\log) \supseteq \{0\}
\end{sequation}
with quotients
\begin{sequation}\label{eq:relDiffQ}
  0 \to \sF^{r+1}(\log) \to \sF^r(\log) \to \phi^*\Omega_T^r\otimes
  \Omega_{X/T}^{p-r}(\log D) \to 0
\end{sequation}
for all $r$. We refer to \cite[Ex.~II.5.16]{Ha77} for the construction
of~\eqref{eq:relDiffFilt}. For the reader's convenience, we recall without
proof of the following elementary properties of the relative differential
sequence.

\begin{fact}[Composition with étale morphisms]\label{fact:filtrationPullback}
  Let $(X,D)$ be a reduced snc pair, and let $\phi: X \to T$ be an snc
  morphism of $(X,D)$. If $\gamma: Z \to X$ is an étale morphism, and $\Delta
  := \gamma^{*}(D)$, then $\psi := \phi \circ \gamma$ is an snc morphism of
  $(Z,\Delta)$, the natural pull-back map $d\gamma: \gamma^*\Omega^1_X(\log D)
  \to \Omega_Z^1(\log \Delta)$ is isomorphic, and induces isomorphisms between
  the pull-back of the filtration~\eqref{eq:relDiffFilt} induced by $\phi$,
  and the filtration $\wtilde \sF^r(\log)$ of $\Omega^p_Z(\log \Delta)$
  induced by the composition $\psi$,
  $$
  d\gamma \bigl( \gamma^* \sF^{r}(\log) \bigr) = \widetilde \sF^r(\log), \quad
  \forall\, r.
  $$
\end{fact}

\begin{fact}[Compatibility with fiber-preserving groups actions]\label{fact:groupactionfiltrations1}
  Let $G$ be a finite group which acts on $X$, with associated isomorphisms
  $\phi_g : X \to X$. Assume in addition that the $G$-action is fibre
  preserving, i.e., assume that $\phi \circ \phi_g = \phi$ for every $g \in
  G$. Then all sheaves that appear in Sequences~\eqref{eq:relDiff} and
  \eqref{eq:relDiffQ} as well as in the filtration in~\eqref{eq:relDiffFilt}
  can naturally be endowed with $G$-sheaf structures. All the morphisms
  discussed above preserve this additional structure, i.e., they are morphisms
  of $G$-sheaves in the sense of Definition~\ref{def:Gsheaf}.
\end{fact}

\subsection{Main result of this section}
\label{ssec:111main}

The main result of this section, Theorem~\ref{thm:relativedifferentialfiltration},
gives analogues of~\eqref{eq:relDiff}--\eqref{eq:relDiffQ} in case where $(X,D)$ is
dlt. \change{In the absolute case Theorem~\ref{thm:relativedifferentialfiltration} essentially says
  that all properties of the differential sequence discussed in
  Section~\ref{sec:relReflSeqSNC} still hold on a dlt pair $(X, D)$ if one removes
  from $X$ a set $Z$ of codimension at least $3$.}

\begin{thm}[Relative differential sequence on dlt
  pairs]\label{thm:relativedifferentialfiltration}
  Let $(X, D)$ be a dlt pair with $X$ connected. Let $\phi: X \to T$ be a
  surjective morphism to a normal variety $T$.  Then, there exists a non-empty
  smooth open subset $T^\circ \subseteq T$ with preimages $X^\circ =
  \phi^{-1}(T^ \circ)$, $D^\circ = D \cap X^\circ $, and a filtration
  \begin{equation}\label{eq:relDiffFilt2}
    \Omega^{[p]}_{X^\circ}(\log \lfloor D^{\circ} \rfloor) = \sF^{[0]}(\log)
    \supseteq \dots \supseteq \sF^{[p]}(\log)\supseteq \{0\}
  \end{equation}
  on $X^\circ$ with the following properties.
  \begin{enumerate-cont}
  \item\label{il:_A} The filtrations \eqref{eq:relDiffFilt} and
    \eqref{eq:relDiffFilt2} agree wherever the pair $(X^\circ, \lfloor D^\circ
    \rfloor)$ is snc, and $\phi$ is an snc morphism of $(X^\circ, \lfloor
    D^\circ \rfloor)$.

  \item\label{il:_B} For any $r$, the sheaf $\sF^{[r]}(\log)$ is reflexive,
    and $\sF^{[r+1]}(\log)$ is a saturated subsheaf of $\sF^{[r]}(\log)$.

  \item\label{seq:quotsF} For any $r$, there exists a sequence of sheaves of
    $\sO_{X^\circ}$-modules,
    $$
    0 \to \sF^{[r+1]}(\log) \to \sF^{[r]}(\log) \to \phi^*\Omega_{T^{\circ}}^r \otimes
    \Omega_{X^\circ/T^\circ}^{[p-r]}(\log \lfloor D^\circ \rfloor) \to 0,
    $$
    which is exact and analytically locally split in codimension $2$.

  \item\label{lastisom} There exists an isomorphism $\sF^{[p]}(\log) \simeq
    \phi^*\Omega^p_{T^\circ}$.
  \end{enumerate-cont}
\end{thm}

\begin{subrem}\label{srem:XXa}
  To construct the filtration in~\eqref{eq:relDiffFilt2}, one takes the
  filtration~\eqref{eq:relDiffFilt} which exists on the open set $X \setminus
  X_{\sing}$ wherever the morphism $\phi$ is snc, and extends the sheaves to
  reflexive sheaves that are defined on all of $X$. It is then not very
  difficult to show that the sequences
  (\ref{thm:relativedifferentialfiltration}.\ref{seq:quotsF}) are exact and
  locally split away from a subset $Z \subset X$ of codimension $\codim_X Z
  \geq 2$. The main point of Theorem~\ref{thm:relativedifferentialfiltration}
  is, however, that it suffices to remove from $X$ a set of codimension
  $\codim_X Z \geq 3$.
\end{subrem}

Before proving Theorem~\ref{thm:relativedifferentialfiltration} in
Section~\ref{ssec:rdfp} below, we first draw an important corollary. The
assertion that
Sequences~(\ref{thm:relativedifferentialfiltration}.\ref{seq:quotsF}) are
exact and locally split away from a set of codimension \emph{three} plays a
pivotal role here.

\begin{cor}[Restriction of the relative differentials sequence to boundary components]\label{cor:rrdsbc}
  In the setup of Theorem~\ref{thm:relativedifferentialfiltration}, assume
  that $\lfloor D \rfloor \not = 0$ and let $D_0 \subseteq \supp \lfloor D
  \rfloor$ be any irreducible component that dominates $T$.  Recall that $D_0$
  is normal, \cite[Cor.~5.52]{KM98}.

  If $r$ is any number, then
  Sequences~(\ref{thm:relativedifferentialfiltration}.\ref{seq:quotsF}) induce
  exact sequences of reflexive sheaves on $D_0^\circ := D_0 \cap X^\circ$, as
  follows\footnote{\label{foot:brevity}For brevity of notation, we write
    $\sF^{[r]}(\log)|_{D_0}^{**}$ and $\phi^*\Omega_T^r \otimes
    \Omega_{X/T}^{[p-r]}(\log D_0)|_{D_0}^{**}$ instead of the more correct
    forms $(\sF^{[r]}(\log)|_{D_0})^{**}$ and $\phi^*\Omega_T^r|_{D_0}
    \otimes_{\sO_{D_0}} \bigl(\Omega_{X/T}^{[p-r]}(\log
    D_0)|_{D_0}\bigr)^{**}$ here and throughout.}
  \begin{equation}\label{eq:xxx}
    0 \to \sF^{[r+1]}(\log)|^{**}_{D^\circ_0} \to
    \sF^{[r]}(\log)|^{**}_{D^\circ_0} \to \phi^*\Omega_T^r \otimes
    \Omega_{X^\circ/T^\circ}^{[p-r]}(\log \lfloor D^\circ \rfloor)|^{**}_{D^\circ_0}.
  \end{equation}
\end{cor}
\begin{proof}
  Since $D_0$ is normal, there exists a subset $Z \subset X^\circ$ with
  $\codim_{X^\circ} Z \geq 3$ such that
  \begin{itemize}
  \item the divisor $D_0^\circ := D_0 \cap X^\circ$ is smooth away from $Z$, and
  \item the
    Sequences~(\ref{thm:relativedifferentialfiltration}.\ref{seq:quotsF}) are
    exact and locally split away from $Z$.
  \end{itemize}
  It follows from the local splitting of
  (\ref{thm:relativedifferentialfiltration}.\ref{seq:quotsF}) that the
  sequence obtained by restriction,
  $$
  0 \to \sF^{[r+1]}(\log)|_{D_0^\circ \setminus Z} \to
  \sF^{[r]}(\log)|_{D_0^\circ \setminus Z} \to \phi^*\Omega_T^r \otimes
  \Omega_{X^\circ/T^\circ}^{[p-r]}(\log \lfloor D^\circ \rfloor)|_{D_0^\circ
    \setminus Z} \to 0,
  $$
  is still exact. The exactness of~\eqref{eq:xxx} follows when one recalls
  that the functor which maps a sheaf to its double dual can be expressed in
  terms of a push-forward map and is therefore exact on the left.
\end{proof}

\subsection{Proof of Theorem~\ref*{thm:relativedifferentialfiltration}}
\label{ssec:rdfp}

We prove Theorem~\ref{thm:relativedifferentialfiltration} in the remainder of
Section~\ref{sec:relReflSeq}.

\subsubsection{Proof of Theorem~\ref*{thm:relativedifferentialfiltration}, setup and  start of proof}

By Remark~\ref{rem:genSNC} we are allowed to make the following assumption
without loss of generality.

\begin{awlog}\label{awlog:aXa}
  The divisor $D \cap (X, D)_{\reg}$ is relatively snc over $T$. In particular,
  \change{$T$ is smooth,} and the restriction of $\phi$ to the smooth locus $X_{\reg}$ of $X$ is
  a smooth morphism.
\end{awlog}

As we have seen in Section~\ref{sec:relReflSeqSNC}, the morphism $\phi: X \to
T$ induces on the open set $(X, D)_{\reg} \subseteq X$ a filtration of
$\Omega^p_{(X, D)_{\reg}}(\log \lfloor D \rfloor)$ by locally free saturated
subsheaves, say $\sF_\circ^{\change{r}}(\log)$. Let $i: (X, D)_{\reg} \to X$ be the
inclusion map and set
$$
\sF^{[r]}(\log) := i_* \bigl( \sF_\circ^r(\log) \bigr).
$$
We will then obtain a filtration as in~\eqref{eq:relDiffFilt2}. Notice that
all sheaves $\sF^{[r]}(\log)$ are saturated in $\Omega^{[p]}_X(\log \lfloor D
\rfloor)$ since $\sF_\circ^r(\log)$ is saturated in $\Omega^p_{(X,
  D)_{\reg}}(\log \lfloor D \rfloor)$, cf.~\cite[Lem.~1.1.16]{OSS}. This shows
the properties (\ref{thm:relativedifferentialfiltration}.\ref{il:_A}) and
(\ref{thm:relativedifferentialfiltration}.\ref{il:_B}).

Using that push-forward is a left-exact functor, we also obtain exact
sequences of reflexive sheaves on $X$ as follows,
\begin{equation}\label{eq:reflexivefiltrationsequence}
  0 \to \sF^{[r+1]}(\log) \to \sF^{[r]}(\log) \to \phi^*\Omega^r_T \otimes
  \Omega_{X/T}^{[p-r]}(\log \lfloor D \rfloor).
\end{equation}
We have to check that \eqref{eq:reflexivefiltrationsequence} is also right exact and
locally split in codimension $2$, in the analytic topology. For this we will compare
the sheaves just defined with certain $G$-invariant push-forward sheaves of local
index-one covers. Once this is shown, the property
(\ref{thm:relativedifferentialfiltration}.\ref{lastisom}) will follow automatically.

\subsubsection{Proof of Theorem~\ref*{thm:relativedifferentialfiltration}, simplifications}\label{subsubsect:reflfiltsimplifications}

We use the description of the local structure of dlt pairs in codimension $2$,
done in Chapter~\ref{sect:dltlocalstructure}, to simplify our situation.

The assertion of Theorem~\ref{thm:relativedifferentialfiltration} is
local. Since the sheaves $\sF^{[\bullet]}(\log)$ are reflexive, and since we
only claim right-exactness of \eqref{eq:reflexivefiltrationsequence} in
codimension $2$, we are allowed to remove subsets of codimension greater than
or equal to $3$ in $X$.  We will use this observation to make a number of
reduction steps.

Recall from Proposition~\ref{prop:QfactofDLT} that $X$ is $\mathbb{Q}$-factorial in
codimension $2$, and hence the pair $(X, \lfloor D \rfloor )$ is dlt in codimension
$2$, see \cite[Cor.~2.39 (1)]{KM98}. This justifies the following.

\begin{awlog}
  The variety $X$ is $\mathbb{Q}$-factorial, and the boundary divisor $D$ is
  reduced, that is, $D = \lfloor D \rfloor$.
\end{awlog}

Corollary~\ref{cor:dltisQuot} allows us to assume the following.

\begin{awlog}
  There exists a cover $X = \cup_{\alpha \in A} U_\alpha$ by open subsets and
  there are finite morphisms $\gamma_\alpha : V_\alpha \to U_\alpha$, as
  described in Corollary~\ref{cor:dltisQuot}.
\end{awlog}

\subsubsection{Proof of Theorem~\ref*{thm:relativedifferentialfiltration}, study of
  composed morphisms}

In Section~\ref{ssec:pfddrext}, we study the sequence
\eqref{eq:reflexivefiltrationsequence} by pulling it back to the smooth spaces
$V_\alpha$, and by discussing relative differential sequences associated with
the compositions $\psi_\alpha := \phi \circ \gamma_\alpha$. We will show
in this section that we may assume without loss of generality that these maps
are snc morphisms of the pairs $\bigl( V_\alpha, \gamma_\alpha^*D \bigr)$.

Shrinking $T$, if necessary, and removing from $X$ a further set of
codimension $3$, the following will hold.

\begin{awlog}\label{awlog:1a}
  The singular locus $X_{\sing}$ (with its reduced structure) is smooth, and
  so is the restriction $\phi|_{X_{\sing}}$.
\end{awlog}

\begin{awlog}\label{awlog:1b}
  If $\alpha \in A$ is one of the finitely many indices for which $U_\alpha
  \cap \supp D \not = \emptyset$, then the composition $\psi_\alpha := \phi
  \circ \gamma_\alpha$ is an snc morphism of the pair $\bigl( V_\alpha,
  \gamma_\alpha^*D \bigr)$.
\end{awlog}

As a matter of fact, Assumptions~\ref{awlog:aXa} and \ref{awlog:1a} guarantee
that all pairs $\bigl( V_\alpha, \gamma_\alpha^*D \bigr)$ are relatively snc
over $T$, not just for those indices $\alpha \in A$ where $U_\alpha$
intersects $\supp D$:

\begin{claim}\label{claim:schoen}
  If $\alpha \in A$ is any index, then the composition $\psi_\alpha := \phi
  \circ \gamma_\alpha$ is an snc morphism of the pair $\bigl( V_\alpha,
  \gamma_\alpha^*D \bigr)$.
\end{claim}
\begin{proof}
  Let $\alpha \in A$. If $U_\alpha \cap \supp D \not = \emptyset$, then
  Claim~\ref{claim:schoen} follows directly from Assumption~\ref{awlog:1b},
  and there is nothing to show. Otherwise, we have $\gamma_\alpha^*D =
  0$. Claim~\ref{claim:schoen} will follow once we show that $\psi_\alpha :
  V_\alpha \to T$ has maximal rank at all points $v \in V_\alpha$. We consider
  the cases where $\gamma_\alpha(v)$ is a smooth, (resp.~singular) point of
  $X$ separately.

  If $\gamma_\alpha(v)$ is a smooth point of $X$, then
  (\ref{cor:dltisQuot}.\ref{il:xxy}) of Corollary~\ref{cor:dltisQuot} asserts that
  $\gamma_\alpha$ is étale at $v$. Near $v$, the morphism $\psi_\alpha$ is thus a
  composition of an étale and a smooth map, and therefore of maximal rank.

  If $\gamma_\alpha(v)$ is a singular point of $X$, consider the preimage $\Sigma :=
  \gamma_\alpha^{-1}(X_{\sing})$ with its reduced structure, and observe that $v \in
  \Sigma$. In this setting, (\ref{cor:dltisQuot}.\ref{il:xxy}) of
  Corollary~\ref{cor:dltisQuot} asserts that $\gamma_\alpha$ is totally branched
  along $\Sigma$. In particular, the restriction $\gamma_\alpha|_{\Sigma} : \Sigma
  \to X_{\sing}$ is isomorphic and thus of maximal rank. By
  Assumption~\ref{awlog:1a}, the restriction $\psi_\alpha|_{\Sigma} : \Sigma \to T$
  is thus a composition of two morphisms with maximal rank, and has therefore maximal
  rank itself. It follows that $\psi_\alpha : V_\alpha \to T$ has maximal rank at
  $v$.
\end{proof}

Right-exactness of the sequence~\eqref{eq:reflexivefiltrationsequence} and its
local splitting are properties that can be checked locally in the analytic
topology on the open subsets $U_\alpha$. To simplify notation, we replace $X$
by one of the $U_\alpha$. Claim~\ref{claim:schoen} and Additional
Assumption~\ref{awlog:1b} then allow to assume the following.

\begin{awlog}\label{awlog:quotient}
  There exists a smooth manifold $Z$, endowed with an action of a finite group
  $G$ and associated quotient map $\gamma: Z \to X$. The cycle-theoretic
  preimage $\Delta:= \gamma^{*}(D)$ is a reduced snc divisor. Furthermore, the
  quotient map $\gamma$ is étale in codimension one, and the composition of
  $\psi := \phi \circ \gamma: Z \to T$ is an snc morphism of the pair $(Z,
  \Delta)$.
\end{awlog}

\subsubsection{Proof of Theorem~\ref{thm:relativedifferentialfiltration},
  right-exactness of \eqref{eq:reflexivefiltrationsequence}}
\label{ssec:pfddrext}

Since $\psi$ is a $G$-invariant snc morphism between of the pair $(Z,
\Delta)$, Fact~\ref{fact:groupactionfiltrations1} yields a filtration of
$\Omega_{Z}^p (\log \Delta)$ by locally free $G$-subsheaves $\widetilde
\sF^{r}(\log)$ and $G$-equivariant exact sequences,
\begin{equation}\label{eq:exactsequencefiltrationoncover}
  0 \to \widetilde \sF^{r+1}(\log) \to \widetilde \sF^r(\log) \to
  \psi^*\Omega_T^r\otimes \Omega_{Z/T}^{p-r}(\log \Delta ) \to 0.
\end{equation}
By the Reflexivity Lemma~\ref{lemma:reflexivepushforward} the $G$-invariant
push-forward-sheaves $\gamma_*\widetilde \sF^{r}(\log)^G$ are then
reflexive. By the Exactness Lemma~\ref{lem:invariantsexact} these reflexive
sheaves fit into the following exact sequences
\begin{equation}\label{eq:push-forwardexactI}
  0 \to \gamma_*\widetilde\sF^{r+1}(\log)^G \to \gamma_*
  \widetilde \sF^{r}(\log)^G \to \gamma_* \bigl( \psi^* \Omega^r_T
  \otimes \Omega_{Z/T}^{p-r} (\log \Delta ) \bigr)^G \to 0.
\end{equation}
Since $\gamma$ is étale in codimension one, Fact~\ref{fact:filtrationPullback}
implies that the differential $d\gamma$ induces isomorphisms
\begin{equation}\label{eq:isoI}
  \sF^{[r]}(\log) \xrightarrow{\simeq} \gamma_* \widetilde\sF^{r}(\log)^G.
\end{equation}
Furthermore, since $\psi = \phi \circ \gamma$, since $\Omega^r_T$ is locally
free, and since $G$ acts trivially on $T$, it follows from the projection
formula that there exist isomorphisms
\begin{align}
  \label{eq:isoII-1} \phi^*\Omega^r_T \otimes \Omega_{X/T}^{[p-r]}(\log D ) &
  \xrightarrow{\simeq} \phi^*\Omega^r_T \otimes
  \gamma_*  \Omega_{Z/T}^{p-r}(\log D) ^G \\
  \label{eq:isoII-2} & \xrightarrow{\simeq} \gamma_* \bigl( \psi^*\Omega^r_T
  \otimes \Omega_{Z/T}^{p-r}(\log D) \bigr)^G.
\end{align}
In summary, we note that the isomorphisms \eqref{eq:isoI}--\eqref{eq:isoII-2}
make the following diagram commutative:
$$
\xymatrix{%
  \gamma_*  \wtilde \sF^{r+1}(\log) ^G \ar@{^{(}->}[r] & \gamma_*
  \wtilde \sF^{r}(\log) ^G\ar@{->>}[r] & \gamma_* \bigl( \psi^*
  \Omega^r_T \otimes \Omega_{Z/T}^{p-r} (\log \Delta) \bigr)^G \\
  \sF^{[r+1]}(\log) \ar[r]\ar[u]^\simeq & \sF^{[r]}(\log) \ar[r]\ar[u]^\simeq &
  \phi^*\Omega^r_T \otimes \Omega_{X/T}^{[p-r]}(\log)\ar[u]^\simeq .}
$$
This shows that \eqref{eq:reflexivefiltrationsequence} is also right-exact, as
claimed in (\ref{thm:relativedifferentialfiltration}.\ref{seq:quotsF}).

\subsubsection{Proof of Theorem~\ref{thm:relativedifferentialfiltration}, existence of local analytic splittings}

It remains to show that \eqref{eq:reflexivefiltrationsequence} admits local
analytic splittings in codimension $2$. This follows directly from the
Splitting Lemma~\ref{lem:splittings}, concluding the proof of
Theorem~\ref{thm:relativedifferentialfiltration}. \qed

\section{Residue sequences for reflexive differential forms}

A very important feature of logarithmic differentials is the existence of a
residue map.  In its simplest form consider a smooth hypersurface $D \subset
X$ in a manifold $X$. The residue map is then the cokernel map in the exact
sequence
$$
0 \to \Omega^1_X \to \Omega^1_X(\log D) \to \mathcal O_D \to 0.
$$
In Section~\ref{ssect:sncresiduesequence}, we first recall the general
situation for an snc pair, for forms of arbitrary degree and in a relative
setting. A generalisation to dlt pairs is the established in
Sections~\ref{ssec:12B}--\ref{ssec:12C} below. Without the dlt assumption,
residue maps fail to exist in general.

\subsection{Residue sequences for snc pairs}
\label{ssect:sncresiduesequence}

Let $(X, D)$ be a reduced snc pair. Let $D_0 \subseteq D$ be an irreducible component
and recall from \cite[2.3(b)]{EV92} that there exists a residue sequence,
$$
\xymatrix{0 \to \Omega^p_X(\log (D - D_0)) \ar[r] &
  \Omega^p_X(\log D) \ar[r]^(.4){\rho^p} &
  \Omega^{p-1}_{D_0}(\log D_0^c) \to 0,
}
$$
where $D_0^c := (D-D_0)|_{D_0}$ denotes the ``restricted complement'' of
$D_0$.  More generally, if $\phi: X \to T$ is an snc morphism of $(X,D)$ we
have a relative residue sequence
\begin{sequation}\label{eq:stdResidue}
  \xymatrix{0 \to \Omega^p_{X/T}(\log(D - D_0)) \ar[r] & \Omega^p_{X/T}(\log D)
    \ar[r]^(.4){\rho^p} & \Omega^{p-1}_{D_0/T}(\log D_0^c) \to 0.}
\end{sequation}
The sequence~\eqref{eq:stdResidue} is not a sequence of locally free sheaves
on $X$, and its restriction to $D_0$ will never be exact on the left.
However, an elementary argument, cf.~\cite[Lem.~2.13.2]{KK08}, shows that
restriction of~\eqref{eq:stdResidue} to $D_0$ induces the following exact
sequence
\begin{sequation}\label{eq:restrictedsncresidue}
  0 \to \Omega_{D_0/T}^p(\log D_0^c) \xrightarrow{i^p} \Omega^p_{X/T}(\log
  D)|_{D_0} \xrightarrow{\rho^p_D} \Omega_{D_0}^{p-1}(\log D_0^c) \to 0,
\end{sequation}
which is very useful for inductive purposes.  We recall without proof several
elementary facts about the residue sequence.

\begin{fact}[Residue map as a test for logarithmic poles]\label{fact:poletest}
  If $\sigma \in H^0\bigl( X,\, \Omega^{p}_{X/T}(\log D) \bigr)$ is any
  reflexive form, then $\sigma \in H^0\bigl( X,\, \Omega^{p}_{X/T}(\log
  (D-D_0)) \bigr)$ if and only if $\rho^{p}(\sigma) = 0$.
\end{fact}

\begin{fact}\label{fact:Q}
  In the simple case where $T$ is a point, $p = 1$ and $D = D_0$, the restricted
  residue sequence \eqref{eq:restrictedsncresidue} reads
  $$
  0 \to \Omega_D^1 \xrightarrow{i^1} \Omega^1_X(\log
  D)|_D \xrightarrow{\rho^1_D} \sO_D \to 0.
  $$
  The sheaf morphisms $i^1$ and $\rho^1_D$ are then described as follows.  If
  $V \subseteq X$ is any open set, and if $f \in \sO_X(V)$ is a function that
  vanishes only along $D$, then
  \begin{equation}\label{eq:QA}
    \rho_D^1\bigl( (d \log f)|D \bigr)  =
    \ord_D f \cdot \mathbf{1}_{D \cap V},
  \end{equation}
  where $\mathbf{1}_{D \cap V}$ is the constant function with value one. If $g
  \in \sO_X(V)$ is any function, then
  \begin{equation}\label{eq:QB}
    i^1\bigl( d (g|_{D \cap V}) \bigr)  = (d g)|_{D \cap V}.
  \end{equation}
\end{fact}

\begin{fact}[Base change property of the residue map]\label{fact:baseChange}
  Let $(X,D)$ be a reduced snc pair, and $\pi : \wtilde X \to X$ a surjective
  morphism such that the pair $(\wtilde X, \wtilde D)$ is snc, where $\wtilde
  D := \supp \pi^*D$. If $\wtilde D_0 \subset \pi^{-1}(D_0)$ is any
  irreducible component, then there exists a diagram
  $$
  \xymatrix{ \pi^* \bigl( \Omega^p_X(\log D) \bigr) \ar[rr]^(.4){\pi^*(\rho^p)}
    \ar[d]_{d\pi} & & \pi^* \bigl( \Omega^{p-1}_{D_0}(\log D_0^c) \bigr)
    \ar[d]^{d(\pi|_{\wtilde D_0})}
    \\
    \Omega^p_{\wtilde X}(\log \wtilde D) \ar[rr]^(.4){\rho^p} &&
    \Omega^{p-1}_{\wtilde D_0}(\log \wtilde D_0^c).}
  $$
\end{fact}

\begin{fact}[Compatibility with fiber-preserving groups actions]\label{fact:groupactionfiltrations2}
  Let $G$ be a finite group which acts on $X$, with associated isomorphisms
  $\phi_g : X \to X$. Assume that the $G$-action stabilises both the divisor
  $D$, and the component $D_0 \subseteq D$, and assume that the action is
  fibre preserving, that is $\phi \circ \phi_g = \phi$ for every $g \in
  G$. Then all sheaves that appear in Sequences~\eqref{eq:stdResidue} and
  \eqref{eq:restrictedsncresidue} are $G$-sheaves, in the sense of
  Definition~\vref{def:Gsheaf}, and all morphisms that appear in
  \eqref{eq:stdResidue} and \eqref{eq:restrictedsncresidue} are morphisms of
  $G$-sheaves.
\end{fact}

\subsection{Main result of this section}
\label{ssec:12B}

If the pair $(X, D)$ is not snc, no residue map exists in general.  However,
if $(X, D)$ is dlt, then \cite[Cor.~5.52]{KM98} applies to show that $D_0$ is
normal, and an analogue of the residue map $\rho^p$ exists for sheaves of
reflexive differentials, as we will show now.

To illustrate the problem we are dealing with, consider a normal space $X$
that contains a smooth Weil divisor $D = D_0$\PreprintAndPublication{, similar
  to the one sketched in Figure~\vref{fig:sct}}. One can easily construct
examples where the singular set $Z := X_{\sing}$ is contained in $D$ and has
codimension $2$ in $X$, but codimension one in $D$. In this setting, a
reflexive form $\sigma \in H^0 \bigl( D_0,\, \Omega^{[p]}_{D_0}(\log D_0)
|_{D_0} \bigr)$ is simply the restriction of a form defined outside of $Z$,
and the form $\rho^{[p]}(\sigma)$ is the extension of the well-defined form
$\rho^{p}(\sigma|_{D_0 \setminus Z})$ over $Z$, as a rational form with poles
along $Z \subset D_0$. If the singularities of $X$ are bad, it will generally
happen that the extension $\rho^{[p]}(\sigma)$ has poles of arbitrarily high
order.  Theorem~\ref{thm:relativereflexiveresidue} asserts that this does not
happen when $(X,D)$ is dlt.

\begin{thm}[Residue sequences for dlt pairs]\label{thm:relativereflexiveresidue}
  Let $(X, D)$ be a dlt pair with $\lfloor D \rfloor \neq \emptyset$ and let
  $D_0 \subseteq \lfloor D \rfloor$ be an irreducible component. Let $\phi: X
  \to T$ be a surjective morphism to a normal variety $T$ such that the
  restricted map $\phi|_{D_0}: D_0 \to T$ is still surjective. Then, there
  exists a non-empty open subset $T^\circ \subseteq T$, such that the
  following holds if we denote the preimages as $X^\circ = \phi^{-1}(T^
  \circ)$, $D^\circ = D \cap X^\circ $, and the ``complement'' of $D_0^\circ$
  as $D_0^{\circ,c} := \bigl( \lfloor D^\circ \rfloor - D_0^\circ \bigr)
  |_{D_0^\circ}$.
  \begin{enumerate}
  \item\label{il:RelReflResidue} There exists a sequence
    \begin{multline*}
      \quad \quad \quad \quad 0 \to \Omega^{[r]}_{X^\circ/T^\circ}(\log
      (\lfloor D^\circ\rfloor - D_0^\circ )) \to
      \Omega_{X^\circ/T^\circ}^{[r]}( \log \lfloor D^\circ
      \rfloor ) \\
      \xrightarrow{\rho^{[r]}} \Omega^{[r-1]}_{D^\circ_0/T^\circ}(\log
      D_0^{\circ,c} ) \to 0
    \end{multline*}
    which is exact in $X^\circ$ outside a set of codimension at least $3$. This
    sequence coincides with the usual residue sequence \eqref{eq:stdResidue} wherever
    the pair $(X^\circ, D^\circ)$ is snc and the map $\phi^\circ: X^\circ \to
    T^\circ$ is an snc morphism of $(X^\circ, D^\circ)$.

  \item\label{il:RestrRelReflResidue} The restriction of
    Sequence~(\ref{thm:relativereflexiveresidue}.\ref{il:RelReflResidue}) to
    $D_0$ induces a sequence
    \begin{multline*}
      \quad \quad \quad \quad 0 \to \Omega^{[r]}_{D^\circ_0/T^\circ}(\log
      D_0^{\circ,c} ) \to \Omega_{X^\circ/T^\circ}^{[r]}(\log \lfloor D^\circ
      \rfloor )|_{D_0^\circ} ^{**} \\ %
      \xrightarrow{\rho^{[r]}_{D^\circ_0}}
      \Omega^{[r-1]}_{D^\circ_0/T^\circ}(\log D_0^{\circ,c} ) \to 0
    \end{multline*}
    which is exact on $D_0^\circ$ outside a set of codimension at least $2$ and
    coincides with the usual restricted residue sequence
    \eqref{eq:restrictedsncresidue} wherever the pair $(X^\circ, D^\circ)$ is snc and
    the map $\phi^\circ: X^\circ \to T^\circ$ is an snc morphism of $(X^\circ,
    D^\circ)$.
  \end{enumerate}
\end{thm}

Fact~\ref{fact:poletest} and Theorem~\ref{thm:relativereflexiveresidue}
together immediately imply that the residue map for reflexive differentials
can be used to check if a reflexive form has logarithmic poles along a given
boundary divisor.

\begin{rem}[Residue map as a test for logarithmic poles]\label{rem:poletest}
  In the setting of Theorem~\ref{thm:relativereflexiveresidue}, if $\sigma \in
  H^0\bigl( X,\, \Omega^{[p]}_X(\log \lfloor D \rfloor) \bigr)$ is any
  reflexive form, then $\sigma \in H^0\bigl( X,\, \Omega^{[p]}_X(\log \lfloor
  D \rfloor -D_0) \bigr)$ if and only if $\rho^{[p]}(\sigma) = 0$.
\end{rem}

\subsection{Proof of Theorem~\ref{thm:relativereflexiveresidue}}
\label{ssec:12C}

We prove Theorem~\ref{thm:relativereflexiveresidue} in the remainder of the present
chapter. As in the setup of Theorem~\ref{thm:relativedifferentialfiltration},
discussed in Remark~\ref{srem:XXa}, it is not difficult to construct
Sequences~(\ref{thm:relativereflexiveresidue}.\ref{il:RelReflResidue}) and
(\ref{thm:relativereflexiveresidue}.\ref{il:RestrRelReflResidue}) and to prove
exactness outside a set of codimension $2$, but the main point is the exactness
outside a set of codimension at least $3$.

\subsubsection{Proof of Theorem~\ref{thm:relativereflexiveresidue}, simplifications}

Again, as in Section~\ref{subsubsect:reflfiltsimplifications} we use the
description of the codimension $2$ structure of dlt pairs, obtained in
Chapter~\ref{sect:dltlocalstructure}, to simplify our situation. Since all the
sheaves appearing in
Sequences~(\ref{thm:relativereflexiveresidue}.\ref{il:RelReflResidue}) and
(\ref{thm:relativereflexiveresidue}.\ref{il:RestrRelReflResidue}) are
reflexive, it suffices to construct the sheaf morphism $\rho^{[r]}$ outside a
set of codimension at least $3$.  Notice also that existence and exactness
of (\ref{thm:relativereflexiveresidue}.\ref{il:RelReflResidue}) and
(\ref{thm:relativereflexiveresidue}.\ref{il:RestrRelReflResidue}) are clear at
all points in $(X, D)_{\text{reg}}$ where $\phi$ is an snc morphism of $(X,
D)$. We will use these two observations to make a number of reduction steps.

As in Section~\ref{subsubsect:reflfiltsimplifications}, removing from $X$ a
set of codimension $3$, we may assume the following without loss of
generality.

\begin{awlog}
  The variety $X$ is $\mathbb{Q}$-factorial, and the boundary divisor $D$ is
  reduced, that is, $D = \lfloor D \rfloor$.
\end{awlog}

Since the target of the residue map is a sheaf supported on $D$, we may work
locally in a neighbourhood of $D$.  Removing a further set of codimension more
than $2$, Corollary~\ref{cor:dltisQuot} therefore allows to assume the
following.

\begin{awlog}\label{awlog:XXA}
  There exists a cover $X = \cup_{\alpha \in A} U_\alpha$ by a finite number
  of affine Zariski open subsets $U_\alpha$ of $X$, and there exist finite
  Galois covers $\gamma_\alpha: V_\alpha \to U_\alpha$, étale in codimension
  one, such that the pairs $\bigl( V_\alpha, \gamma_\alpha^* D \bigr)$ are snc
  for all indices $\alpha$.
\end{awlog}

Observe that the construction of the desired map $\rho^{[p]}$ can be done on
the open subsets $U_\alpha$, once we have established the claim that the local
maps constructed on the $U_\alpha$ coincide with the usual residue maps
wherever this makes sense. To simplify notation, we will hence replace $X$ by
one of the $U_\alpha$, writing $\gamma := \gamma_\alpha$, $Z := U_\alpha$ and
$\Delta := \gamma^*D$. The Galois group of $\gamma$ will be denoted by
$G$. Shrinking $T$ if necessary, we may suppose the following.

\begin{awlog}\label{awlog:smoothonsmooth}
  The restriction of $\phi$ to the snc locus $(X, D)_{\text{reg}}$ is an snc morphism
  of $(X, D)$. The composition $\psi := \phi \circ \gamma$ is an snc morphism of $(Z,
  \Delta)$.
\end{awlog}

With Assumption~\ref{awlog:smoothonsmooth} in place, and the assertion of
Theorem~\ref{thm:relativereflexiveresidue} being clear near points where $(X,
D)$ is snc, the description of the codimension $2$ structure of dlt pairs
along the boundary, Proposition~\ref{prop:reducedDLT}, allows us to assume the
following.

\begin{awlog}\label{awlog:plt}
  The pair $(X, D)$ is plt. The divisors $D \subset X$ and $\Delta \subset Z$
  are smooth and irreducible. In particular, we have $D = D_0$, $\lfloor D
  \rfloor - D_0 = 0$, $D_0^c = 0$, and the restricted maps $\psi|_\Delta:
  \Delta \to T$ and $\phi|_D: D \to T$ are smooth morphisms of smooth
  varieties.
\end{awlog}

\subsubsection{Proof of Theorem~\ref*{thm:relativereflexiveresidue}, construction and
  exactness of (\ref*{thm:relativereflexiveresidue}.\ref*{il:RelReflResidue})}

Since $\psi: Z \to T$ is an snc morphism of $(Z, \Delta)$, and since the irreducible
divisor $\Delta \subset Z$ is invariant under the action of $G$,
Fact~\ref{fact:groupactionfiltrations2} and the standard residue sequence
\eqref{eq:stdResidue} yield an exact sequence of morphisms of $G$-sheaves, as follows
$$
0 \to \Omega^{r}_{Z/T} \to \Omega_{Z/T}^{r}( \log \Delta )
\overset{\rho^{r}}{\longrightarrow} \Omega^{r-1}_{\widetilde D/T} \to 0.
$$
Recalling from Lemma~\ref{lem:invariantsexact} that $\gamma_*(\cdot)^G$ is an
exact functor, this induces an exact sequence of morphisms of $G$-sheaves, for
the trivial $G$-action on $X$,
\begin{equation}\label{eq:pushforwardexact}
  0 \to \gamma_*(\Omega^{r}_{Z/T})^G \to \gamma_*(\Omega_{Z/T}^{r}( \log \Delta
  ))^G \xrightarrow{\gamma_*(\rho^{r})^G} \gamma_*(\Omega^{r-1}_{\Delta/T})^G \to 0.
\end{equation}
Recall from Lemma~\ref{lemma:reflexivepushforward} that all the sheaves
appearing in \eqref{eq:pushforwardexact} are reflexive. The fact that $\gamma$
is étale in codimension one then implies that the pull-back of reflexive forms
via $\gamma$ induces isomorphisms
\begin{align}
  \Omega^{[r]}_{X/T}&\overset{\simeq}{\longrightarrow}\gamma_*
  (\Omega^{r}_{Z/T})^G \quad \quad \text{and}\label{eq:isomI}\\
  \Omega_{X/T}^{[r]}( \log D ) &\overset{\simeq}{\longrightarrow}
  \gamma_*(\Omega_{Z/T}^{r}(\log \Delta ))^G. \label{eq:isomII}
\end{align}

It remains to describe the last term of \eqref{eq:pushforwardexact}.

\begin{claim}\label{claim:isomIII}
  The restriction of $\gamma$ to $\Delta$ induces an isomorphism
  $\gamma_*(\Omega^{r-1}_{\Delta/T})^G \simeq \Omega^{r-1}_{D/T}$.
\end{claim}
\begin{proof}
  By Assumption~\ref{awlog:plt}, the restricted morphism $\gamma|_{\Delta} :
  \Delta \to D$ is a finite morphism of smooth spaces. The branch locus $S
  \subset D$ and the ramification locus $\wtilde S \subset \Delta$ are
  therefore both of pure codimension one.

  The pull-back map of differential forms associated with $\gamma|_{\Delta}$
  yields an injection $\Omega^{r-1}_{D/T} \hookrightarrow
  \gamma_*(\Omega^{r-1}_{\Delta/T})^G$. To prove Claim~\ref{claim:isomIII}, it
  remains to show surjectivity. To this end, recall from
  Assumption~\ref{awlog:XXA} that $D$ and $\Delta$ are affine, and let $\sigma
  \in H^0\bigl(\Delta,\ \Omega_{\Delta}^{r-1} \bigr)^G $ be any $G$-invariant
  $(r-1)$-form on $\Delta$. Then there exists a rational differential form
  $\tau$ on $D$, possibly with poles along the divisor $S \subset D$
  satisfying the relation
  \begin{equation}\label{eq:relTAU}
    (\gamma|_{\Delta})^*(\tau)|_{\Delta \setminus \wtilde S} = \sigma|_{\Delta
      \setminus \wtilde S}.
  \end{equation}
  Recalling that regularity of differential forms can be checked on any
  finite cover, \cite[Cor.~2.12.ii]{GKK08}, Equation~\eqref{eq:relTAU} implies
  that $\tau$ is in fact a regular form on $D$, that is, $\tau \in
  H^0\bigl(D,\ \Omega_{D}^{r-1} \bigr)$ with $(\gamma|_{\Delta})^*(\tau)
  = \sigma$. This finishes the proof of Claim~\ref{claim:isomIII}.
\end{proof}

Finally, using the isomorphisms~\eqref{eq:isomI}, \eqref{eq:isomII} and
Claim~\ref{claim:isomIII} established above, Sequence
\eqref{eq:pushforwardexact} translates into
\begin{equation}\label{eq:XXB}
  0 \to \Omega^{[r]}_{X/T} \to  \Omega_{X/T}^{[r]}( \log  D )
  \overset{\rho^{[r]}}{\longrightarrow} \Omega^{r-1}_{D/T} \to 0,
\end{equation}
which is the sequence whose existence is asserted in
(\ref{thm:relativereflexiveresidue}.\ref{il:RelReflResidue}). Using
Fact~\ref{fact:baseChange} and using that the finite covering $\gamma$ is
étale away from the singular locus of $(X, D)$, it follows by construction
that the map $\rho^{[r]}$ coincides with the usual relative residue map
wherever the pair $(X, D)$ is snc.

\subsubsection{Proof of Theorem~\ref{thm:relativereflexiveresidue}, construction and
  exactness of (\ref{thm:relativereflexiveresidue}.\ref{il:RestrRelReflResidue})}

Restricting the morphism $\rho^{[r]}$ of the sequence~\eqref{eq:XXB} to the
smooth variety $D \subset X$, and recalling that restriction is a right-exact
functor, we obtain a surjection
\begin{equation}\label{eq:XXC}
  \rho^{[r]}|_D : \Omega_{X/T}^{[r]}(\log D )|_{D} \to
  \Omega^{r-1}_{D/T} \to 0.
\end{equation}
Since any sheaf morphism to a reflexive sheaf factors via the reflexive hull of the
domain, \eqref{eq:XXC} induces a surjective map between reflexive hulls, and
therefore an exact sequence
\begin{equation}\label{eq:XXD}
  0 \to \ker\bigl(\rho^{[r]}_D\bigr) \to
  \Omega_{X/T}^{[r]}(\log D )|_D^{**} \xrightarrow{\rho^{[r]}_D}
  \Omega^{r-1}_{D/T} \to 0.
\end{equation}
Comparing (\ref{thm:relativereflexiveresidue}.\ref{il:RestrRelReflResidue})
and \eqref{eq:XXD}, we see that to finish the proof of
Theorem~\ref{thm:relativereflexiveresidue}, we need to show that
$$
\ker\bigl(\rho^{[r]}_D\bigr) \simeq \Omega^{[r]}_{D/T}.
$$
To this end, we consider the standard restricted residue sequence
\eqref{eq:restrictedsncresidue} for the morphism $\psi$, and its $G$-invariant
push-forward,
\begin{equation}\label{eq:XXE}
  0 \to \underbrace{\gamma_*\bigl( \Omega_{\Delta/T}^r
    \bigr)^G}_{\makebox[0pt][r]{\scriptsize $\simeq
      \Omega^{r}_{D/T}$ by Claim~\ref{claim:isomIII}}} \to
  \gamma_*\bigl(\Omega^r_{Z/T}(\log \Delta)|_{\Delta} \bigr)^G \to
  \underbrace{\gamma_*\bigl(\Omega_{\Delta/T}^{r-1}
    \bigr)^G}_{\makebox[0pt][r]{\scriptsize $\simeq
      \Omega^{r-1}_{D/T}$ by Claim~\ref{claim:isomIII}}} \to 0.
\end{equation}
By Lemma~\ref{lem:invariantsexact} from Appendix B, this sequence is exact.
In order to describe the middle term of \eqref{eq:XXE} and to relate
\eqref{eq:XXE} to \eqref{eq:XXD}, observe that the Restriction
Lemma~\ref{lem:surjection} together with the isomorphism \eqref{eq:isomII}
yields a surjective sheaf morphism
$$
\varphi: \Omega_{X/T}^{[r]}(\log D )|_{D}^{**} \twoheadrightarrow
\gamma_*(\Omega_{Z/T}^r(\log \Delta)|_{\Delta})^G.
$$
Since $\gamma$ is étale in codimension one, it is étale at the general point
of $\Delta$, and hence $\varphi$ is generically an isomorphism. Consequently
$\varphi$ is an isomorphism as $ \Omega_{X/T}^{[r]}(\log D)|_{D}^{**} $ is
torsion-free. Additionally, it follows from Fact~\ref{fact:baseChange} that
the map $\varphi$ fits into the following commutative diagram with exact rows,
$$
\begin{xymatrix}{ %
    0 \ar[r]& \ker\bigl(\rho^{[r]}_D\bigr) \ar[r] \ar[d]_{\theta} & %
    \Omega_{X/T}^{[r]}(\log D)|_{D}^{**} \ar[r]^(.6){\rho^{[r]}_D}
    \ar[d]_{\varphi}^{\simeq} &
    \Omega^{r-1}_{D/T} \ar[r]\ar[d]_{\simeq} & 0 \\
    0 \ar[r]& \Omega^{r}_{ D/T} \ar[r]& \gamma_*(\Omega_{Z/T}^{r}( \log \Delta
    )|_{\Delta})^G \ar[r]& \Omega^{r-1}_{D/T} \ar[r]& 0. }
\end{xymatrix}
$$
This shows that $\theta$ is an isomorphism, and completes the proof of
Theorem~\ref{thm:relativereflexiveresidue}.  \qed

\section{The residue map for $1$-forms}
\label{sec:part2-last}

Let $X$ be a smooth variety and $D \subset X$ a smooth, irreducible
divisor. The first residue sequence~\eqref{eq:stdResidue} of the pair $(X,D)$
then reads
$$
0 \to \Omega^1_D \to \Omega^1_X(\log D)|_D \xrightarrow{\rho^1} \sO_D \to 0,
$$
and we obtain a connecting morphism of the long exact cohomology sequence,
$$
\delta : H^0 \bigl( D, \sO_D \bigr) \to H^1 \bigl( D, \Omega^1_D \bigr).
$$
In this setting, the standard description of the first Chern class in terms of
the connecting morphism, \cite[III.~Ex.~7.4]{Ha77}, asserts that
\begin{sequation}\label{eq:c1descr}
  c_1\bigl( \sO_X(D)|_D \bigr) =
  \delta(\mathbf{1}_D) \in H^1 \bigl( D, \Omega^1_D \bigr),
\end{sequation}
where $\mathbf{1}_D$ is the constant function on $D$ with value one.

\subsection{Main result of this section}

Theorem~\ref{thm:Chernclass} generalises the Identity~\eqref{eq:c1descr} to
the case where $(X,D)$ is a reduced dlt pair with irreducible boundary
divisor.

\begin{thm}\label{thm:Chernclass}
  Let $(X, D)$ be a dlt pair, $D = \lfloor D \rfloor$ irreducible.  Then,
  there exists a closed subset $Z \subset X$ with $\codim_{X}Z \geq 3$ and a
  number $m \in \mathbb{N}$ such that $mD$ is Cartier on $X^\circ := X
  \setminus Z$, such that $D^\circ := D \cap X^\circ$ is smooth, and such that
  the restricted residue sequence
  \begin{equation}\label{eq:simplerestrictedreflexiveresidue}
    0 \to \Omega_D^1 \to \Omega_X^{[1]}(\log D)|_D^{**}
    \overset{\rho_D}{\longrightarrow} \sO_{D} \to 0
  \end{equation}
  defined in Theorem~\ref{thm:relativereflexiveresidue} is exact on
  $D^\circ$. Moreover, for the connecting homomorphism $\delta$ in the
  associated long exact cohomology sequence
  $$
  \delta : H^0 \bigl(D^\circ,\, \sO_{D^\circ} \bigr) \to
  H^1\bigl(D^\circ,\, \Omega_{D^\circ}^1\bigr)
  $$
  we have
  \begin{equation}
    \delta(m\cdot \mathbf{1}_{D^\circ} ) =
    c_1 \left(\left. \sO_{X^\circ}(mD^\circ)\right|_{D^\circ}\right).
  \end{equation}
\end{thm}

\subsection{Proof of Theorem~\ref*{thm:Chernclass}}

Using Propositions~\ref{prop:QfactofDLT}, \ref{prop:reducedDLT} and
Theorem~\ref{thm:relativereflexiveresidue} to remove from $X$ a suitable subset of
codimension $3$, we may assume that the following holds.

\begin{awlog}
  The divisor $D$ is smooth. The variety $X$ is $\mathbb{Q}$-factorial, so that there
  exists a number $m$ such that $mD$ is Cartier. The restricted residue
  sequence,
  \begin{equation}\label{eq:SRR2}
    0 \to \Omega^1_D \xrightarrow{i^1} \Omega_X^{[1]}(\log
    D)|_D^{**} \xrightarrow{\rho_D} \sO_D \to 0,
  \end{equation}
  is exact.
\end{awlog}

Let $X^{\circ\circ} = (X,D)_{\reg}$ be the snc locus of $(X,D)$, and set
$D^{\circ\circ} = D \cap X^{\circ\circ}$.

\subsubsection{\u{C}ech-cocycles describing the line bundle $\sO_X(mD)$ and its Chern
  class}

Since $mD$ is Cartier, there exists a covering of $D$ by open affine subsets
$(U_\alpha)_{\alpha \in I}$ and there are functions $f_\alpha \in
\sO_X(U_\alpha)$ cutting out the divisors $m D|_{U_\alpha}$, for all $\alpha
\in A$.

Setting $g_{\alpha\beta} := f_\alpha/f_\beta \in H^0\bigl(U_\alpha \cap
U_\beta,\ \sO^*_{U_\alpha \cap U_\beta} \bigr)$, the line bundle $\sO_X(mD)|_D
\in \Pic(D) = H^1\bigl( D,\, \sO_D^*\bigr)$ is represented by the
\u{C}ech-cocycle
$$
(g_{\alpha \beta}|_D)_{\alpha,\beta} \in \text{\u{C}}^1\bigl(\{U_\alpha \cap
D\}_{\alpha \in I}, \sO_D^*\bigr).
$$
In particular, the first Chern class $c_1( \sO_X(mD)|_D) \in H^1\bigl( D,\,
\Omega^1_D \bigr)$ is represented by the \u{C}ech-cocycle
\begin{equation}\label{eq:ccher}
(d \log (g_{\alpha \beta}|_D))_{\alpha,\beta} \in \text{\u{C}}^1\bigl(\{U_\alpha \cap
D\}_{\alpha \in I}, \Omega^1_D \bigr).
\end{equation}

\subsubsection{Computation of the connecting morphism, completion of the proof}

We finish the proof of Theorem~\ref{thm:Chernclass} with an explicit computation of
the connecting morphism. The following claim will prove to be crucial.

\begin{claim}\label{claim:chern1}
  For any index $\alpha$, consider the Kähler differential $d\log f_\alpha \in
  H^0\bigl( U_\alpha,\, \Omega^1_X(\log D) \bigr)$, with associated section
  $$
  \sigma_\alpha \in H^0 \bigl( U_\alpha \cap D,\, \Omega^{[1]}_X(\log
  D)|_D^{**} \bigr).
  $$
  Then $\rho_D(\sigma_\alpha) = m\cdot \mathbf{1}_{D \cap U_{\alpha}}$.
\end{claim}
\begin{proof}
  Given an index $\alpha$,
  Claim~\ref{claim:chern1} needs only to be checked on the open set $U_\alpha
  \cap D^{\circ\circ} \subseteq U_\alpha \cap D$. There, it follows
  from Equation~\eqref{eq:QA} of Fact~\ref{fact:Q}.
\end{proof}

\begin{claim}\label{claim:chern2}
  For any indices $\alpha$, $\beta$, consider the Kähler differential
  $$
  \tau_{\alpha\beta} := d\log (g_{\alpha \beta}|_D)
  \in H^0\bigl( U_\alpha \cap U_\beta \cap D,\, \Omega^1_D \bigr).
  $$
  Then $i^1(\tau_{\alpha\beta}) = \sigma_\alpha - \sigma_\beta$.
\end{claim}
\begin{proof}
  Given any two indices $\alpha$, $\beta$, Claim~\ref{claim:chern2} needs only
  to be checked on $U_{\alpha} \cap U_{\beta} \cap D^{\circ\circ}$. There, we
  have
  \begin{align*}
    i^1\bigl(d \log(g_{\alpha\beta}|_{D^{\circ\circ}}) \bigr) &=
    \frac{1}{g_{\alpha \beta}|_{D^{\circ\circ}}} i^1 \bigl( d(g_{\alpha
      \beta}|_{D^{\circ\circ}}) \bigr) = \left. \frac{1}{g_{\alpha
          \beta}}d(g_{\alpha \beta}) \right|_{D^{\circ\circ}}  \\
    &= \bigl( d \log g_{\alpha \beta} \bigr)|_{D^{\circ\circ}} = \bigl(d \log
    f_\alpha - d \log f_\beta\bigr)|_{D^{\circ\circ}} = \bigl(\sigma_\alpha -
    \sigma_\beta\bigr)|_{D^{\circ\circ}},
  \end{align*}
  the second equality coming from Equation~\eqref{eq:QB} of
  Fact~\ref{fact:Q}, proving Claim~\ref{claim:chern2}.
\end{proof}

As an immediate consequence of Claim~\ref{claim:chern2}, we obtain that
$\delta(m\cdot \mathbf{1}_{D}) \in H^1 \bigl( D,\, \Omega^1_D
\bigr)$ is represented by the \u{C}ech-cocycle
$$
\tau_{\alpha\beta} \in \text{\u{C}}^1\bigl(\{U_\alpha \cap D\}_{\alpha \in I},
\Omega^1_D \bigr).
$$
Since $\tau_{\alpha\beta} = d \log (g_{\alpha\beta}|_D)$, a
comparison with the \u{C}ech-cocycle that describes $c_1( \sO_X(mD)|_D)$, as
given in \eqref{eq:ccher}, then finishes the proof of
Theorem~\ref{thm:Chernclass}.

\part{COHOMOLOGICAL METHODS}
\label{part:3}

\section{Vanishing results for pairs of Du~Bois spaces}

In this section we prove a vanishing theorem for reduced pairs $(X,D)$ where
both $X$ and $D$ are Du~Bois. A vanishing theorem for ideal sheaves on log
canonical pairs (that are not necessarily reduced) will follow.  Du~Bois
singularities are defined via Deligne's Hodge theory.  We will briefly recall
Du~Bois's construction of the generalised de~Rham complex, which is called the
\emph{Deligne-Du~Bois complex}.  Recall, that if $X$ is a smooth complex
algebraic variety of dimension $n$, then the sheaves of differential $p$-forms
with the usual exterior differentiation give a resolution of the constant
sheaf $\bC_X$. I.e., one has a complex of sheaves,
$$
\xymatrix{%
\sO_X \ar[r]^{d} & \Omega_X^1 \ar[r]^{d} & \Omega_X^2 \ar[r]^{d} & \Omega_X^3
\ar[r]^{d} & \dots \ar[r]^(.35){d} & \Omega_X^n\simeq \omega_X,
}
$$
which is quasi-isomorphic to the constant sheaf $\bC_X$ via the natural map
$\bC_X\to \sO_X$ given by considering constants as holomorphic functions on
$X$. Recall that this complex \emph{is not} a complex of quasi-coherent
sheaves. The sheaves in the complex are quasi-coherent, but the maps between
them are not $\sO_X$-module morphisms. Notice however that this is actually
not a shortcoming; as $\bC_X$ is not a quasi-coherent sheaf, one cannot expect
a resolution of it in the category of quasi-coherent sheaves.

The Deligne-Du~Bois complex is a generalisation of the de~Rham complex to
singular varieties.  It is a filtered complex of sheaves on $X$ that is
quasi-isomorphic to the constant sheaf, $\bC_X$. The terms of this complex are
harder to describe but its properties, especially cohomological properties are
very similar to the de~Rham complex of smooth varieties. In fact, for a smooth
variety the Deligne-Du~Bois complex is quasi-isomorphic to the de~Rham
complex, so it is indeed a direct generalisation.

The construction of this complex, $\FullDuBois{X}$, is based on simplicial
resolutions. The reader interested in the details is referred to the original
article \cite{DuBois81}.  Note also that a simplified construction was later
obtained in \cite{Carlson85} and \cite{GNPP88} via the general theory of
polyhedral and cubic resolutions.  An easily accessible introduction can be
found in \cite{Steenbrink85}.  Other useful references are the recent book
\cite{PetersSteenbrinkBook} and the survey \cite{Kovacs-Schwede09}.  The word
``hyperresolution'' will refer to either a simplicial, polyhedral, or cubic
resolution. Formally, the construction of $\FullDuBois{X}$ is the same
regardless the type of resolution used and no specific aspects of either types
will be used.  We will actually not use these resolutions here. They are
needed for the construction, but if one is willing to believe the basic
properties then one should be able follow the material presented here.

The bare minimum we need is that there exists a filtered complex
$\FullDuBois{X}$ unique up to quasi-isomorphism satisfying a number of
properties.  As a filtered complex, it admits an associated graded complex,
which we denote by $Gr^{p}_{\rm filt}\, \FullDuBois{X}$. In order to make the
formulas work the way they do in the smooth case we need to make a shift. We
will actually prefer to use the following notation:
$$
\Ox p\leteq Gr^{p}_{\rm filt}\, \FullDuBois{X}[p],
$$
Here ``$[p]$'' means that the $m^\text{th}$ object of the complex $\Ox p$ is
defined to be the $(m+p)^\text{th}$ object of the complex $Gr^{p}_{\rm filt}\,
\FullDuBois{X}$. In other words, these complexes are almost the same, only one
is a shifted version of the other. They naturally live in the filtered derived
category of $\sO_X$-modules with differentials of order $\leq 1$.  For an
extensive list of their properties see \cite{DuBois81} or
\cite[4.2]{Kovacs-Schwede09}. Here we will only recall a few of them.

One of the most important characteristics of the Deligne-Du~Bois complex is
the existence of a natural morphism in the derived category $\sO_{X} \to
\Om^0_X$, cf.~\cite[4.1]{DuBois81}. We will be interested in situations where
this map is a quasi-isomorphism. If this is the case and if in addition $X$ is
proper over $\bC$, the degeneration of the Frölicher spectral sequence at
$E_1$, cf.~\cite[4.5]{DuBois81} or \cite[4.2.4]{Kovacs-Schwede09}, implies
that the natural map
\begin{equation*}
  H^i(X^{\rm an}, \bC) \rightarrow H^i(X, \sO_{X}) = \bH^i(X, \DuBois{X})
\end{equation*}
is surjective. Here $\bH^i$ stands for hypercohomology of complexes, i.e.,
$\bH^i=R^i\Gamma$.

\begin{defn}\label{def:db-sing}
  A scheme $X$ is said to have \emph{Du~Bois singularities} (or \emph{DB
    singularities} for short) if the natural map $\sO_{X}\to \Om^0_X$ is a
  quasi-isomorphism.
\end{defn}

\begin{ex}\label{ex:deligne}
  It is easy to see that smooth points are Du~Bois. Deligne proved that normal
  crossing singularities are Du~Bois as well cf.~\cite[Lem.~2(b)]{MR0376678}.
\end{ex}

We are now ready to state and prove our vanishing results for pairs of Du~Bois
spaces. While we will only use
Corollary~\ref{cor:idealsheafvanishingwithboundary} in this paper, we believe
that these vanishing results are interesting on their own. For instance, based
on these observations one may argue that a pair of Du~Bois spaces is not too
far from a space with rational singularities. Indeed, if $X$ has rational
singularities and $D=\emptyset$, then the result of
Theorem~\ref{thm:idealsheafvanishingwithboundary} follows directly from the
definition of rational singularities. Of course, Du~Bois singularities are not
necessarily rational and hence one cannot expect vanishing theorems for the
higher direct images of the structure sheaf, but our result says that there
are vanishing results for ideal sheaves of Du~Bois subspaces.

\begin{thm}[Vanishing for ideal sheaves on pairs of Du~Bois spaces]\label{thm:idealsheafvanishingwithboundary}
  Let $(X, D)$ be a reduced pair such that $X$ and $D$ are both Du~Bois, and
  let $\pi : \widetilde X \to X$ be a \wlr of $(X,D)$ with $\pi$-exceptional
  set $E$. If we set $\wtilde D := \supp \bigl( E + \pi^{-1}(D) \bigr)$, then
  $$
  R^{i}\pi_* \sO_{\wt X}(-{\widetilde D}) = 0 \text{\qquad for all \, $i > \max
    \bigl(\dim \overline{\pi(E) \setminus D}, 0 \bigr)$.}
  $$
  In particular, if $X$ is of dimension $n\geq 2$, then $R^{n-1}\pi_* \sO_{\wt
    X}(-{\widetilde D}) = 0$.
\end{thm}

\begin{cor}[Vanishing for ideal sheaves on log canonical pairs]\label{cor:idealsheafvanishingwithboundary}
  Let $(X, D)$ be a log canonical pair of dimension $n \geq 2$. Let $\pi :
  \widetilde X \to X$ be a \wlr of $(X,D)$ with $\pi$-exceptional set $E$. If
  we set $\widetilde D := \supp \bigl( E + \pi^{-1} \lfloor D \rfloor \bigr)$,
  then
  $$
  R^{n-1}\pi_* \, \sO_{\widetilde X}(- \widetilde D) = 0.
  $$
\end{cor}
\begin{proof}
  Recall from \cite[Theorem~1.4]{KKLogCanonicalDuBois} that $X$ is Du~Bois,
  and that any finite union of log canonical centres is likewise
  Du~Bois. Since the components of $\lfloor D \rfloor$ are log canonical
  centres, Theorem~\ref{thm:idealsheafvanishingwithboundary} applies to the
  reduced pair $\bigl( X, \lfloor D \rfloor \bigr)$ to prove the claim. For
  this, recall from Lemma~\ref{lem:logresforsmalldiv} that the morphism $\pi$
  is a \wlr of the pair $\bigl( X, \lfloor D \rfloor \bigr)$ and therefore
  satisfies all the conditions listed in
  Theorem~\ref{thm:idealsheafvanishingwithboundary}.
\end{proof}

\subsection{Preparation for the proof of Theorem~\ref*{thm:idealsheafvanishingwithboundary}}

Before we give the proof of Theorem~\ref{thm:idealsheafvanishingwithboundary}
in Section~\ref{ssec:THMISVBDRY}, we need the following auxiliary result. This
generalises parts of \cite[III.1.17]{GNPP88}.

\begin{lem}\label{lem:top-coh-vanishes}
  Let $X$ be a positive dimensional variety. Then the $i^{\text{th}}$
  cohomology sheaf of $\ul{\Omega}_X^0$ vanishes for all $i\geq \dim X$, i.e.,
  $h^i(\ul{\Omega}_X^0)=0$ for all $i\geq \dim X$.
\end{lem}

\begin{proof}
  For $i>\dim X$, the statement follows from \cite[III.1.17]{GNPP88}, so we
  only need to prove the case when $i=n:=\dim X$.  Let $S:=\Sing X$ and $\pi :
  \widetilde X \to X$ a \slr with exceptional divisor $E$.  Recall from
  \cite[3.2]{DuBois81} that there are natural restriction maps,
  $\ul{\Omega}_X^0 \to \ul{\Omega}_S^0$ and $\ul{\Omega}_{\wt X}^0 \to
  \ul{\Omega}_E^0$ that reduce to the usual restriction of regular functions
  if the spaces are Du~Bois. These maps are connected via an an exact triangle
  by \cite[Prop.~4.11]{DuBois81}:
  \begin{equation}\label{eq:exacttriangle}
    \xymatrix{%
      \ul{\Omega}_X^0 \ar[r] & \ul{\Omega}_S^0 \oplus R\pi_*\ul{\Omega}_{\wt X}^0
      \ar[r]^-\alpha & R\pi_*\ul{\Omega}_{E}^0 \ar^-{+1}[r] & }.
  \end{equation}
  Since $\wt X$ is smooth and $E$ is an snc divisor, they are both Du Bois,
  cf.~ Example~\ref{ex:deligne}. Hence, there exist quasi-isomorphisms
  $\ul{\Omega}_{\wt X}^0\simeq \sO_{\wt X}$ and $\ul{\Omega}_{E}^0\simeq
  \sO_{E}$.  It follows that $\alpha(0,{\_})$ is the map $R\pi_*\sO_{\wt X}
  \to R\pi_*\sO_{E}$ induced by the short exact sequence
  $$
  0\to \sO_{\wt X}(-E)\to \sO_{\wt X}\to \sO_E\to 0.
  $$
  Next, consider the long exact sequence of cohomology sheaves induced by the
  exact triangle~\eqref{eq:exacttriangle},
  $$
  \dots\to h^{n-1}(\ul{\Omega}_{S}^0)\oplus R^{n-1}\pi_*\sO_{\wt X}\overset{\
    \alpha^{n-1}}\longrightarrow R^{n-1}\pi_*\sO_E\to h^{n}(\ul{\Omega}_{X}^0)
  \to h^{n}(\ul{\Omega}_{S}^0)\oplus R^{n}\pi_*\sO_{\wt X}.
  $$
  Since $\dim S<n$, \cite[III.1.17]{GNPP88} implies that
  $h^{n}(\ul{\Omega}_{S}^0)=0$. Furthermore, as $\pi$ is birational, the
  dimension of any fibre of $\pi$ is at most $n-1$ and hence
  $R^{n}\pi_*\sO_{\wt X}=0$.  This implies that
  $h^{n}(\ul{\Omega}_{X}^0)\simeq \coker\alpha^{n-1}$.  The bound on the
  dimension of the fibres of $\pi$ also implies that $R^{n}\pi_*\sO_{\wt
    X}(-E)=0$, so taking into account the observation above about the map
  $(\alpha,0)$, we obtain that $\alpha^{n-1}(0,\_ )$ is surjective, and then
  naturally so is $\alpha^{n-1}$.  Therefore, $h^{n}(\ul{\Omega}_{X}^0)\simeq
  \coker\alpha^{n-1}=0$.
\end{proof}

\subsection{Proof of Theorem~\ref*{thm:idealsheafvanishingwithboundary}}
\label{ssec:THMISVBDRY}

Since the divisor $D$ is assumed to be reduced, we simplify notation in this
proof and use the symbol $D$ to denote both the divisor and its support. To
start the proof, set $\Sigma := \ol{\pi(E) \setminus D}$ and $s := \max
\bigl(\dim \Sigma, 0 \bigr)$. Let $\Gamma := D \cup \pi(E)$ and consider the
exact triangle from \cite[4.11]{DuBois81},
$$
\xymatrix{%
  \ul{\Omega}_X^0 \ar[r] & \ul{\Omega}_\Gamma^0 \oplus R\pi_*\ul{\Omega}_{\wt
    X}^0 \ar[r] & R\pi_*\ul{\Omega}_{\wt D}^0 \ar^-{+1}[r] & }.
$$
Since $\wt X$ is smooth and $\wt D$ is a snc divisor, we have
quasi-isomorphisms $R\pi_*\ul{\Omega}_{\wt X}^0\simeq R\pi_*\sO_{\wt X}$ and
$R\pi_*\ul{\Omega}_{\wt D}^0\simeq R\pi_*\sO_{\wt D}$, so this exact triangle
induces the following long exact sequence of sheaves:
$$
\dots\to h^i(\ul{\Omega}_X^0)\to h^i(\ul{\Omega}_\Gamma^0)\oplus
R^i\pi_*\sO_{\wt X} \to R^i\pi_*\sO_{\wt D} \to
h^{i+1}(\ul{\Omega}_X^0)\to\cdots
$$
By assumption $h^i(\ul{\Omega}_X^0)=h^i(\ul{\Omega}_D^0)=0$ for
$i>0$. Furthermore, $h^i(\ul{\Omega}_\Sigma^0)=0$ and
$h^{i-1}(\ul{\Omega}_{\Sigma\cap D}^0)=0$ for $i\geq s$ by
Lemma~\ref{lem:top-coh-vanishes}. Hence, $h^i(\ul{\Omega}_\Gamma^0)=0$ for
$i\geq s$ by \cite[3.8]{DuBois81}. As in the proof of
Lemma~\ref{lem:top-coh-vanishes} we obtain that the natural restriction map
$$
R^i\pi_*\sO_{\wt X} \to R^i\pi_*\sO_{\wt D}
$$
is surjective for $i\geq s$ and is an isomorphism for $i> s$. This in turn
implies that $R^i\pi_* \sO_{\wt X}(-{\widetilde D})=0$ for $i>s$ as
desired. \qed

\section{Steenbrink-type vanishing results for log canonical pairs}

The second vanishing theorem we shall need to prove the main result is
concerned with direct images of logarithmic sheaves.

\begin{thm}[Steenbrink-type vanishing for log canonical pairs]\label{thm:Omegavanishing}
  Let $(X, D)$ be a log canonical pair of dimension $n \geq 2$.  If $\pi :
  \wtilde X \to X$ is a \wlr of $(X,D)$ with $\pi$-exceptional set $E$ and
  $\wtilde D := \supp \bigl(E + \pi^{-1} \lfloor D\rfloor \bigr)$, then
  $$
  R^{n-1}\pi_*\bigl(\Omega^{p}_{\wtilde X}(\log \wtilde D) \otimes \sO_X
  (-{\wtilde D})\bigr) = 0 \quad \text{for all $0\leq p \leq n$.}
  $$
\end{thm}
\begin{subrem}
  Recall from Lemma~\ref{lem:logresforsmalldiv} that $\pi$ is also a \wlr of
  the pair $(X, \lfloor D\rfloor)$. In particular, it follows from the
  definition that $\wtilde D$ is of pure codimension one and has simple normal
  crossing support.
\end{subrem}
\begin{subrem}\label{rem:Steenbrinkvanishing}
  For $p > 1$ the claim of Theorem~\ref{thm:Omegavanishing} is proven in
  \cite[Thm.~2(b)]{Steenbrink85} without any assumption on the nature of the
  singularities of $X$. The case $p=0$ is covered by
  Theorem~\ref{thm:idealsheafvanishingwithboundary}. Hence, the crucial
  statement is the vanishing for $p=1$ in the case of log canonical
  singularities.
\end{subrem}

\begin{cor}[Steenbrink-type vanishing for cohomology with supports]\label{cor:H1rest}
  Let $(X, D)$ be a log canonical pair of dimension $n\geq 2$. Let $\pi : \wtilde X
  \to X$ be a \wlr of $(X,D)$ with $\pi$-exceptional set $E$ and set $\wtilde D :=
  \supp \bigl( E + \pi^{-1} \supp \lfloor D\rfloor \bigr)$. If $x \in X$ is any point
  with set-theoretic fibre $F_x = \pi^{-1}(x)_{\red}$, then
  $$
  H^1_{F_x}\bigl(\wtilde X,\, \Omega^p_{\wtilde X}(\log \wtilde D) \bigr) = 0
  \quad \quad \text{for all } 0 \leq p \leq n.
  $$
\end{cor}

\begin{subrem}\label{rem:H1rest}
  Using the standard exact sequence for cohomology with support,
  \cite[Ex.III.2.3(e)]{Ha77}, the conclusion of Corollary~\ref{cor:H1rest} can
  equivalently be reformulated in terms of restriction maps as follows.
  \begin{enumerate}
  \item\label{item:4.1.1} The map $H^0\bigl(\wtilde X,\, \Omega^p_{\wtilde
      X}(\log \wtilde D) \bigr) \to H^0\bigl(\wtilde X \setminus F_x,\,
    \Omega^p_{\wtilde X}(\log \wtilde D) \bigr)$ is surjective, and

  \item\label{item:4.1.2} the map $H^1\bigl(\wtilde X,\, \Omega^p_{\wtilde
      X}(\log \wtilde D) \bigr) \to H^1\bigl(\wtilde X \setminus F_x,\,
    \Omega^p_{\wtilde X}(\log \wtilde D) \bigr)$ is injective.
  \end{enumerate}
\end{subrem}

\begin{proof}[Proof of Corollary~\ref{cor:H1rest}]
  Duality for cohomology groups with support, cf.~\cite[Appendix]{GKK08},
  yields that
  $$
  H^1_{F_x}\bigl(\wtilde X,\, \Omega^p_{\wtilde X}(\log \wtilde D)\bigr)
  \overset{\textrm{dual}}{\sim} \bigl(R^{n-1}\pi_* \Omega_{\wtilde
    X}^{n-p}(\log \wtilde D)(- \wtilde D)_x\bigr)^{\widehat\ },
  $$
  where $\ ^{\widehat\ }$ denotes completion with respect to the maximal ideal
  $\mathfrak{m}_x $ of the point $x \in X$. The latter group vanishes by
  Theorem~\ref{thm:Omegavanishing}.
\end{proof}

\subsection{Preparation for the proof of Theorem~\ref*{thm:Omegavanishing}: Topological vanishing}

To prepare for the proof of Theorem~\ref*{thm:Omegavanishing}, we first
discuss the local topology of the pair $(\wtilde X, \wtilde D)$ near a fibre
of $\pi$ and derive a topological vanishing result, which is probably
well-known to experts.  Subsequently, the vanishing for coherent cohomology
groups claimed in Theorem~\ref{thm:Omegavanishing} follows from an argument
going back to Wahl \cite[§1.5]{Wahl85}.

\begin{rem}
  Note that we will work in the complex topology of $X$ and $\wtilde X$ and we
  will switch back and forth between cohomology of coherent algebraic sheaves
  and the cohomology of their analytification without further indication. This
  is justified by the relative version of Serre's GAGA results,
  cf.~\cite[Thm.~2.48]{KM98}.
\end{rem}

\begin{lem}[Topological vanishing]\label{lem:topvanishing}
  In the setup of Theorem~\ref{thm:Omegavanishing}, if $j: \wtilde X \setminus
  \wtilde D \hookrightarrow \wtilde X$ is the inclusion map, and if
  $j_!\mathbb{C}_{\wtilde X \setminus \wtilde D}$ is the sheaf that is defined
  by the short exact sequence
  \begin{equation}\label{eq:definingrelativetopology}
    \xymatrix{ 0 \ar[r] & j_!\mathbb{C}_{\wtilde X \setminus \wtilde D}
      \ar[r] & \mathbb{C}_{\wtilde X} \ar[rr]^{\text{restriction}} &&
      \mathbb{C}_{\wtilde D} \ar[r] & 0,}
  \end{equation}
  then $R^k\pi_*\bigl( j_!\mathbb{C}_{\wtilde X \setminus \wtilde D}\bigr) =
  0$ for all numbers $k$.
\end{lem}
\begin{proof}
  \change{Let
    $F_x$ denote the reduced fiber of $\pi$ over a point $x \in \mathrm{Supp} [D]
    \cup\pi(E)$}.  By \cite[Thms.~2 and 3]{Lojasiewicz64} we can find arbitrarily
  fine triangulations of $\wtilde X$ and $\wtilde D$ such that $\wtilde D$ is a
  subcomplex of the triangulation of $\wtilde X$ and such that $F_x$ is a subcomplex
  of the triangulation of $\wtilde D$. It follows that there exist arbitrarily small
  neighbourhoods $\wtilde U= \wtilde U(F_x)$ of $F_x$ in $\wtilde X$ such that the
  inclusions $F_x \hookrightarrow \wtilde{D}\cap \wtilde U \hookrightarrow \wtilde U$
  are homotopy-equivalences. Since $\pi$ is proper, preimages of small open
  neighbourhoods of $x$ in $X$ form a neighbourhood basis of the fibre $F_x$. As a
  consequence, there exist arbitrarily small neighbourhoods $U$ of $x$ in $X$ such
  that the natural morphisms
  $$
  H^k \bigl(\pi^{-1}(U),\, \mathbb{C}_{\wtilde X}|_{\pi^{-1}(U)}\bigr) \to
  H^k\bigl(\wtilde D \cap \pi^{-1}(U),\, \mathbb{C}_{\wtilde D}|_{\wtilde D
    \cap\pi^{-1}(U)}\bigr)
  $$
  are isomorphisms for all $k$. The long exact sequence derived from
  \eqref{eq:definingrelativetopology} then implies the claimed vanishing.
\end{proof}

\subsection{Proof of Theorem~\ref*{thm:Omegavanishing}}

As observed in Remark~\ref{rem:Steenbrinkvanishing}, we may assume that
$p=1$. Consequently, have to prove that $R^{n-1}\pi_*\bigl(\Omega^{1}_{\wtilde
  X}(\log \wtilde D) \otimes \sO_{\wtilde X}(-{\wtilde D}) \bigr) = 0$.  A
straightforward local computation shows that that the following sequence of
sheaves is exact,
\begin{multline}\label{eq:longsequenceforms}
  0 \to j_!\mathbb{C}_{\wtilde X \setminus \wtilde D} \to
  \sO_X(-{\wtilde D}) \xrightarrow{d} \Omega^1_{\wtilde X}(\log \wtilde
  D) \otimes \sO_{\wtilde X}(-{\wtilde D}) \xrightarrow{d} \cdots \\
  \cdots \xrightarrow{d} \Omega^{n-1}_{\wtilde X}(\log \wtilde D)
  \otimes \sO_{\wtilde X}(-{\wtilde D}) \xrightarrow{d} \Omega_{\wtilde
    X}^n \xrightarrow{d} 0,
\end{multline}
where $d$ denotes the usual exterior differential. For brevity of notation,
set $\sG_p := \Omega^p_{\wt X}(\log \wtilde D) \otimes \sO_{\wtilde
  X}(-{\wtilde D})$. In particular, set $\sG_0 := \sO_X(-{\wtilde D})$.

\begin{claim}\label{claim:RisR}
  We have $R^{n-1}\pi_*(d\sG_0) = 0$ and $R^n\pi_*(d\sG_0) = 0$.
\end{claim}
\begin{proof}
  The following short exact sequence forms the first part of the long exact
  sequence~\eqref{eq:longsequenceforms}:
  $$
  0 \to j_!\mathbb{C}_{\wtilde X \setminus \wtilde D} \to \sG_0
  \xrightarrow{d} d\sG_0 \to 0.
  $$
  Hence it follows from topological vanishing, Lemma~\ref{lem:topvanishing},
  that $R^{n-1}\pi_*(d\sG_0) \simeq R^{n-1}\pi_* \sG_0$ and $R^n\pi_*(d\sG_0)
  \simeq R^n\pi_* \sG_0$. While $R^n\pi_* \sG_0$ vanishes for dimensional
  reasons, the vanishing of $R^{n-1}\pi_* \sG_0$ follows from
  Theorem~\ref{thm:idealsheafvanishingwithboundary}. This finishes the proof
  of Claim~\ref{claim:RisR}.
\end{proof}

\begin{claim}\label{claim:SisS}
  The differential $d$ induces an isomorphism $R^{n-1}\pi_*\sG_1 \simeq
  R^{n-1}\pi_*(d\sG_1)$.
\end{claim}
\begin{proof}
  The second short exact sequence derived from \eqref{eq:longsequenceforms},
  $$
  0 \to d\sG_0 \to \sG_1 \xrightarrow{d} d\sG_1 \to 0,
  $$
  induces the following long exact sequence of higher push-forward sheaves,
  $$
  \cdots \to \underbrace{R^{n-1}\pi_*(d\sG_0)}_{=0 \text{ by Claim~\ref{claim:RisR}}}
  \to R^{n-1}\pi_* \sG_1 \xrightarrow{d} R^{n-1}\pi_*(d\sG_1) \to
  \underbrace{R^n\pi_*(d\sG_0)}_{=0 \text{ by Claim~\ref{claim:RisR}}} \to \cdots .
  $$
  Claim~\ref{claim:SisS} then follows.
\end{proof}

As a consequence of Claim~\ref{claim:SisS}, in order to prove
Theorem~\ref{thm:Omegavanishing}, it suffices to show that $R^{n-1}\pi_*(d\sG_1) =
0$. This certainly follows from the following claim.

\begin{claim}\label{claim:greuel}
  $R^{n-p}\pi_*( d\sG_p ) = 0$ for all $1 \leq p \leq n$.
\end{claim}

We prove Claim~\ref{claim:greuel} by descending induction on $p$. For $p = n$,
the claim is true since $R^{0}\pi_* \bigl(d \sG_n \bigr)$ is isomorphic to the
push-forward of the zero sheaf, and hence equals the zero sheaf. In general,
assume that Claim~\ref{claim:greuel} has been shown for all numbers that are
larger than $p$, and consider the short exact sequence
$$
0 \to d\sG_p \to \sG_{p+1} \to d\sG_{p+1} \to 0
$$
derived from \eqref{eq:longsequenceforms}. This yields a long exact sequence
\begin{equation}\label{eq:melisande}
  \cdots \to R^{n-(p+1)}\pi_* (d\sG_{p+1})\to R^{n-p}\pi_*(d\sG_p
  ) \to R^{n-p}\pi_* \sG_{p+1} \to \cdots .
\end{equation}
Observe that the first group in~\eqref{eq:melisande} vanishes by induction,
and that the last group vanishes by Steenbrink vanishing
\cite[Thm.~2(b)]{Steenbrink85}. This proves the claim and concludes the proof
of Theorem~\ref{thm:Omegavanishing}. \qed

\begin{subrem}
  Greuel proves a similar result for isolated complete intersection singularities in
  \cite{MR582515}.
\end{subrem}

\section{Generic base change for cohomology with supports}
\label{sec:BSCHCS}

In this section we provide another technical tool for the proof of the main
results: we give a local-to-global statement for cohomology groups with
support in a family of normal varieties.

\begin{thm}[Generic base change for cohomology with supports]\label{thm:BC}
  Let $\phi: X \to S$ be a surjective morphism with connected fibres between
  normal, irreducible varieties, and let $E \subset X$ be an algebraic subset
  such that the restriction $\phi|_E$ is proper. Further, let $\sF$ be a
  locally free sheaf on $X$ such that
  \begin{equation}\label{eq:chws1}
    H^1_{E_s} \bigl( X_s,\, \sF|_{X_s} \bigr) = 0 \quad\text{for all $s \in
      S$,}
  \end{equation}
  where $X_s := \phi^{-1}(s)$ and $E_s := (\phi|_E)^{-1}(s)$. Then there
  exists a non-empty Zariski-open subset $S^\circ \subseteq S$, with preimage
  $X^\circ := \phi^{-1}(S^\circ)$, such that
  \begin{equation}\label{eq:chws2}
    H^1_{E \cap X^\circ} \bigl(X^\circ,\, \sF|_{X^\circ} \bigr) = 0.
  \end{equation}
\end{thm}

\PreprintAndPublication{\begin{figure}
  \centering

  \ \\

  $$
  \xymatrix{
    \begin{picture}(4.8,4)(0,0)
      \put( 1.0, 4.2){normal variety $X$}
      \put( 1.0, 0.2){\includegraphics[height=3.5cm]{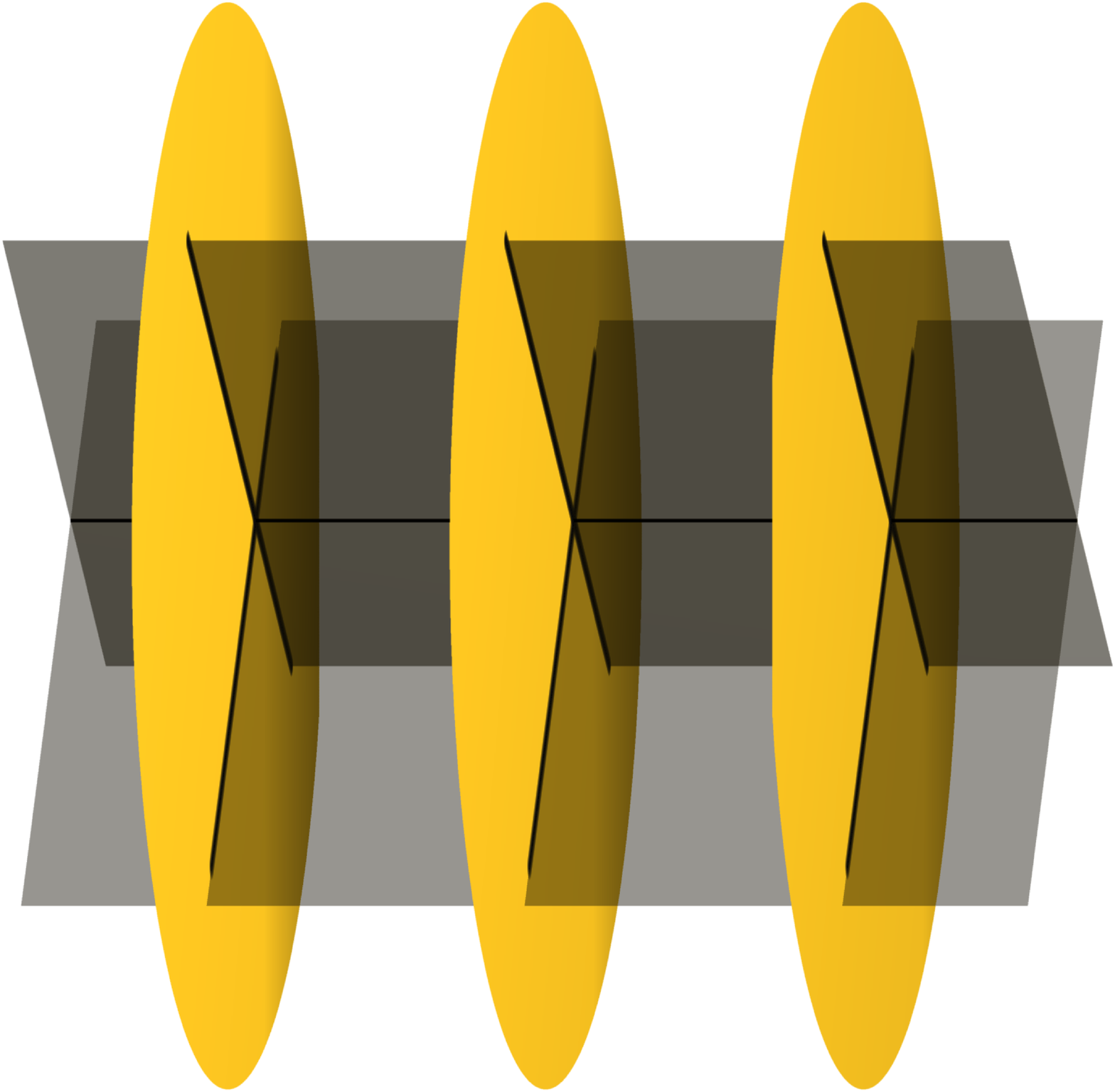}}
      \put(-0.4, 3.0){\scriptsize divisor $E_0$}
      \put( 0.7, 3.0){\vector(3, -1){0.4}}
      \put(-0.4, 1.2){\scriptsize divisor $E_1$}
      \put( 0.7, 1.2){\vector(4, -1){0.4}}
    \end{picture}
    \ar[d]_(.75){\phi} &&
    \begin{picture}(4.8,4)(0,0)
      \put( 1.0, 4.2){normal variety $X$}
      \put( 1.0, 0.2){\includegraphics[height=3.5cm]{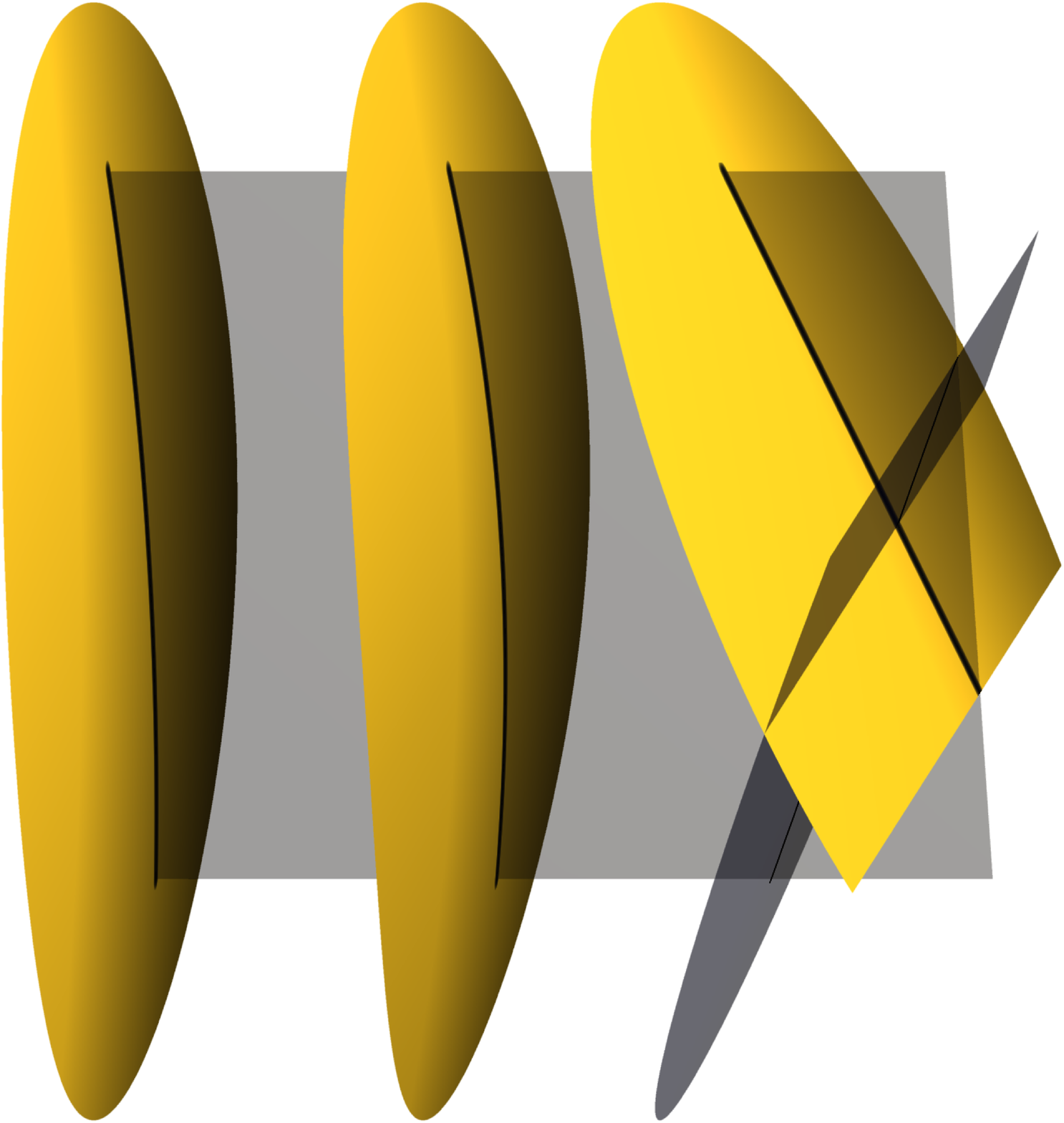}}
      \put( 4.4, 3.7){\scriptsize divisor $E_1$}
      \put( 4.4, 3.6){\vector(-1, -1){0.4}}
      \put( 3.9, 0.3){\scriptsize divisor $E_0$}
      \put( 3.8, 0.4){\vector(-2, 1){0.4}}
    \end{picture}
    \ar[d]_(.75){\phi} \\
    \begin{picture}(3.5, 0.7)(0,0)
      \put( 0.0, 0.3){$S$}
      \put( 0.0, 0.0){\includegraphics[width=3.5cm]{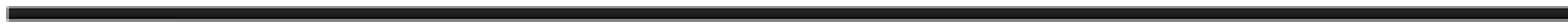}}
    \end{picture}
    &&
    \begin{picture}(3.5, 0.7)(0,0)
      \put( 0.0, 0.3){$S$}
      \put( 0.0, 0.0){\includegraphics[width=3.5cm]{family_base}}
    \end{picture}
  }
  $$

  \caption{Two morphisms for which the assumptions of Theorem~\ref*{thm:BC} hold}
  \label{fig:exci}
\end{figure}
Figure~\vref{fig:exci} illustrates the setup of Theorem~\ref{thm:BC}.}{} %
We prove Theorem~\ref{thm:BC} in the remainder of the present
Section~\ref{sec:BSCHCS}.

\subsection{Proof of Theorem~\ref*{thm:BC}: simplifications}

To start, choose a normal, relative compactification $\overline X$ of $X$,
i.e., a normal variety $\overline X$ that contains $X$ and a morphism $\Phi:
\overline{X} \to S$, such that $\Phi$ is proper and $\Phi|_X = \phi$. By
\cite[I.~Thm.~9.4.7]{EGA1} there exists a coherent extension $\overline{\sF}$
of $\sF$, i.e., a coherent sheaf $\overline \sF$ of $\sO_{\overline
  X}$-modules such that $\overline \sF|_X = \sF$.  Then excision for
cohomology with supports \cite[III~Ex.~2.3(f)]{Ha77} asserts that the
cohomology groups of \eqref{eq:chws1} and \eqref{eq:chws2} can be computed on
$\overline X$.  More precisely, if $S^\circ \subseteq S$ is a subset with
preimages $X^\circ := \phi^{-1}(S^\circ)$ and $\overline X{}^\circ :=
\Phi^{-1}(S^\circ)$, then it follows from the relative properness of $E$ that
$$
H^1_{E \cap X^\circ} \bigl( X^\circ,\, \sF \bigr) \simeq H^1_{E \cap \overline
  X^\circ} \bigl( {\overline X}{}^\circ,\, \overline \sF \bigr).
$$
As a consequence, we see that it suffices to show Theorem~\ref{thm:BC} under
the following additional assumptions.

\begin{awlog}
  The morphism $\phi$ is proper. In particular, the higher direct image
  sheaves $R^i\phi_*\sF$ are coherent sheaves of $\sO_S$-modules for all $i$.
\end{awlog}

Let $\sF_s := \sF|_{X_s}$. Using semicontinuity we can replace $S$ by a
suitable subset and assume without loss of generality to be in the following
situation.

\begin{awlog}\label{awlog:S1}
  The variety $S$ is affine, the morphism $\phi$ is flat and the a priori
  upper-semicontinuous functions $s \mapsto h^i \bigl( X_s,\, \sF_s \bigr)$
  are constant for all $i$. In particular, the higher direct image sheaves
  $R^i\phi_*\sF$ are all locally free.
\end{awlog}

The following excerpt from the standard cohomology sequence for cohomology
with support \cite[III~Ex.~2.3(e)]{Ha77}
$$
H^0\bigl(X,\,\sF\bigr) \xrightarrow{\alpha} H^0\bigl(X \setminus E,\,\sF\bigr)
\to H^1_E\bigl(X,\,\sF\bigr) \to H^1\bigl(X,\,\sF\bigr) \xrightarrow{\beta}
H^1\bigl(X \setminus E,\,\sF\bigr),
$$
shows that to prove the claim of Theorem~\ref{thm:BC}, it is equivalent to
show that $\alpha$ is surjective and that $\beta$ is injective. This is what
we do next.

\subsection{Proof of Theorem~\ref*{thm:BC}: surjectivity of $\pmb \alpha$}

To show surjectivity of $\alpha$, let $\sigma \in H^0 \bigl(X \setminus E,\,
\sF \bigr)$ be any element.  We need to show that there exists an element
$\overline \sigma \in H^0 \bigl(X,\, \sF \bigr)$ such that $\overline
\sigma|_{X \setminus E} = \sigma$.

Decompose $E =
E_{\rm div} \change{\cup} E_{\rm small}$, where $E_{\rm div}$ has pure codimension
one in $X$, and $\codim_X E_{\rm small} \geq 2$. Since $\sF$ is locally free in a
neighbourhood of $E$, it follows immediately from the normality of $X$ that there
exists a section $\sigma' \in H^0 \bigl(X \setminus E_{\rm div},\, \sF \bigr)$ such
that $\sigma'|_{X \setminus E} = \sigma$. In other words, we may assume that the
following holds.

\begin{awlog}
  The algebraic set $E$ has pure codimension one in $X$.
\end{awlog}

Since $\sigma$ is algebraic, it is clear that there exists an extension of
$\sigma$ as a rational section. In other words, there exists a minimal number
$k \in \mathbb N$ and a section
$$
\tau \in H^0 \bigl(X,\, \sF \otimes \sO_X(kE)\bigr)
$$
with $\tau|_{X\setminus E} = \sigma$. To prove surjectivity of $\alpha$, it is
then sufficient to show that $k = 0$. Now, if $s \in S$ is any point, it
follows from the assumption made in \eqref{eq:chws1} of Theorem~\ref{thm:BC}
that there exists a section $\overline \sigma_s \in H^0 \bigl( X_s, \, \sF_s
\bigr)$ such that $\overline \sigma_s |_{X_s \setminus E} = \sigma|_{X_s
  \setminus E} = \tau|_{X_s \setminus E}$.  Since $\sF$ is locally free near
$E$, this immediately implies that $k=0$ and that $\sigma $ is in the image of
$\alpha$, as claimed.

\subsection{Proof of Theorem~\ref*{thm:BC}: injectivity of $\pmb \beta$}

Concerning the injectivity of $\beta$, we consider the following commutative
diagram of restrictions
\begin{equation}
  \xymatrix{
    H^1\bigl(X,\sF \bigr) \ar[d]_{\beta} \ar[rr]^-{\gamma}_-{\text{restr.~to
        $\phi$-fibres}}
    &&
    \prod_{s \in S} H^1\bigl(X_s,\sF_s \bigr) \ar[d]_{\delta}^{\txt{\scriptsize
        restr.~to  open\\\scriptsize part of fibres}} \\
    H^1\bigl(X \setminus E, \sF \bigr) \ar[rr] &&
    \prod_{s \in S} H^1\bigl(X_s \setminus E_s, \sF_s \bigr).
  }
\end{equation}
To prove injectivity of $\beta$, it is then sufficient to prove injectivity of
$\gamma$ and $\delta$.

\subsubsection{Injectivity of $\gamma$}

Since $S$ is affine by Assumption~\ref{awlog:S1}, we have that $H^p\bigl( S,\,
R^q\phi_*\sF\bigr) = 0$ for all $p > 0$ and all $q$. The Leray spectral
sequence, \cite[II.~Thm.~4.17.1]{Godement73}, thus gives a canonical identification
$$
H^1 \bigl(X,\,\sF\bigr) \xrightarrow{\simeq} H^0\bigl(S,\,R^1\phi_*\sF \bigr).
$$
By the second part of Assumption~\ref{awlog:S1} we may apply Grauert's
Theorem~\cite[III~Cor.~12.9]{Ha77} to obtain that the natural map
$$
R^1\phi_*\sF \otimes \mathbb{C}(s) \xrightarrow{\simeq} H^1(X_s,\sF_s)
$$
is an isomorphism for any point $s\in S$. Hence the map $\gamma$ may be identified
with the evaluation map,
$$
H^0\bigl(S,\,R^1\phi_*\sF \bigr) \to
\prod_{s\in S}R^1\phi_*\sF \otimes \kappa(s),
$$
that maps a section of the locally free sheaf $R^1\phi_*\sF$ to its values at the
points of $S$. This map is clearly injective.

\subsubsection{Injectivity of $\delta$}

The injectivity of $\delta$ follows immediately from the assumption made in
\eqref{eq:chws1} of Theorem~\ref{thm:BC} and from the cohomology sequence for
cohomology with support, \cite[III~Ex.~2.3(e)]{Ha77}, already discussed
above. This shows injectivity of $\beta$ and completes the proof of
Theorem~\ref{thm:BC}. \qed

\part{EXTENSION WITH LOGARITHMIC POLES}
\label{part:4}

\section{Main result of this part}

In the present Part~\ref{part:4} of this paper, we make an important step
towards a full proof of the main Extension Theorem~\ref{thm:main} by proving
the following, weaker version of Theorem~\ref{thm:main}.

\begin{thm}[Extension theorem for differential forms on log canonical pairs]\label{thm:extension-lc}
  Let $(X, D)$ be a log canonical pair of dimension $\dim X \geq 2$.  Let
  $\pi: \wtilde X \to X$ be a \wlr of $(X,D)$ with exceptional set $E \subset
  \wtilde X$, and consider the reduced divisor
  $$
  \wtilde D' := \supp ( E + \pi^{-1} \lfloor D \rfloor ).
  $$
  Then the sheaf $\pi_* \, \Omega^p_{\wtilde X}(\log \wtilde D')$ is
  reflexive for any number $0 \leq p \leq n$.
\end{thm}

Theorem~\ref{thm:main} and Theorem~\ref{thm:extension-lc} differ only in the
choice of the divisors $\wtilde D$ and $\wtilde D'$,
respectively. Theorem~\ref{thm:extension-lc} is weaker than
Theorem~\ref{thm:main} because $\wtilde D'$ is larger than $\wtilde D$, so
that Theorem~\ref{thm:extension-lc} allows the extended differential forms to
have poles along a larger number of exceptional divisors then
Theorem~\ref{thm:main} would allow.

\subsection{Reformulation of Theorem~\ref*{thm:extension-lc}}

In Part~\ref{part:5} of this paper, Theorem~\ref{thm:extension-lc} will be
used to give a proof of the main Extension Theorem~\ref{thm:main}, and the
formulation of Theorem~\ref{thm:extension-lc} is designed to make this
application as simple as possible. The formulation is, however, not optimal
for proof. Rather than proving Theorem~\ref{thm:extension-lc} directly, we
have therefore found it easier to prove the following equivalent reformulation
which is more suitable for inductive arguments.

\begin{thm}[Reformulation of Theorem~\ref{thm:extension-lc}]\label{thm:extension-lc2}
  Let $(X, D)$ be a log canonical pair and let $\pi: \wtilde X \to X$ be a
  \wlr with exceptional set $E = \Exc(\pi)$. Consider the reduced divisor
  $$
  \wtilde D' := \supp (E + \pi^{-1} \supp \lfloor D \rfloor ).
  $$
  If $p$ is any index and $E_0 \subseteq E$ any irreducible component, then
  the injective restriction map
  \begin{equation}\label{eq:z0surj}
    r : H^0 \bigl( \wtilde X \setminus \supp(E-E_0),\, \Omega^p_{\wtilde X}(\log
    \wtilde D') \bigr) \to H^0 \bigl( \wtilde X \setminus E,\,
    \Omega^p_{\wtilde X}(\log \wtilde D') \bigr)
  \end{equation}
  is in fact an isomorphism.
\end{thm}

\begin{explanation}\label{expl:sigma}
  We aim to show that Theorem~\ref{thm:extension-lc2} implies
  Theorem~\ref{thm:extension-lc}. To prove Theorem~\ref{thm:extension-lc} we
  need to show that for any open set $U \subseteq X$ with preimage $\wtilde U
  \subseteq \wtilde X$, the natural restriction map
  $$
  r_U : H^0 \bigl( \wtilde U,\, \Omega^p_{\wtilde X}(\log \wtilde D') \bigr)
  \to H^0 \bigl( \wtilde U \setminus E,\, \Omega^p_{\wtilde X}(\log \wtilde
  D') \bigr)
  $$
  is in fact surjective. Thus, let $U \subseteq X$ be any open set, and let
  $\sigma \in H^0 \bigl( \wtilde U \setminus E,\, \Omega^p_{\wtilde X}(\log
  \wtilde D') \bigr)$ be any form.

  Assuming that Theorem~\ref{thm:extension-lc2} holds, it can be applied to
  the lc pair $(U, D)$ and to its \wlr $\pi|_{\wtilde U} : \wtilde U \to U$. A
  repeated application of~\eqref{eq:z0surj} shows that $\sigma$ extends over
  every single component of $E \cap \wtilde U$, and therefore over all of $E
  \cap \wtilde U$.  Surjectivity of the map $r_U$ then follows, and
  Theorem~\ref{thm:extension-lc} is shown.
\end{explanation}

\section{Proof of Theorem~\ref*{thm:extension-lc2}}
\label{sec:Pthmextlc}

The proof of Theorem~\ref{thm:extension-lc2} will be presented in this
section. We will maintain the assumptions and the notation of
\eqref{thm:extension-lc2}. Since the proof is long, we chose to present it as
a sequence of clearly marked and relatively independent steps.

\subsection{Setup of notation and of the main induction loop}

An elementary computation, explained in all detail in \cite[Lem.~2.13]{GKK08},
shows that to prove Theorem~\ref{thm:extension-lc2} for all \wlrs of a given
pair, it suffices to prove the result for one \wlr only. We may therefore
assume the following without loss of generality.

\begin{awlog}
  The \wlr morphism $\pi$ is a \slr.
\end{awlog}

The proof of Theorem~\ref{thm:extension-lc2} involves two nested induction
loops. The main, outer loop considers pairs of numbers $\bigl(\dim X, \codim
\pi(E_0) \bigr)$, which we order lexicographically as indicated in
Table~\vref{tab:lexo}.
\begin{table}[htbp]
  \centering
  \begin{tabular}{lccccccccccc}
    No. & 1 & 2 & 3 & 4 & 5&6&7&8&9&10&$\cdots$ \\
    \hline
    $\dim X$ & 2 & 3 & 3 & 4&4&4&5&5&5&5&$\cdots$ \\
    $\codim \pi(E_0)$ & 2&2&3&2&3&4&2&3&4&5&$\cdots$ \\
    \\
  \end{tabular}
  \caption{Lexicographical ordering of dimensions and codimensions}
  \label{tab:lexo}
\end{table}

\subsection{Main induction loop: start of induction}

The first column of Table~\vref{tab:lexo} describes the case where $\dim X =
2$ and $\codim_X \pi(E_0) = 2$. After some reductions, it will turn out that
this case has essentially been treated previously, in \cite{GKK08}. Given a
surface pair $(X, D)$ as in Theorem~\ref{thm:extension-lc2}, consider the open
subsets
$$
X^0 := X \setminus \supp(D) \,\text{ and }\, X^1 := (X,
D)_{\reg} \cup \supp(D).
$$
Observe that $X^1$ is open and that the complement of $(X, D)_{\reg}$ is
finite. For $i \in \{0,1\}$, we also consider the preimages $\wtilde X^i :=
\pi^{-1}(X^i)$ and induced \wlr $\pi^i : \wtilde X^i \to X^i$.  Since the
statement of Theorem~\ref{thm:extension-lc2} is local on $X$, and since $X =
X^0 \cup X^1$ it suffices to prove Theorem~\ref{thm:extension-lc2} for the two
pairs $(X^0, \emptyset)$ and $(X^1, D)$ independently.

\subsubsection{Resolutions of the pair $(X^0, \emptyset)$}
\label{sss:RBempty}

Since $X$ is a surface, the index $p$ is either zero, one or two. The case
where $p=0$ is trivial. Since $(X^0, \emptyset)$ is reduced and log canonical,
the two remaining cases are covered by earlier results. For $p=1$,
Theorem~\ref{thm:extension-lc2} is shown in \cite[Prop.~7.1]{GKK08}. The case
where $p=2$ is covered by \cite[Prop.~5.1]{GKK08}.

\subsubsection{Resolutions of the pair $(X^1, D)$}

Again, we aim to apply the results of \cite{GKK08}, this time employing ideas
from the discussion of \emph{boundary-lc pairs}, \cite[Sect~3.2]{GKK08}, for
the reduction to known cases.

In complete analogy to the argument of the previous Section~\ref{sss:RBempty},
Theorem~\ref{thm:extension-lc2} follows if we can apply \cite[Prop.~5.1 and
Prop.~7.1]{GKK08} to the reduced pair $(X^1, \lfloor D \rfloor)$. For that, it
suffices to show that the pair $(X^1, \lfloor D \rfloor)$ is log canonical. This
follows trivially from the monotonicity of discrepancies, \cite[Lem.~2.27]{KM98},
once we show that the variety $X^1$ is $\mathbb Q$-factorial.

To this end, observe that for any sufficiently small rational number
$\varepsilon > 0$, the non-reduced pair $\bigl( X^1, (1-\varepsilon) D \bigr)$
is \emph{numerically dlt}; see \cite[Notation~4.1]{KM98} for the definition
and use \cite[Lem.~3.41]{KM98} for an explicit discrepancy computation. By
\cite[Prop.~4.11]{KM98}, the space $X^1$ is then $\mathbb Q$-factorial, as
required.

\subsection{Main induction loop: proof of the inductive step}
\label{ssec:p1mainind}

We are now in a setting where $\dim X \geq 3$. We assume that a number $p \leq
\dim X$ and an irreducible component $E_0 \subseteq E$ are given.

\begin{notation}\label{not:cptsofE}
  If $E$ is reducible, we denote the irreducible components of $E$ by $E_0,
  \ldots, E_N$, numbered in a way such that $\dim \pi(E_1) \leq \dim \pi(E_2)
  \leq \cdots \leq \dim \pi(E_N)$. In particular, if $E$ is reducible, then
  there exists a number $k \geq 0$ so that
  \begin{equation}\label{eq:meaningofk}
    \dim \pi(E_i) > \dim \pi(E_0) \Leftrightarrow N \geq i > k.
  \end{equation}
\end{notation}

If $E$ is irreducible, we use the following obvious notational convention.

\begin{conv}\label{conv:cptsofE}
  If $E$ is irreducible, set $k := N := 0$, and write
  \begin{align*}
    E_1 \cup \cdots \cup E_k & := E_1 \cup \cdots \cup E_N := \emptyset, \quad \text{and}\\
    E_0 \cup \cdots \cup E_k & := E_0 \cup \cdots \cup E_N :=  E_0.
  \end{align*}
\end{conv}
Convention~\ref{conv:cptsofE} admittedly abuses notation. However, it has the
advantage that we can give uniform formulas that work both in the irreducible
and the reducible case. For instance, the restriction morphism
\eqref{eq:z0surj} of Theorem~\ref{thm:extension-lc2} can now be written as
$$
r : H^0 \bigl( \wtilde X \setminus (E_1 \cup \cdots \cup E_N),\,
\Omega^p_{\wtilde X}(\log \wtilde D') \bigr) \to H^0 \bigl( \wtilde X
\setminus (E_0 \cup \cdots \cup E_N),\, \Omega^p_{\wtilde X}(\log \wtilde D')
\bigr).
$$

\subsubsection{Main induction loop: induction hypothesis}

The induction hypothesis asserts that Theorem~\ref{thm:extension-lc2} holds
for all \wlrs of log canonical pairs $(\overline X, \overline D)$ with $\dim
\overline X < \dim X$, and if $\dim \overline X = \dim X$,
then~\eqref{eq:z0surj} holds for all divisors $E_i \subseteq E \subset \wtilde
X$ with $\dim \pi(E_i) > \dim \pi(E_0)$.

Using Convention~\ref{conv:cptsofE} and Formula~\eqref{eq:meaningofk} of
Notation~\ref{not:cptsofE}, the second part of the induction hypothesis
implies that the horizontal arrows in following commutative diagram of
restriction morphisms are both isomorphic,
$$
\xymatrix{ H^0 \bigl( \wtilde X \setminus (E_1 \cup \cdots \cup E_k),\,
  \Omega^p_{\wtilde X}(\log \wtilde D') \bigr) \ar[d]_{s} \ar[r]^{\simeq} & %
  H^0 \bigl( \wtilde X \setminus (E_1 \cup \cdots \cup E_N),\,
  \Omega^p_{\wtilde X}(\log \wtilde D') \bigr) \ar[d]^{\txt{\tiny $r$, want surjectivity}} \\
  H^0 \bigl( \wtilde X \setminus (E_0 \cup \cdots \cup E_k),\,
  \Omega^p_{\wtilde X}(\log \wtilde D') \bigr) \ar[r]^{\simeq} & %
  H^0 \bigl( \wtilde X \setminus (E_0 \cup \cdots \cup E_N),\,
  \Omega^p_{\wtilde X}(\log \wtilde D') \bigr). }
$$

In particular, we obtain the following reformulation of the problem.

\begin{claim}
  To prove Theorem~\ref{thm:extension-lc2} and to show surjectivity of
  \eqref{eq:z0surj}, it suffices to show that the natural restriction map $s$
  is surjective. \qed
\end{claim}

\subsubsection{Simplifications}
\label{ssec:14c2}

To show surjectivity of $s$ and to prove Theorem~\ref{thm:extension-lc2}, it
suffices to consider a Zariski-open subset of $X$ that intersects $\pi(E_0)$
non-trivially. This will allow us to simplify the setup substantially, here
and in Section~\ref{sssect:projective} below.

\begin{claim}\label{claim:ossuff}
  Let $X^\circ \subseteq X$ be any open set that intersects $\pi(E_0)$ non-trivially,
  and let $\wtilde X^\circ := \pi^{-1}(X^\circ)$ be its preimage. If the restriction
  map
  $$
  s^\circ : H^0 \bigl( \wtilde X^\circ \setminus (E_1 \cup \cdots \cup E_k),\,
  \Omega^p_{\wtilde X}(\log \wtilde D') \bigr) \to H^0 \bigl( \wtilde X^\circ
  \setminus (E_0 \cup \cdots \cup E_k),\, \Omega^p_{\wtilde X}(\log \wtilde
  D') \bigr)
  $$
  is surjective, then the map $s$ is surjective and
  Theorem~\ref{thm:extension-lc} holds.
\end{claim}
\begin{proof}
  Given an open set $X^\circ$ and assuming that the associated restriction map
  $s^\circ$ is surjective, we need to show surjectivity of $s$. As in
  Explanation~\ref{expl:sigma}, let
  $$
  \sigma \in H^0 \bigl( \wtilde X \setminus (E_0 \cup \cdots \cup E_k),\,
  \Omega^p_{\wtilde X}(\log \wtilde D') \bigr)
  $$
  be any form defined away from $E_0 \cup \cdots \cup E_k$, and let $c \in
  \mathbb N$ be the minimal number such that $\sigma$ extends to a section
  $$
  \wtilde \sigma \in H^0 \bigl( \wtilde X \setminus (E_1 \cup \cdots \cup
  E_k),\, \sO_{\wtilde X}(cE_0) \otimes \Omega^p_{\wtilde X}(\log \wtilde D')
  \bigr).
  $$
  We need to show that $c = 0$. However, it follows from the surjectivity of
  \eqref{eq:z0surj} on $\wtilde X^{\circ}$ that
  $$
  \wtilde \sigma|_{\wtilde X^\circ \setminus (E_1 \cup \cdots \cup E_k)} \in
  H^0 \bigl( \wtilde X^\circ \setminus (E_1 \cup \cdots \cup E_k),\,
  \Omega^p_{\wtilde X}(\log \wtilde D') \bigr).
  $$
  Since $\bigl( \wtilde X^\circ \setminus (E_1 \cup \cdots \cup E_k) \bigr)
  \cap E_0 \not = \emptyset$, this shows the claim.
\end{proof}

\change{
  \begin{num}\label{num:simplifications}
    We will use Claim~\ref{claim:ossuff} to simplify the situtation by replacing $X$
    with appropriate open subsets successively.
\end{num}}

\begin{awlog}\label{awlog:1}
  The variety $X$ is affine.
\end{awlog}

\PreprintAndPublication{\begin{figure}
  \centering

  \ \\

  $$
  \xymatrix{
    \begin{picture}(4,4)(0,0)
      \put( 0.0, 4.2){\slr $\wtilde X$}
      \put( 0.0, 0.2){\includegraphics[height=3.5cm]{ass197}}
      \put( 3.4, 3.7){\scriptsize divisor $E_1$}
      \put( 3.4, 3.6){\vector(-1, -1){0.4}}
      \put( 2.9, 0.3){\scriptsize divisor $E_0$}
      \put( 2.8, 0.4){\vector(-2, 1){0.4}}
    \end{picture}
    \ar[rr]^{\pi}_{\text{\slr}} &&
    \begin{picture}(4,4)(0,0)
      \put( 0.0, 4.2){singular space $X$}
      \put( 0.0, 0.2){\includegraphics[height=3.5cm]{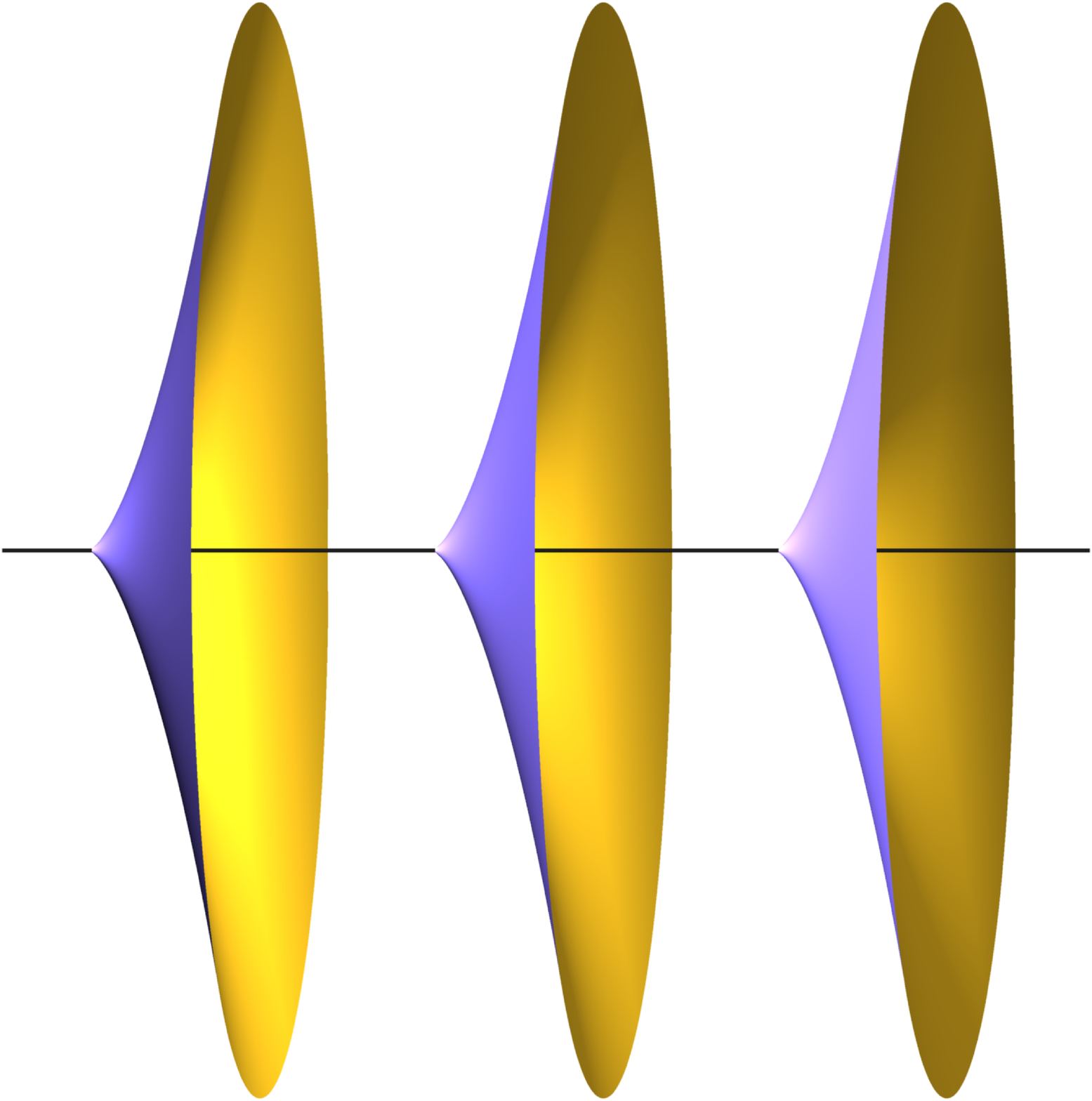}}
      \put( 3.9, 2.6){\scriptsize point $\pi(E_0)$}
      \put( 3.8, 2.6){\vector(-2, -1){1.2}}
      \put( 2.4, 1.87){\scriptsize $\bullet$}
      \put( 3.8, 1.3){\scriptsize curve $\pi(E_1)$}
      \put( 3.8, 1.5){\vector(-1, 1){0.4}}
    \end{picture}
  }
  $$

  {\small The figure sketches a situation where Assumption~\ref{awlog:2} holds. Here,
    $X$ is a threefold whose singular locus is a curve. The exceptional set of the
    \slr $\pi$ contains two divisors $E_0$ and $E_1$. Assumption~\ref{awlog:2} is
    satisfied because $E_0$ is mapped to a point that is contained in the image of
    $E_1$. Another example where $\pi(E_0) = \pi(E_1)$ is shown in
    Figure~\vref{fig:SAPTP}.}

  \caption{A three-dimensional example where Assumption~\ref*{awlog:2} holds}
  \label{fig:TSWAH}
\end{figure}}{} %
Claim~\ref{claim:ossuff} also allows to remove from $X$ all images $\pi(E_i)$ of
exceptional divisors $E_i \subseteq E$ with $\pi(E_0) \not\subseteq
\pi(E_i)$\PreprintAndPublication{, bringing us to the situation sketched in
  Figure~\vref{fig:TSWAH}}{}. This will again simplify notation substantially.

\begin{awlog}\label{awlog:2}
  If $E_i \subseteq E$ is an irreducible component, then $\pi(E_0) \subseteq
  \pi(E_i)$.
\end{awlog}

\begin{cons}
  We have $\pi(E_0) = \cdots = \pi(E_k)$, $E_0 \cup \cdots \cup E_k \subseteq
  \pi^{-1}\bigl(\pi(E_0) \bigr)$, and
  \begin{equation}\label{eq:cdpipie0}
    \codim_{\wtilde X} \pi^{-1}\bigl(\pi(E_0)\bigr) \setminus (E_0 \cup \cdots
    \cup E_k) \geq 2.
  \end{equation}
\end{cons}

Assumption~\ref{awlog:2} has further consequences. Because of the
inequality~\eqref{eq:cdpipie0}, and because $\Omega^p_{\wtilde X}(\log \wtilde
D')$ is a locally free sheaf on $\wtilde X$, any $p$-form defined on $\wtilde
X \setminus \pi^{-1}\bigl(\pi(E_0)\bigr)$ will immediately extend to a
$p$-form on $\wtilde X \setminus (E_0 \cup \cdots \cup E_k)$. It follows that
the bottom arrow in the following commutative diagram of restriction maps is
in fact an isomorphism,
$$
\xymatrix{ H^0 \bigl( \wtilde X,\, \Omega^p_{\wtilde X}(\log \wtilde D') \bigr)
  \ar[d]_{t} \ar[r]^(.4){\txt{\tiny injective}} & %
  H^0 \bigl( \wtilde X \setminus (E_1 \cup \cdots \cup E_k),\, \Omega^p_{\wtilde
    X}(\log \wtilde D') \bigr) \ar[d]^{\txt{\tiny $s$, want
      surjectivity}} \\
  H^0 \bigl( \wtilde X \setminus \pi^{-1}\bigl(\pi(E_0)\bigr),\, \Omega^p_{\wtilde X
  }(\log \wtilde D') \bigr) & %
  H^0 \bigl( \wtilde X \setminus (E_0 \cup \cdots \cup E_k),\, \Omega^p_{\wtilde
    X}(\log \wtilde D') \bigr).  \ar[l]_(.51){\simeq}}
$$
Maintaining Assumptions~\ref{awlog:1} and \ref{awlog:2}, the following is thus
immediate.

\begin{cons}\label{cons:suffT}
  To show surjectivity of $s$ and to prove Theorem~\ref{thm:extension-lc2}, it
  suffices to show that the natural restriction map $t$ is surjective. \qed
\end{cons}

\subsubsection{The case $\dim \pi(E_0) = 0$}
\label{sssect:ising}

If the divisor $E_0$ is mapped to a point, Steenbrink-type vanishing for cohomology
with supports, Corollary~\ref{cor:H1rest}, applies. More precisely, the surjectivity
statement (\ref{cor:H1rest}.\ref{item:4.1.1}) of Remark~\ref{rem:H1rest} asserts that
the restriction morphism $t$ is surjective. This will finish the proof in case where
$\dim \pi(E_0) = 0$. We can therefore assume from now on that $E_0$ is \emph{not}
mapped to a point.

\begin{awlog}
  The variety $\pi(E_0)$ \change{is smooth and} has positive dimension.
\end{awlog}

\subsubsection{Projection to $\pi(E_0)$}
\label{sssect:projective}

Given a base change diagram
$$
\xymatrix{%
  Z \times_X \wtilde X \ar[d]_{\wtilde \pi} \ar[rr]^{\Gamma}_{\txt{\scriptsize
      étale, open}} &&
  \wtilde X \ar[d]^{\pi} \\
  Z \ar[rr]^{\gamma}_{\txt{\scriptsize étale, open}} && X }
$$
such that $\gamma^{-1} \bigl(\pi(E_0) \bigr) \not = \emptyset$, surjectivity of the
restriction map $t$ will follow as soon as we prove surjectivity of the analogously
defined map
$$
H^0 \bigl( \wtilde Z,\, \Omega^p_{\wtilde Z}(\log \wtilde \Delta') \bigr) \to
H^0 \bigl( \wtilde Z \setminus \pi^{-1}\bigl(\pi(F_0)\bigr),\,
\Omega^p_{\wtilde Z}(\log \wtilde \Delta') \bigr),
$$
where $\wtilde Z := Z \times_X \wtilde X$, $\wtilde \Delta' =
\Gamma^{-1}(\wtilde D')$, and $F_0$ is a component of
$\Gamma^{-1}(E_0)$. Since $X$ is affine by Assumption~\ref{awlog:1}, one such
diagram is given by Proposition~\ref{prop:projection} when projecting to the
affine subvariety $\pi(E_0) \subset X$. Observing that $\bigl(Z, \gamma^{*}(D)
\bigr)$ is lc with log-resolution $\wtilde \pi$, that
$$
\wtilde \Delta' := \supp \bigl( (\text{$\wtilde \pi$-exceptional set}) +
\wtilde \pi^{-1}( \supp \lfloor \gamma^*D \rfloor) \bigr),
$$
and that all additional assumptions made so far will also hold for $\wtilde
\pi : \wtilde Z \to Z$, we may replace $X$ by $Z$ and assume the following
without loss of generality.

\begin{awlog}\label{awlog:SNC1}
  There exists a smooth affine variety $T$ with a free sheaf of differentials,
  $\Omega^1_T \simeq \sO_T^{\oplus \dim T}$, and a commutative diagram of
  surjective morphisms
  $$
  \xymatrix{ \wtilde X \ar[rr]_{\pi \text{\tiny, log. resolution}}
    \ar@/^.5cm/[rrrr]^{\psi \text{\tiny, smooth}} && X \ar[rr]_{\phi} && T }
  $$
  where the restriction $\phi|_{\pi(E_0)} : \pi(E_0) \to T$ is an isomorphism and
  both $\phi$ and $\psi$ have connected fibres.
\end{awlog}

\begin{awlog}\label{awlog:SNC2}
  The composition $\psi := \phi \circ \pi$ is an snc morphism of the pair $(\wtilde
  X, \wtilde D')$, in the sense of Definition~\ref{def:sncMorphism}. In particular,
  recall from Remark~\ref{rem:fiberSNC} that if $t \in T$ is any point, then the
  scheme-theoretic intersection $\wtilde D' \cap \psi^{-1}(t)$ is reduced, of pure
  codimension one in $\psi^{-1}(t)$, and has simple normal crossing support.
\end{awlog}

\begin{notation}
  If $t \in T$ is any point, we consider the varieties $X_t := \phi^{-1}(t)$,
  $\wtilde X_t := \psi^{-1}(t)$, divisors $E_t := E \cap \wtilde X_t$,
  $E_{0,t} := E_0 \cap \wtilde X_t$, $\wtilde D'_t := \wtilde D' \cap \wtilde
  X_t$, \ldots, and morphisms $\pi_t := \pi|_{\wtilde X_t} : \wtilde X_t \to
  X_t$, \ldots
\end{notation}

\PreprintAndPublication{
\begin{figure}
  \centering

  \ \\

  $$
  \xymatrix{
    \begin{picture}(4.8,4)(0,0)
      \put( 0.0, 4.2){smooth threefold $\wtilde X$}
      \put( 1.0, 0.2){\includegraphics[height=3.5cm]{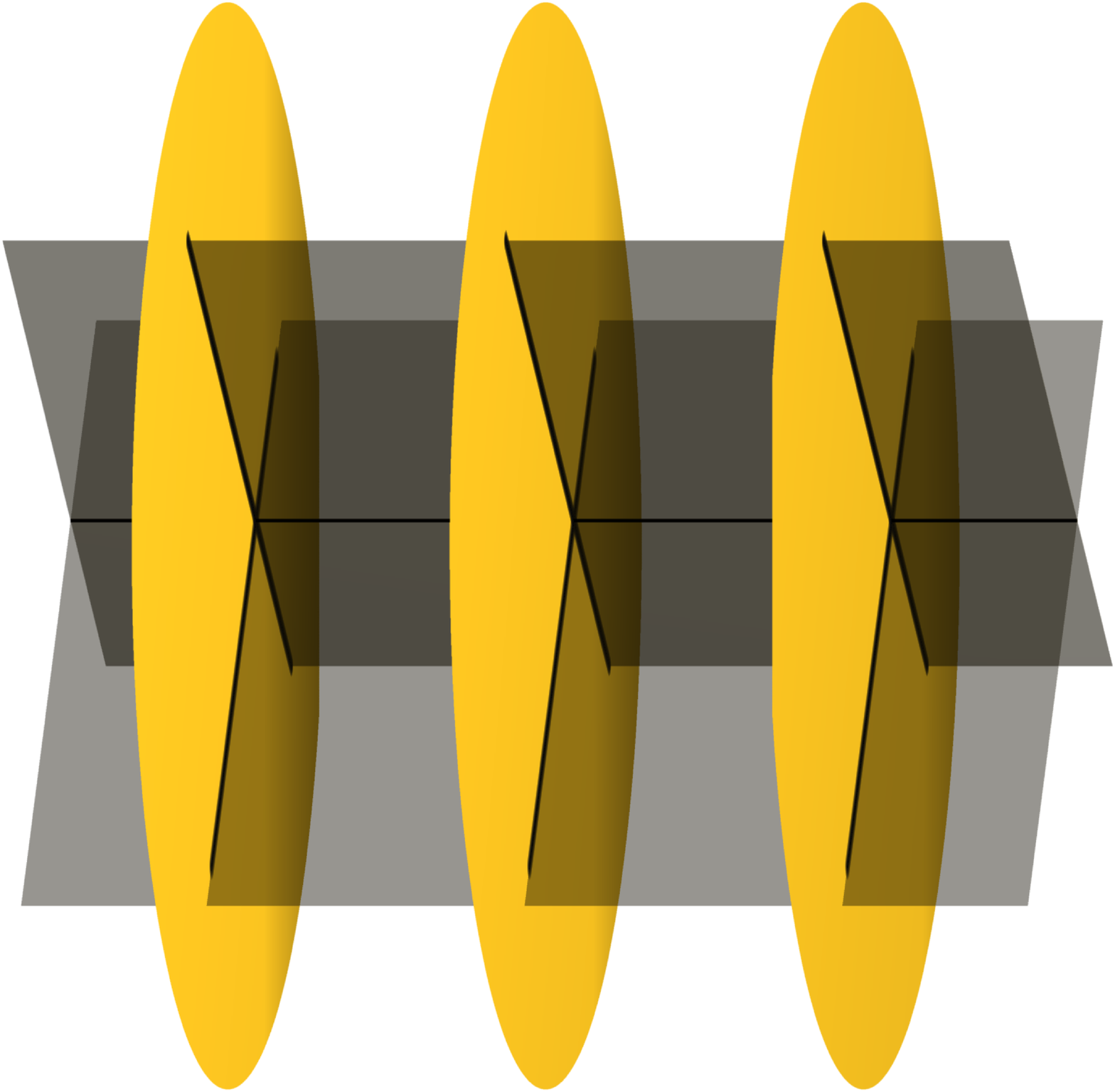}}
      \put(-0.4, 3.0){\scriptsize divisor $E_0$}
      \put( 0.7, 3.0){\vector(3, -1){0.4}}
      \put(-0.4, 1.2){\scriptsize divisor $E_1$}
      \put( 0.7, 1.2){\vector(4, -1){0.4}}
    \end{picture}
    \ar[rr]^(.52){\pi}_(.52){\text{\slr}}
    \ar@<-7mm>@/_5mm/[drr]_(.55){\psi\text{, smooth map}\quad} &&
    \begin{picture}(4,4)(0,0)
      \put( 0.0, 4.2){singular space $X$}
      \put( 0.0, 0.2){\includegraphics[height=3.5cm]{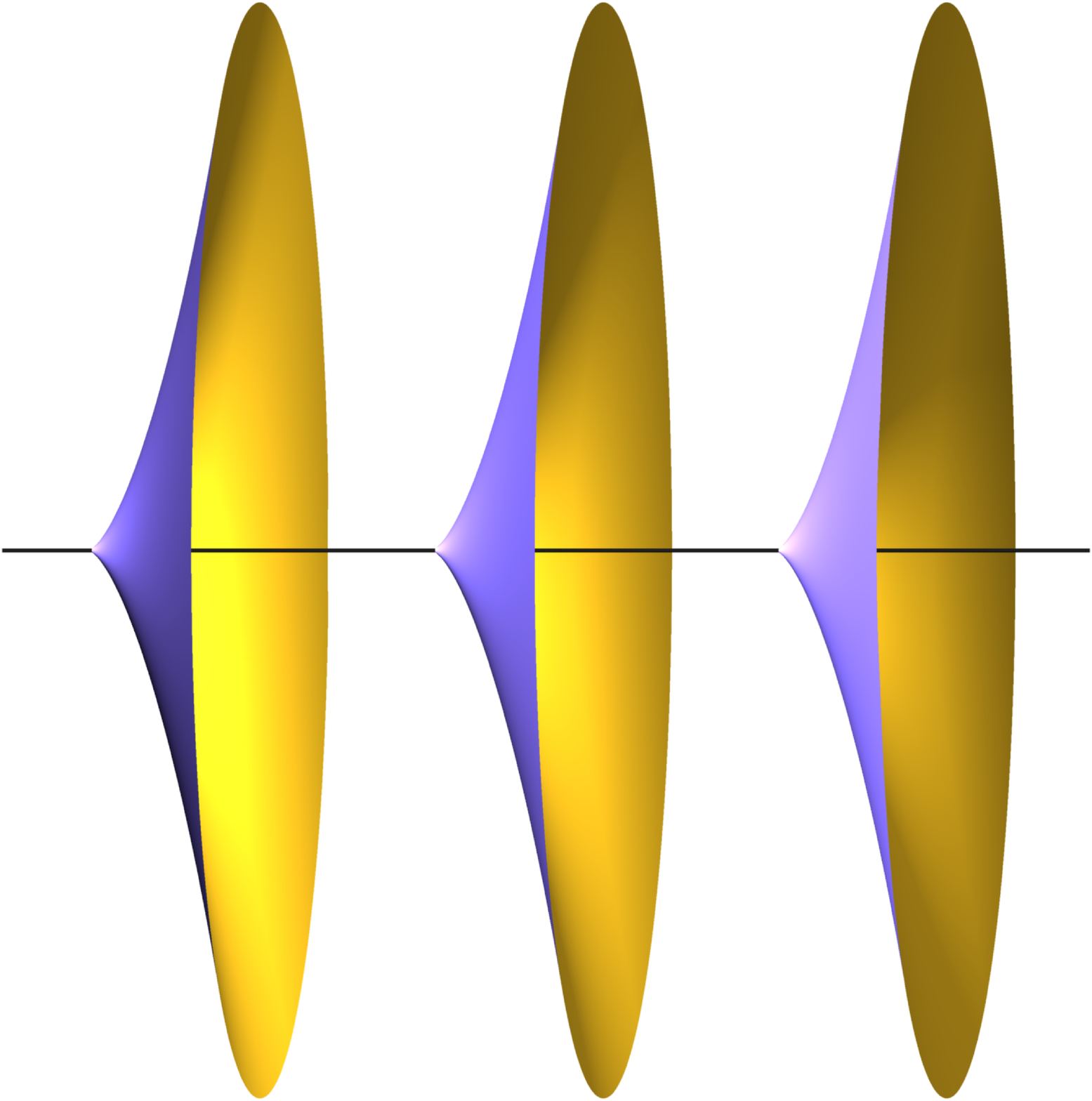}}
      \put( 3.8, 1.3){\scriptsize $\pi(E_0)$}
      \put( 3.8, 1.5){\vector(-1, 1){0.4}}
    \end{picture}
    \ar[d]_(.75){\phi}^(.75){\text{projection}} \\ &&
    \begin{picture}(3.5, 0.7)(0,0)
      \put( 0.0, 0.3){smooth affine variety $T$}
      \put( 0.0, 0.0){\includegraphics[width=3.5cm]{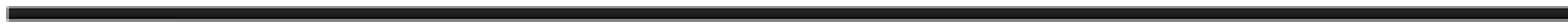}}
    \end{picture}
  }
  $$

  {\small The figure sketches a situation where Assumption~\ref{awlog:SNC2}
    holds, in the simple case where $\lfloor D \rfloor = 0$ and $\wtilde
    \Delta' = E_0 \cup E_1$. The morphism $\pi_t$ maps the curves $E_{0,t}$ to
    isolated singularities of $\phi$-fibres. The morphism $\psi$ is an snc
    morphism of the pair $(\wtilde X, \wtilde D)$. }

  \caption{Situation after projection to $\pi(E_0)$}
  \label{fig:SAPTP}
\end{figure}}{}

\PreprintAndPublication{The present setup is sketched in
  Figure~\vref{fig:SAPTP}.}{} We will now show that all assumptions made in
Theorem~\ref{thm:extension-lc2} also hold for the general fibre $\wtilde X_t$
of $\psi$. Better still, the morphism $\pi_t$ maps $E_{0,t}$ to a point. In
Section~\ref{sssec:reldif}, we will then be able to apply
Corollary~\ref{cor:H1rest} to fibres of $\psi$. A vanishing result for
cohomology with support will follow.

\begin{claim}\label{claim:p61atf}
  If $t \in T$ is a general point, then $(X_t, D_t)$ is a \change{log canonical pair}, and
  the morphism $\pi_t : \wtilde X_t \to X_t$ is a \wlr of the pair $(X_t, D_t)$ which
  has $E_t$ as its exceptional set and contracts the divisor $E_{0,t}$ to a point.
  Further, we have
  $$
  \wtilde D'_t = (\supp E_t) \cup \pi_t^{-1}\bigl( \supp \lfloor D_t
  \rfloor \bigr).
  $$
\end{claim}
\begin{proof}
  The fact that $\pi_t(E_{0,t})$ is a point is immediate from
  Assumption~\ref{awlog:SNC1}.  The remaining assertions follow from
  Lemma~\ref{lem:cuttingDown} and \vref{lem:cuttingDownRes}.
\end{proof}

Again, shrinking $T$ and $X$ to simplify notation, we may assume without loss
of generality that the following holds.

\begin{awlog}
  The conclusion of Claim~\ref{claim:p61atf} holds for all points $t \in T$.
\end{awlog}

\subsubsection{Vanishing results for relative differentials}
\label{sssec:reldif}

Claim~\ref{claim:p61atf} asserts that $\pi$ maps $E_{0,t} := E_0 \cap \wtilde
X_t$ to a single point. The Steenbrink-type vanishing result for cohomology
with supports, Corollary~\ref{cor:H1rest}, therefore guarantees the vanishing
of cohomology groups with support on $E_{0,t}$, for sheaves of differentials
on $\wtilde X_t$.

\begin{claim}\label{claim:fbwv}
  If $t \in T$ is any point, and if $z \in \pi(E_0)$ is the unique point with
  $\phi(z)=t$, then $H^1_{\pi^{-1}(z)}\bigl(\widetilde X_t,\,
  \Omega^q_{\wtilde X_t}(\log \wtilde D'_t) \bigr) = 0$ for all numbers $0
  \leq q \leq \dim X - \dim T$. \qed
\end{claim}

Claim~\ref{claim:fbwv} and the Generic Base Change Theorem for cohomology with
supports, Theorem~\ref{thm:BC}, then immediately give the following vanishing
of cohomology with support on $E_0$, for sheaves of relative differentials on
$\wtilde X$, possibly after shrinking $T$.

\begin{claim}\label{claim:loccohom}
  We have $H^1_{\pi^{-1}(\pi(E_0))}\bigl(\widetilde X,\, \Omega^q_{\wtilde
    X/T}(\log \wtilde D') \bigr) = 0$ for all numbers $1 \leq q \leq \dim X -
  \dim T$. \qed
\end{claim}

\subsubsection{Relative differential sequences, completion of the proof}
\label{sssect:rdeof}

By Assumption~\ref{awlog:SNC2}, the divisor $\wtilde D'$ is relatively snc
over $T$.  As we have recalled in Section~\ref{sec:relReflSeqSNC}, this
implies the existence of a filtration
$$
\Omega^p_{\wtilde X}(\log \wtilde D') = \sF^0 \supseteq \sF^1 \supseteq \cdots
\supseteq \sF^p \supseteq \sF^{p+1} = 0,
$$
with quotients
\begin{equation}\label{eq:ses}
  \xymatrix{ %
    0 \ar[r] & \sF^{r+1} \ar[r] & \sF^r \ar[r] &
    \psi^*  \Omega^r_T \otimes \Omega^{p-r}_{\wtilde X/T}(\log
    \wtilde D') \ar[r] & 0.
  }
\end{equation}
By Assumption~\ref{awlog:SNC1}, the pull-backs $\psi^* \Omega^r_T$ are trivial vector
bundles, and the sheaves $\sF^r/\sF^{r+1}$ are therefore isomorphic to direct sums of
several copies of $\Omega^{p-r}_{\wtilde X/T}(\log \wtilde D')$. For simplicity, we
will therefore use the somewhat sloppy notation
$$
\sF^r/\sF^{r+1} = \Omega^{p-r}_{\wtilde X/T}(\log \wtilde D')^{\oplus
  \bullet}.
$$

Recall Observation~\ref{cons:suffT}, which asserts that to prove
Theorem~\ref{thm:extension-lc2}, it suffices to show that the injective
restriction map
\begin{equation}\label{eq:bois}
  t : H^0\bigl(\wtilde X,\, \sF^0 \bigr) \to H^0\bigl(\underbrace{\wtilde X
    \setminus \pi^{-1}( \pi(E_0))}_{=: \wtilde X^\circ},\, \sF^0 \bigr).
\end{equation}
is surjective. To this end, we consider the long exact cohomology sequences
associated with~\eqref{eq:ses}, and with its restriction to $\wtilde X^\circ =
\wtilde X \setminus \pi^{-1}( \pi(E_0))$. Table~\vref{tab:castor} shows an
excerpt of the commutative diagram that is relevant to our discussion.
\begin{table}
  \centering
  $$
  \xymatrix{
    H^0\bigl(\wtilde X,\, \sF^{r+1} \bigr) \ar[d] \ar@{^{(}->}[rr]^{a_r} &&
    H^0\bigl(\wtilde X^\circ,\, \sF^{r+1} \bigr) \ar[d] \\
    H^0\bigl(\wtilde X,\, \sF^r \bigr) \ar[d] \ar@{^{(}->}[rr]^{b_r} &&
    H^0\bigl(\wtilde X^\circ,\, \sF^r \bigr) \ar[d] \\
    H^0\bigl(\wtilde X,\, \Omega^{p-r}_{\wtilde X/T}(\log \wtilde D')^{\oplus \bullet} \bigr) \ar[d] \ar[rr]^{c_r}_{\simeq} &&
    H^0\bigl(\wtilde X^\circ,\, \Omega^{p-r}_{\wtilde X/T}(\log \wtilde D')^{\oplus \bullet} \bigr) \ar[d] \\
    H^1\bigl(\wtilde X,\, \sF^{r+1} \bigr) \ar[rr]^{d_r} \ar[d] &&
    H^1\bigl(\wtilde X^\circ,\, \sF^{r+1} \bigr) \ar[d] \\
    H^1\bigl(\wtilde X,\, \sF^r \bigr) \ar[d] \ar[rr]^{e_r} &&
    H^1\bigl(\wtilde X^\circ,\, \sF^r \bigr) \ar[d] \\
    H^1\bigl(\wtilde X,\, \Omega^{p-r}_{\wtilde X/T}(\log \wtilde D')^{\oplus \bullet} \bigr) \ar@{^{(}->}[rr]^{f_r} &&
    H^1\bigl(\wtilde X^\circ,\, \Omega^{p-r}_{\wtilde X/T}(\log \wtilde D')^{\oplus \bullet} \bigr)
  }
  $$
  \caption{Long exact cohomology sequences for relative differentials}
  \label{tab:castor}
\end{table}

Note that the restriction map $t$ of~\eqref{eq:bois} appears under the name
$b_0$ in Table~\ref{tab:castor}. While it is clear that the restriction
morphisms $a_r$, $b_r$ and $c_r$ are injective, surjectivity of $c_r$ and
injectivity of $f_r$ both follow from Claim~\ref{claim:loccohom} when one
applies the standard long exact sequence for cohomology with supports,
\cite[III~Ex.~2.3(e)]{Ha77}, to the sheaf $\sA := \Omega^{p-r}_{\wtilde
  X/T}(\log \wtilde D')^{\oplus \bullet}$,
\begin{multline*}
  \underbrace{H^0_{\pi^{-1}(\pi(E_0))}\bigl(\wtilde X,\, \sA \bigr)}_{= \{0\}
    \text{ because $\sA$ is torsion free}} \to H^0\bigl(\wtilde X,\, \sA
  \bigr) \xrightarrow{c_r} H^0\bigl(\wtilde X^\circ,\, \sA \bigr) \\
  \to \underbrace{H^1_{\pi^{-1}(\pi(E_0))}\bigl(\wtilde X,\, \sA \bigr)}_{=
    \{0\} \text{ by Claim~\ref{claim:loccohom}}} \to H^1\bigl(\wtilde X,\, \sA
  \bigr) \xrightarrow{f_r} H^1\bigl(\wtilde X^\circ,\, \sA \bigr) \to \cdots
\end{multline*}
In this setting, surjectivity of the restriction map $t = b_0$ follows from an
inductive argument. More precisely, we use descending induction to show that
the following stronger statement holds true.

\begin{claim}\label{claim:sfar}
  For all numbers $r \leq p$ the following two statements hold true.
  \begin{enumerate}
  \item\ilabel{il:laurel} The map $b_r : H^0\bigl(\wtilde X,\, \sF^r \bigr)
    \to H^0\bigl(\wtilde X^\circ,\, \sF^r \bigr)$ is surjective.

  \item\ilabel{il:hardy} The map $e_r : H^1\bigl(\wtilde X,\, \sF^r \bigr) \to
    H^1\bigl(\wtilde X^\circ,\, \sF^r \bigr)$ is injective.
  \end{enumerate}
\end{claim}

\subsubsection*{Proof of Claim~\ref{claim:sfar}, start of induction: $r = p$}

In this setup, $\sF^{r+1} = 0$, the map $d_r$ is obviously injective, and
$a_r$ is surjective.  Statement~\iref{il:laurel} follows when one applies the
Four-Lemma for Surjectivity,
\PreprintAndPublication{Lemma~\ref{lem:4sur}}{\cite[XII~Lem.~3.1(ii)]{McL75}},
to the first four rows of Table~\ref{tab:castor}. Statement~\iref{il:hardy}
then immediately follows when one applies the Four-Lemma for Injectivity,
\PreprintAndPublication{Lemma~\ref{lem:4inj}}{\cite[XII~Lem.~3.1(i)]{McL75}},
to the last four rows of Table~\ref{tab:castor}. \qed

\subsubsection*{Proof of Claim~\ref{claim:sfar}, inductive step}

Let $r < p$ be any given number and assume that Statements~\iref{il:laurel}
and \iref{il:hardy} were known for all indices larger than $r$. Since $d_r =
e_{r+1}$ is injective by assumption, and $a_r = b_{r+1}$ is surjective, we
argue as in case $r=p$ above: Statement~\iref{il:laurel} follows from the
Four-Lemma for Surjectivity,
\PreprintAndPublication{Lemma~\ref{lem:4sur}}{\cite[XII~Lem.~3.1(ii)]{McL75}},
and the first four rows of Table~\ref{tab:castor}. Statement~\iref{il:hardy}
follows from the Four-Lemma for Injectivity,
\PreprintAndPublication{Lemma~\ref{lem:4inj}}{\cite[XII~Lem.~3.1(i)]{McL75}},
and the last four rows of Table~\ref{tab:castor}. \qed

\subsubsection*{Summary}

In summary, we have shown surjectivity of the restriction map $t = b_0$. This
completes the proof of Theorem~\ref{thm:extension-lc2} and hence of
Theorem~\ref{thm:extension-lc}. \qed

\part{PROOF OF THE EXTENSION THEOREM~\ref*{thm:main}}
\label{part:5}

\section{Proof of Theorem~\ref*{thm:main}, idea of Proof}
\label{sec:extwopoles-intro}

To explain the main ideas in the proof of the Extension Theorem~\ref{thm:main},
consider the case where $X$ is a klt space that contains a single isolated
singularity, and let $\pi: \wtilde X \to X$ be a \slr of the pair $(X,\emptyset)$,
with $\pi$-exceptional divisor $E \subset \wtilde X$. As explained in
Remark~\vref{rem:whyextension}, to prove Theorem~\ref{thm:main} we need to show that
for any open set $U \subseteq X$ with preimage $\wtilde U$, any differential form
defined on $\change{\wtilde U} \setminus E$ extends across $E$, to give a differential form
defined on all of \change{$\wtilde U$}. To this end, fix an open set $U \subseteq X$
and let $\sigma \in H^0 \bigl( \wtilde U \setminus E,\, \Omega^p_{\wtilde X} \bigr)$
be any form. For simplicity of notation, we assume without loss of generality that $U
= X$. Also, we consider only the case where $p > 1$ in this sketch.

As a first step towards the extension of $\sigma$, we have seen in
Theorem~\vref{thm:extension-lc} that $\sigma$ extends as a form with
logarithmic poles along $E$, say $\overline \sigma \in H^0 \bigl( \wtilde X
\setminus E,\, \Omega^p_{\wtilde X} (\log E) \bigr)$. Next, we need to show
that $\overline \sigma$ really does not have any poles along $E$. To motivate
the strategy of proof, we consider two simple cases first.

\subsection{The case where $E$ is irreducible}
\label{ssec:extwopoles-introA}

Assume that $E$ is irreducible. To show that $\overline \sigma$ does not have
any logarithmic poles along $E$, recall from Fact~\vref{fact:poletest} that it
suffices to show that $\overline \sigma$ is in the kernel of the residue map
$$
\rho^p : H^0\bigl( \wtilde X,\, \Omega^p_{\wtilde X}(\log E) \bigr) \to
H^0\bigl( E,\, \Omega^{p-1}_E \bigr).
$$
On the other hand, we know from a result of Hacon-McKernan,
\cite[Cor.~1.5(2)]{HMcK07}, that $E$ is rationally connected, so that
$h^0\bigl( E,\, \Omega^{p-1}_E \bigr) = 0$. This clearly shows that $\overline
\sigma$ is in the kernel of $\rho^p$ and completes the proof when $E$ is
irreducible.

\subsection{The case where $(\wtilde X,E)$ has a simple mmp}
\label{ssec:extwopoles-introB}

\PreprintAndPublication{
\begin{figure}
  \centering

  \ \\

  $$
  \xymatrix{
    \begin{picture}(4.8,4)(0,0)
      \put( 1.0, 0.2){\includegraphics[height=3.5cm]{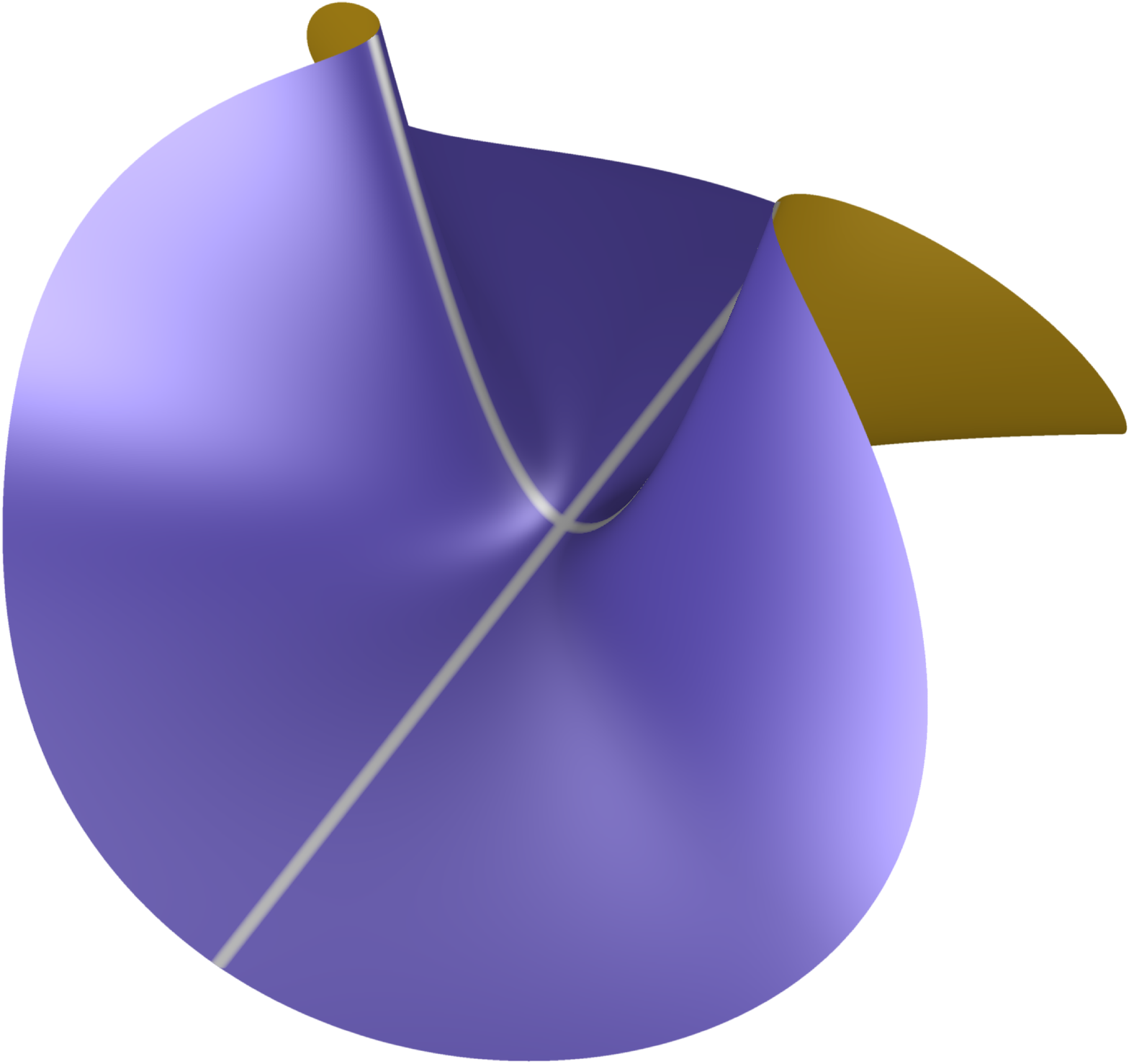}}
      \put( 0.0, 4.2){snc surface pair $(X_0, E_1+E_2)$}
      \put( 0.4, 3.4){\scriptsize divisor $E_1$}
      \put( 1.5, 3.4){\vector(4, -1){0.6}}
      \put(-0.4, 1.2){\scriptsize divisor $E_2$}
      \put( 0.8, 1.2){\vector(4, -1){1.1}}
    \end{picture} \quad
    \ar[rr]^(.52){\lambda_1}_(.52){\text{contracts $E_1$}}
    \ar@<-10mm>@/_5mm/[drr]_(.55){\text{\slr}\,\,\,} &&
    \begin{picture}(4,4)(0,0)
      \put( 0.0, 0.2){\includegraphics[height=3.5cm]{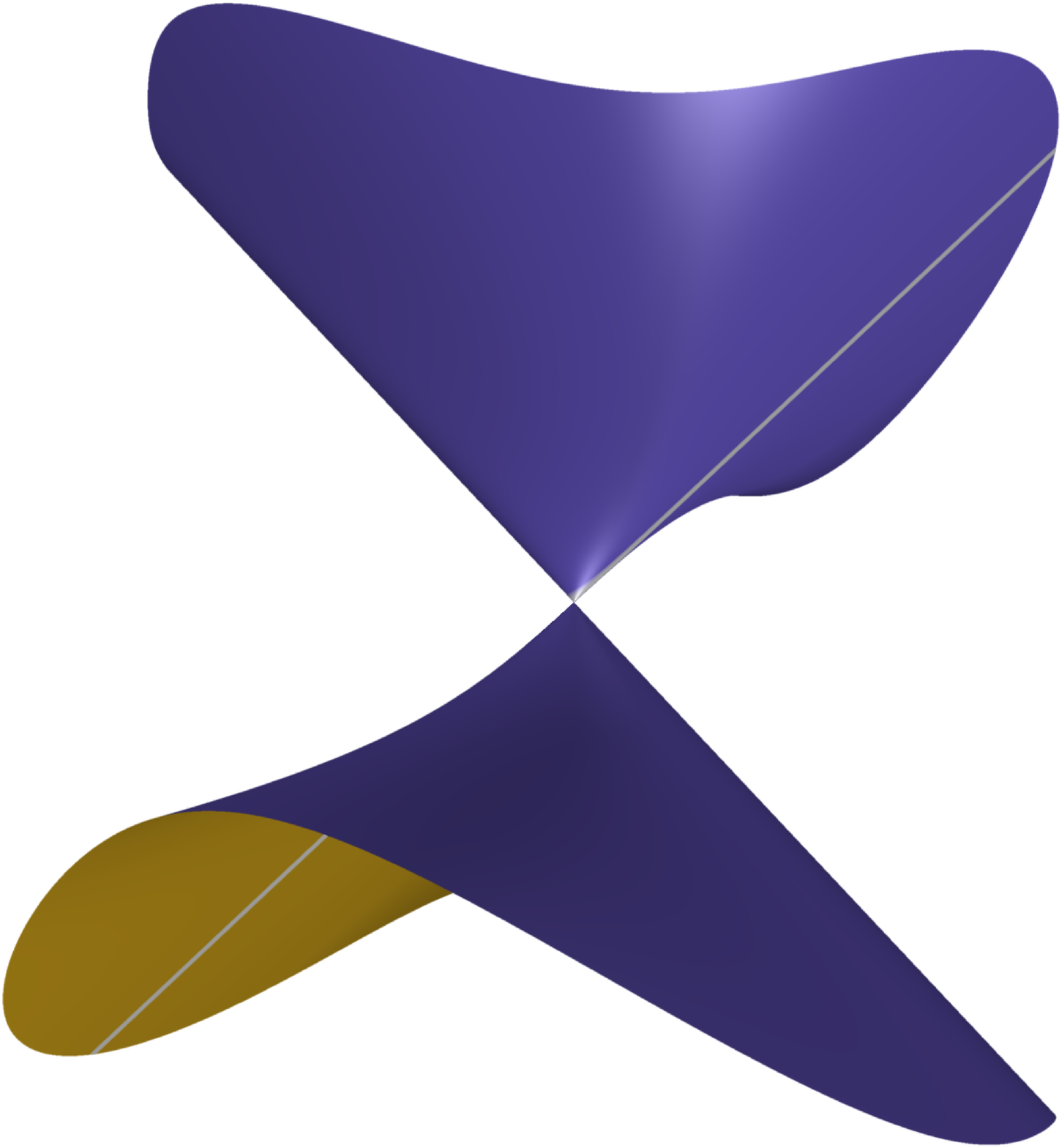}}
      \put( 0.0, 4.2){dlt surface pair $(X_1, E_{2,1})$}
      \put( 2.8, 1.4){\scriptsize divisor $E_{2,1}$}
      \put( 2.7, 1.6){\vector(-1, 4){0.2}}
    \end{picture}
    \ar[d]_(.55){\lambda_2}^(.55){\text{contracts $E_{2,1}$}} \\ &&
    \begin{picture}(3.5, 3.2)(0,0)
      \put( 0.0, 0.0){\includegraphics[width=3.5cm]{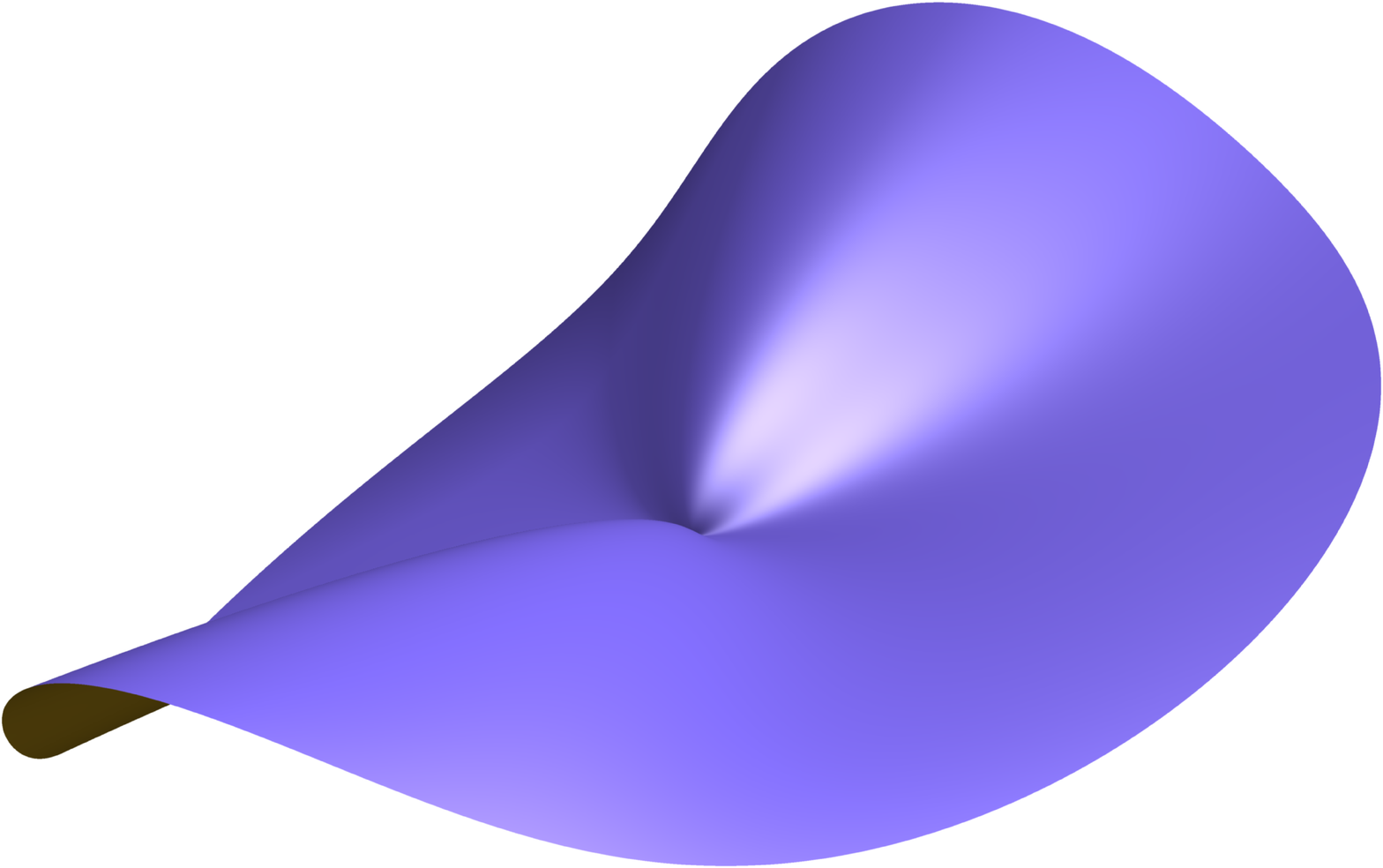}}
      \put( 0.0, 2.5){klt surface pair $(X, \emptyset)$}
    \end{picture}
  }
  $$

  {\small This sketch shows the \slr of an isolated klt surface singularity, and the
    decomposition of the \slr given by the minimal model program of the snc pair
    $(X_0, E_1+E_2)$. The example is taken from \cite{Baum}.}

  \caption{\Slr of an isolated klt surface singularity}
  \label{fig:srklt}
\end{figure}}{} In general, the divisor $E$ need not be irreducible. Let us
therefore consider the next difficult case that where $E$ is reducible with
two components, say $E = E_1 \cup E_2$. The \slr $\pi$ will then factor via a
$\pi$-relative minimal model program of the pair $(\wtilde X, E)$, which we
assume for simplicity to have the following particularly special form,
\PreprintAndPublication{sketched also in Figure~\vref{fig:srklt}}{}
$$
\xymatrix{ %
  \wtilde X = X_0 \ar[rrrr]^(.55){\lambda_1}_(.55){\text{contracts $E_1$ to a
      point}} &&&& X_1 \ar[rrrrr]^{\lambda_2}_{\text{contracts $E_{2,1} :=
      (\lambda_1)_*(E_2)$ to a point}} &&&&& X. }
$$
In this setting, the arguments of Section~\ref{ssec:extwopoles-introA} apply
to show that $\overline \sigma$ has no poles along the divisor $E_1$. To show
that $\overline \sigma$ does not have any poles along the remaining component
$E_2$, observe that it suffices to consider the induced reflexive form on the
possibly singular space $X_1$, say $\overline \sigma_1 \in H^0 \bigl( X_1 ,\,
\Omega^{[p]}_{X_1} (\log E_{2,1}) \bigr)$, where $E_{2,1} :=
(\lambda_1)_*(E_2)$, and to show that $\overline \sigma_1$ does not have any
poles along $E_{2,1}$.  For that, we follow the arguments of
Section~\ref{ssec:extwopoles-introA} once more, carefully accounting for the
singularities of the pair $(X_1, E_{2,1})$.

The pair $(X_1, E_{2,1})$ is dlt, and it follows that the divisor $E_{2,1}$ is
necessarily normal, \cite[Cor.~5.52]{KM98}. Using the residue map for
reflexive differentials on dlt pairs that was constructed in
Theorem~\vref{thm:relativereflexiveresidue},
$$
\rho^{[p]} : H^0\bigl( X_1,\, \Omega^{[p]}_{X_1}(\log E_{2,1}) \bigr) \to
H^0\bigl( E_{2,1},\, \Omega^{[p-1]}_{E_{2,1}} \bigr),
$$
we have seen in Remark~\ref{rem:poletest} that it suffices to show that
$\rho^{[p]}(\overline \sigma_1) = 0$. Because the morphism $\lambda_2$
contracts the divisor $E_{2,1}$ to a point, the result of Hacon-McKernan will
again apply to show that $E_{2,1}$ is rationally connected. Even though there
are numerous examples of rationally connected spaces that carry non-trivial
reflexive forms, we claim that in our special setup we do have the vanishing
\begin{sequation}\label{eq:nrdf}
  h^0\bigl( E_{2,1},\, \Omega^{[p-1]}_{E_{2,1}} \bigr) = 0.
\end{sequation}

Recall from the adjunction theory for Weil divisors on normal spaces,
\cite[Chapt.~16 and Prop.~16.5]{FandA92} and \cite[Sect.~3.9 and
Glossary]{MR2352762}, that there exists a Weil divisor $D_E$ on the normal
variety $E_{2,1}$ which makes the pair $(E_{2,1}, D_E)$ klt. Now, if we knew
that the extension theorem would hold for the pair $(E_{2,1}, D_E)$, we can
prove the vanishing~\eqref{eq:nrdf}, arguing exactly as in the proof of
Theorem~\vref{thm:kltRCC}, where we show the non-existence of reflexive forms
on rationally connected klt spaces as a corollary of the Extension
Theorem~\ref{thm:main}. Since $\dim E_{2,1} < \dim X$, this suggests an
inductive proof, beginning with easy-to-prove extension theorems for reflexive
forms on surfaces, and working our way up to higher-dimensional varieties. The
proof of Theorem~\ref{thm:main} follows this inductive pattern.

\subsection{The general case}

The assumptions made in
Sections~\ref{ssec:extwopoles-introA}--\ref{ssec:extwopoles-introB} of course
do not hold in general. To handle the general case, we need to work with pairs
$(X,D)$ where $D$ is not necessarily empty, the $\pi$-relative minimal model
program might involve flips, and the singularities of $X$ need not be
isolated. All this leads to a slightly protracted inductive argument, which is
outlined in all detail in the next section.

\section{Proof of Theorem~\ref*{thm:main}, overview of the proof}
\label{sec:extwopoles-overview}

\subsection{Notation used in the induction}
\label{ssec:IndNot}

We aim to prove Theorem~\ref{thm:main} for log canonical pairs of arbitrary
dimension. As we will argue by induction, we often need to prove statements of
the form ``If Proposition~\ref{prop:extwopoleslem} holds for all pairs of a
given dimension $n \geq 2$, then Proposition~\ref{prop:extwopoles} will hold
for all pairs of the same dimension $n$''. It makes sense to introduce the
following shorthand notation for this,
$$
\forall n \geq 2 : \text{Proposition~\ref{prop:extwopoleslem}($n$)}
\Longrightarrow \text{Proposition~\ref{prop:extwopoles}($n$)}.
$$
Likewise, to say that ``Given any number $n \geq 2$, if
Proposition~\ref{prop:kltRCC-n} holds for all pairs of dimension $n' \leq n$,
then Proposition~\ref{prop:extwopoles} will hold for all pairs of dimension
$n+1$'', we will write
$$
\forall n\geq 2: \bigl( \text{Proposition~\ref{prop:kltRCC-n}($n'$)}, \forall
n' \leq n \bigr) \Longrightarrow
\text{Proposition~\ref{prop:extwopoleslem}($n+1$)}
$$
If we want to say that Proposition~\ref{prop:extwopoleslem} holds for surface
pairs, we will often write
$$
\text{Proposition~\ref{prop:extwopoleslem}($n=2$)}.
$$

\subsection{Theorems and propositions that appear in the induction}

Before giving an overview of the induction process and listing the
implications that we will prove, we have gathered in this section a complete
list of the theorems and propositions that will play a role in the proof.

In the setup of the Extension Theorem~\ref{thm:main}, we have seen in
Theorem~\vref{thm:extension-lc} that any differential form on $\wtilde X$
which is defined away from the $\pi$-exceptional set $E$ extends as a form
with logarithmic poles along $E$. As a consequence, we will see in
Section~\ref{sec:S0} that to prove the Extension Theorem~\ref{thm:main}, it
suffices to show that the following Proposition holds for all numbers $n \geq
2$.

\begin{prop}[Non-existence of logarithmic poles for pairs of dimension $n$]\label{prop:extwopoles}
  Let $(X, D)$ be a log canonical pair of dimension $\dim X = n$, and let
  $\pi: \wtilde X \to X$ be a \wlr of $(X, D)$, with exceptional set $E
  \subset \wtilde X$.  Consider the two divisors
  \begin{align*}
    \wtilde D & := \text{largest reduced divisor contained
      in } \supp \pi^{-1}(\text{non-klt locus}), \\
    \wtilde D' & := \supp \bigl( E + \pi^{-1} \supp \lfloor D \rfloor \bigr),
  \end{align*}
  and observe that $\wtilde D \subseteq \wtilde D'$. Then the natural
  injection
  \begin{equation}\label{eq:extwopoles}
    H^0 \bigl( \wtilde X,\, \Omega^p_{\wtilde X}(\log \wtilde
    D) \bigr) \to H^0 \bigl( \wtilde X,\, \Omega^p_{\wtilde
      X}(\log \wtilde D') \bigr)
  \end{equation}
  is in fact isomorphic.
\end{prop}

\begin{subrem}
  Recall from Lemma~\vref{lem:logresforsmalldiv} that the pair $(\wtilde X,
  \wtilde D')$ is reduced and snc. Being a subdivisor of a divisor with simple
  normal crossing support, the pair $(\wtilde X, \wtilde D)$ is likewise
  reduced and snc. In particular, it follows that the sheaves
  $\Omega^p_{\wtilde X}(\log \wtilde D)$ and $\Omega^p_{\wtilde X}(\log
  \wtilde D')$ mentioned in~\eqref{eq:extwopoles} are locally free.
\end{subrem}

As indicated in Section~\ref{sec:extwopoles-intro}, we aim to prove
Proposition~\ref{prop:extwopoles} by using the $\pi$-relative minimal model
program of the pair $(X,E)$, in order to contract one irreducible component of
$E$ at a time. The proof of Proposition~\ref{prop:extwopoles} will then depend
on the following statements, which assert that differential forms extend
across irreducible, contractible divisors. For technical reasons, we handle
the cases of $1$-forms and of $p$-forms separately.

\begin{prop}[Extension of 1-forms over contractible divisors]\label{prop:extwopoleslemp1}
  Let $(X, D)$ be a dlt pair of dimension $\dim X \geq 2$, where $X$ is
  $\mathbb Q$-factorial, and let $\lambda : X \to X_{\lambda}$ be a divisorial
  contraction of \change{a} minimal model program associated with the pair $(X,D)$,
  contracting an irreducible divisor $D_0 \subseteq \supp \lfloor D
  \rfloor$. Then the natural injection
  $$
  H^0 \bigl( X,\, \Omega^{[1]}_X \bigr) \to H^0 \bigl( X, \Omega^{[1]}_X(\log
  D_0) \bigr)
  $$
  is isomorphic.
\end{prop}

\begin{prop}[Extension of $p$-forms over contractible divisors in dimension $n$]\label{prop:extwopoleslem}
  Let $(X, D)$ be a dlt pair of dimension $\dim X = n$, where $X$ is $\mathbb
  Q$-factorial, and let $\lambda : X \to X_{\lambda}$ be a divisorial
  contraction of \change{a} minimal model program associated with the pair $(X,D)$,
  contracting an irreducible divisor $D_0 \subseteq \supp \lfloor D
  \rfloor$. Then the natural injection
  $$
  H^0 \bigl( X,\, \Omega^{[p]}_X \bigr) \to H^0 \bigl( X, \Omega^{[p]}_X(\log
  D_0) \bigr)
  $$
  is isomorphic for all numbers $1 < p \leq \dim X$.
\end{prop}

Finally, we have seen in Section~\ref{sec:extwopoles-intro} that the
non-existence of reflexive differentials on rationally chain connected klt
spaces enters the proof of Proposition~\ref{prop:extwopoleslemp1}. The
relevant statement is this, compare also Theorem~\vref{thm:kltRCC}.

\begin{prop}[Reflexive differentials on rcc pairs of dimension $n$]\label{prop:kltRCC-n}
  Let $(X,D)$ be a klt pair of dimension $\dim X = n$. If $X$ is rationally
  chain connected, then $X$ is rationally connected and $H^0 \bigl( X,\,
  \Omega^{[p]}_X \bigr) = 0$ for all numbers $1 \leq p \leq \dim X$.
\end{prop}

\subsection{Overview of the induction process}

Using the notation introduced in Section~\vref{ssec:IndNot},
Table~\vref{tab:pf} shows the structure of the inductive proof of the
Extension Theorem~\ref{thm:main}. The steps are carried out in
Sections~\ref{sec:S0}--\ref{sec:S3}, respectively, Step~5 being by far the
most involved.

\begin{table}[htbp]
  \centering
  \begin{tabular}{cp{10cm}}
    Step  & Statement shown \\
    \hline
    \\
    0 & $\bigl($ Proposition~\ref{prop:extwopoles}($n$), $\forall n\geq 2 \, \bigr)$
    $\Longrightarrow$ Extension Theorem~\ref{thm:main} \\[2mm]

    1 & Proposition~\ref{prop:extwopoleslemp1} \\[2mm]

    2 & Proposition~\ref{prop:extwopoleslem}($n=2$) \\[2mm]

    3 & $\forall n\geq 2:$ Propositions~\ref{prop:extwopoleslemp1} and \ref{prop:extwopoleslem}($n$) $\Longrightarrow$ Proposition~\ref{prop:extwopoles}($n$)\\[2mm]

    4 & $\forall n\geq 2:$ Proposition~\ref{prop:extwopoles}($n$) $\Longrightarrow$ Proposition~\ref{prop:kltRCC-n}($n$).\\[2mm]

    5 & $\forall n\geq 2: \bigl( $ Proposition~\ref{prop:kltRCC-n}($n'$), $\forall n' \leq n\, \bigr) \Longrightarrow$ Proposition~\ref{prop:extwopoleslem}($n+1$)\\
    \\
  \end{tabular}

  \caption{Overview of the induction used to prove Theorem~\ref*{thm:main}.}
  \label{tab:pf}
\end{table}

\section{Step 0 in the proof of Theorem~\ref*{thm:main}}
\label{sec:S0}

Assuming that Proposition~\ref{prop:extwopoles} holds for log canonical pairs
of arbitrary dimension, we show in this section that the Extension
Theorem~\ref{thm:main} follows. To prove Theorem~\ref{thm:main}, let $(X,D)$
be an arbitrary lc pair, and let $\pi: \wtilde X \to X$ be an arbitrary \wlr,
with exceptional set $E \subset \wtilde X$. Following
Remark~\ref{rem:whyextension}, we need to show that for any open set $U
\subseteq X$ with preimage $\wtilde U \subseteq \wtilde X$, any differential
form
$$
\sigma \in H^0 \bigl( \wtilde U \setminus E,\, \Omega^p_{\wtilde X}(\log
\wtilde D) \bigr) \quad\text{extends to a form}\quad \wtilde \sigma \in H^0
\bigl( \wtilde U,\, \Omega^p_{\wtilde X}(\log \wtilde D) \bigr),
$$
where $\wtilde D$ is the divisor on $\wtilde X$ defined in
Theorem~\ref{thm:main}.

As a first step in this direction, given an open set $U$ and a form $\sigma$,
apply Theorem~\ref{thm:extension-lc} to the pair $(U, D)$, to obtain an
extension of $\sigma$ to a differential form
$$
\wtilde \sigma' \in H^0 \bigl( \wtilde U,\, \Omega^p_{\wtilde X}(\log \wtilde
D') \bigr),
$$
where $\wtilde D' \supseteq \wtilde D$ is the larger divisor defined in
Theorem~\ref{thm:extension-lc}. An application of
Proposition~\ref{prop:extwopoles} to the pair $(U,D)$ will then show that
$\wtilde \sigma'$ in fact does not have any logarithmic poles along the
difference divisor $\wtilde D' - \wtilde D$. This finishes Step 0 in the proof
of Theorem~\ref{thm:main}.

\section{Step 1 in the proof of Theorem~\ref*{thm:main}}
\label{sec:S1}

In this section, we will prove Proposition~\ref{prop:extwopoleslemp1}. We
maintain the assumptions and the notation of the proposition. As we will see,
the assertion follows from the Extension Theorem~\cite[Thm.~1.1]{GKK08} for
1-forms on reduced, log canonical pairs. Let $r : \wtilde X \to X_\lambda$ be
a \emph{\wlr} of $\bigl( X_\lambda, \emptyset \bigr)$ that factors through
$X$. We obtain a diagram
$$
\xymatrix{ %
  \wtilde X \ar[rr]_{\rho} \ar@/^5mm/[rrrr]^{r = \lambda \circ \rho} && X
  \ar[rr]_{\lambda} && X_\lambda}.
$$
Now let $\sigma \in H^0 \bigl( X,\, \Omega^{[1]}_X(\log D_0)\bigr)$ be any
given reflexive form on $X$, possibly with logarithmic poles along
$D_0$. Since the divisor $D_0$ is contracted by $\lambda$, the form $\sigma$
induces a reflexive form $\sigma_\lambda \in H^0 \bigl( X_\lambda,\,
\Omega^{[1]}_{X_\lambda} \bigr)$, without any poles.

\begin{claim}\label{claim:extinspec}
  The direct image sheaf $r_*\Omega^1_{\wtilde X}$ is reflexive. In
  particular, the pull-back of $\sigma_\lambda$ to $\wtilde X$ by $r$ defines
  a regular form $\wtilde \sigma \in H^0\bigl( \wtilde X,\, \Omega^1_{\wtilde
    X} \bigr)$, which agrees with the pull-back of $\sigma$ by $\rho$ wherever
  the morphism $\rho$ is isomorphic.
\end{claim}

\begin{subrem}\label{srem:exeq}
  We refer to Remark~\vref{rem:whyextension} for an explanation why
  reflexivity of $r_*(\Omega^p_{\wtilde X})$ and the extension of pull-back
  forms are equivalent.
\end{subrem}

\begin{proof}[Proof of Claim~\ref{claim:extinspec}]
  Let $r' : \wtilde X' \to X_\lambda$ be any \slr of the pair $(X_\lambda,
  \emptyset)$. The Comparison Lemma, \cite[Lem.~2.13]{GKK08}, then asserts
  that the direct image sheaf $r_*(\Omega^1_{\wtilde X})$ is reflexive if
  $(r')_*(\Omega^1_{\wtilde X})$ is reflexive. Reflexivity of
  $(r')_*(\Omega^1_{\wtilde X})$, however, follows from the Extension Theorem
  \cite[Thm.~1.1]{GKK08} for 1-forms on reduced, log canonical pairs once we
  show that $(X_\lambda, \emptyset)$ is klt.

  To this end, recall from \cite[Sect.~3.31]{KM98} that $X_\lambda$ is
  $\mathbb Q$-factorial, and that the pair $\bigl(X_\lambda, \lambda_*D
  \bigr)$ is again dlt.  The fact that $(X_\lambda, \emptyset)$ is klt then
  follows from \cite[Cor.~2.39 and Prop.~2.41]{KM98} because $\lambda_*D$ will
  be $\mathbb Q$-Cartier. This completes the proof of
  Claim~\ref{claim:extinspec}.
\end{proof}

By Claim~\ref{claim:extinspec}, the pull-back form $\wtilde \sigma$ does not
have any poles along the strict transform of $D_0$, this shows that $\sigma$
does not have any poles along $D_0$, as claimed. This completes the proof of
Proposition~\ref{prop:extwopoleslemp1} and therefore finishes Step~1 in the
proof of Theorem~\ref{thm:main}. \qed

\section{Step 2 in the proof of Theorem~\ref*{thm:main}}
\label{ssec:S1b}

We will now prove Proposition~\ref{prop:extwopoleslem}($n=2$).  Thus $n = \dim
X = 2$ and $p =2$. Let $\sigma \in H^0 \bigl( X,\, \Omega^{[2]}_X(\log
D_0)\bigr)$ be any given reflexive form on $X$.
Recall from Theorem~\vref{thm:relativereflexiveresidue} that there exists a residue
map for reflexive differentials,
$$
\rho^{[2]} : \Omega^{[2]}_X(\log D_0) \to \Omega^{[1]}_{D_0},
$$
which agrees with the residue map of the standard residue sequence
\eqref{eq:stdResidue} wherever the dlt pair $(X, D_0)$ is snc. Also, recall
from \cite[Cor.~2.39(1) and Cor.~5.52]{KM98} that $(X, D_0)$ is dlt, and that
$D_0$ is normal. The divisor $D_0$ is therefore a smooth curve, and
$\Omega^{[1]}_{D_0} = \Omega^1_{D_0}$. Adjunction together with the fact that
$-(K_X+D)$ is $\lambda$-ample implies that $D_0 \simeq \mathbb P^1$. The space
of differentials of $D_0$ is therefore trivial, $H^0 \bigl( D_0,\,
\Omega^1_{D_0}) = 0$. In particular, it follows that $\rho^{[2]}(\sigma) =
0$. It follows from the fact that the residue map acts as a test for
logarithmic poles, see Remark~\ref{rem:poletest}, that $\sigma \in H^0 \bigl(
X,\, \Omega^{[2]}_X \bigr)$, as claimed. This completes the proof of
Proposition~\ref{prop:extwopoleslem}($n=2$) and therefore finishes Step~2 in
the proof of Theorem~\ref{thm:main}. \qed

\section{Step 3 in the proof of Theorem~\ref*{thm:main}}
\label{sec:S2}

Let $(X, D)$ be a log canonical pair of dimension $n$, and let $\pi: \wtilde X \to X$
be a \wlr. We need to show surjectivity of the natural inclusion
map~\eqref{eq:extwopoles}, assuming that Proposition \ref{prop:extwopoleslem}($n$)
holds. Observing that the statement of Proposition~\ref{prop:extwopoles}($n$) is
local on $X$, we may assume that the following holds.
\begin{awlog}\label{awlog:S2.a}
  The space $X$ is affine.
\end{awlog}
Furthermore, if $\wtilde D'_0 \subset \wtilde D'$ is any irreducible component such
that $\pi(\wtilde D'_0)$ is contained in the non-klt locus of $(X, D)$, then
\change{$\wtilde D'_0$} is also contained in $\wtilde D$. We may therefore
assume without loss of generality that the following holds

\begin{awlog}\label{awlog:S2.b}
  The pair $(X, D)$ is klt.
\end{awlog}

\change{  Let $E \subset \wtilde X$ denote the $\pi$-exceptional set. In order to prove
  surjectivity of \eqref{eq:extwopoles} it is equivalent to show that the natural map
  \begin{equation}\label{eq:extwopoles2}
    H^0 \bigl( \wtilde X,\, \Omega^p_{\wtilde X} \bigr) \to H^0 \bigl( \wtilde
    X,\, \Omega^p_{\wtilde X}(\log E) \bigr)
  \end{equation}
  is surjective.  By the definition of klt there exist effective $\pi$-exceptional divisors $F$ and $G$ without common
  components such that $\rdown F=0$ and such that the following $\mathbb Q$-linear equivalence holds:
  \begin{equation}\label{eq:qequivlce}
    K_{\wtilde X} + \pi^{-1}_*D + F \,\,\sim_{\mathbb Q}\,\, \pi^*(K_X+D) + G.
  \end{equation}
  Let $\Delta_\varepsilon=\pi^{-1}_*D+F+\varepsilon E$.  Choosing a small enough
  $0<\varepsilon\ll 1$ we may assume that the pair $(X,\Delta_\varepsilon)$ is klt.
  Let $H\subset \wt X$ be a $\pi$-ample divisor such that $(X,\Delta_\varepsilon+H)$
  is still klt and $K_{\wt X}+\Delta_\varepsilon+H$ is $\pi$-nef. We may then run the $\pi$-relative $(\wt X,\Delta_\epsilon)$ minimal model program
  with scaling of $H$ cf.\ \cite[Cor.~1.4.2]{BCHM06},
  \cite[Thms.~5.54,5.63]{Hacon-Kovacs10}.  Therefore there exists a commutative
  diagram }
$$
\xymatrix{ \wtilde X = X_0 \ar@{-->}[rr]^(.55){\lambda_1} \ar@/_0.3cm/[drrrrrr]_{\pi
    = \pi_0} && X_1 \ar@{-->}[rr]^(.4){\lambda_2} \ar@/_0.1cm/[drrrr]^{\pi_1} &&
  X_2 \,\, \cdots \,\, X_{k-1} \ar@{-->}[rr]^(.6){\lambda_k} && X_k \ar[d]^{\pi_k}\\
  &&&&&& X }
$$
\change{%
  where the $\lambda_i$ are either divisorial contractions or flips. The spaces $X_i$ are normal, $\mathbb Q$-factorial, and if $\Delta_i \subset X_i$
  denotes the cycle-theoretic image of $\Delta_\varepsilon$, then the pairs $\bigl(
  X_i, \Delta_i \bigr)$ are klt for all $i$. The minimal model program terminates with a pair $\bigl(
  X_k, \Delta_k \bigr)$ whose associated $\mathbb Q$-divisor $K_{X_k}+\Delta_k$ is
  $\pi_k$-nef.

  \begin{notation}
    Given any $0 \leq i \leq k$, let $E_i$ (respectively $G_i$) denote the
    cycle-theoretic image of $E$ (respectively $G$) on $X_i$.
  \end{notation}
}

\begin{claim}\label{claim:lastissmall}
  The morphism $\pi_k$ is a small map. In particular, $E_k =
  \emptyset$.
\end{claim}

\begin{proof}\change{%
    It is clear from the construction that $\supp E_i$ is precisely the divisorial
    part of the $\pi_i$-exceptional set.  Then the $\mathbb Q$-linear
    equivalence~\eqref{eq:extwopoles2} implies that
    $$
    K_{X_i} + \Delta_i \,\,\sim_{\mathbb Q}\,\, \pi_i^*(K_X + D) + G_i+\varepsilon E_i,
    $$}\change{
    where $G_i+\varepsilon E_i$ is effective and $\supp(G_i+\varepsilon E_i) =
    \supp(E_i)$. By item (\ref{lem:excdiv}.\ref{el:excdiv1}) of Lemma~\ref{lem:excdiv}, this implies that $K_{X_i} +
    \Delta_i$ is not $\pi_i$-nef as long as $E_i \not = \emptyset$.  It follows that
    $E_k$, the divisorial part of the $\pi_k$-exceptional set, is empty. This shows
    Claim~\ref{claim:lastissmall}.}
\end{proof}

\change{%
  Let $\sigma \in H^0 \bigl( \wtilde X,\, \Omega^p_{\wtilde X}(\log E) \bigr)$ be
  arbitrary. In order to complete Step 3 we need to show that $\sigma \in H^0 \bigl(
  \wtilde X,\, \Omega^p_{\wtilde X} \bigr)$. Clearly, $\sigma$ induces} reflexive
  forms $\sigma_i \in H^0 \bigl( X_i,\, \Omega^{[p]}_{X_i}(\log E_i) \bigr)$, for all
  $i$. Since $E_k = \emptyset$, the reflexive form $\sigma_k$ does not have any
  logarithmic poles at all, that is, $\sigma_k \in H^0 \bigl( X_k,\,
  \Omega^{[p]}_{X_k} \bigr)$. Now consider the map $\lambda_k : X_{k-1} \dasharrow
  X_k$.
\begin{itemize}
\item If $\lambda_k$ is a flip, then $\lambda_k$ is isomorphic in codimension one and
  it is clear that $\sigma_{k-1}$ again does not have logarithmic poles along any
  divisor.
\item If $\lambda_k$ is a divisorial contraction, then the $\lambda_k$-exceptional
  set is contained in $E_{k-1}$, and either Proposition~\ref{prop:extwopoleslemp1} or
  Proposition~\ref{prop:extwopoleslem}($n$) applies to the map $\lambda_k$.
\end{itemize}
In either case, we obtain that $\sigma_{k-1} \in H^0 \bigl( X_{k-1},\,
\Omega^{[p]}_{X_{k-1}} \bigr)$. Applying the same argument successively to
$\lambda_k$, $\lambda_{k-1}$, \ldots, $\lambda_{1}$, we find that
$$
\sigma = \sigma_0 \in H^0 \bigl(\change{\wt X,\, \Omega^{p}_{\wt X}} \bigr),
$$
as claimed. This completes the proof of Proposition~\ref{prop:extwopoles}($n$), once
Propositions~\ref{prop:extwopoleslemp1} and \ref{prop:extwopoleslem}($n$) are known
to hold. Step 3 in the proof of Theorem~\ref{thm:main} is thus finished. \qed

\section{Step 4 in the proof of Theorem~\ref*{thm:main}}
\label{sec:S2.3}

As in Proposition~\ref{prop:kltRCC-n}, let $(X,D)$ be a klt pair of dimension $\dim X
= n$, and assume that $X$ is rationally chain connected. Assuming that
Proposition~\ref{prop:extwopoles}($n$) holds, we need to show that $X$ is rationally
connected, and that $H^0 \bigl( X,\, \Omega^{[p]}_X \bigr) = 0$ for all numbers $1
\leq p \leq n$.

To this end, choose a \slr $\pi: \wtilde X \to X$. Since klt pairs are also dlt, a
result of Hacon and McKernan, \cite[Cor.~1.5(2)]{HMcK07}, applies to show that $X$
and $\wtilde X$ are both rationally connected. In particular, recall from
\cite[IV.~Cor.~3.8]{K96} that
\begin{sequation}\label{eq:nfrms}
  H^0 \bigl( \wtilde X,\, \Omega^p_{\wtilde X}) =0 \quad \forall p > 0.
\end{sequation}
Next, let $\sigma \in H^0 \bigl( X,\, \Omega^{[p]}_X \bigr)$ be any reflexive form.
We need to show that $\sigma = 0$. We consider the pull-back $\wtilde \sigma$, which
is a differential form on $\wtilde X$, possibly with poles along the
$\pi$-exceptional set $E$.  However, since $(X,D)$ is klt,
Theorem~\ref{thm:extension-lc} from page~\pageref{thm:extension-lc} asserts that
$\wtilde \sigma$ has at most logarithmic poles along $E$.
Proposition~\ref{prop:extwopoles}($n$) then applies to show that $\wtilde \sigma$
does in fact not have any poles at all. The assertion that $\wtilde \sigma = 0$ then
follows from~\eqref{eq:nfrms}.

This shows that Proposition~\ref{prop:kltRCC-n}($n$) follows from
Proposition~\ref{prop:extwopoles}($n$), and finishes Step~4 in the proof of
Theorem~\ref{thm:main}.

\section{Step 5 in the proof of Theorem~\ref*{thm:main}}
\label{sec:S3}

\subsection{Setup}

Throughout the present Section~\ref{sec:S3}, we consider the following setup.

\begin{setup}\label{setup:19.1}
  Let $(X, D)$ be a dlt pair of dimension $\dim X = n+1 > 2$, where $X$ is is
  $\mathbb Q$-factorial, and let $\lambda : X \to X_{\lambda}$ be a divisorial
  contraction of a minimal model program associated with the pair $(X,D)$,
  contracting a divisor $D_0 \subseteq \supp \lfloor D \rfloor$. We assume
  that Proposition~\ref{prop:kltRCC-n}($n'$) holds for all numbers $n' \leq
  n$.
\end{setup}

\begin{rem}\label{rem:dlt}
  Since $\lambda$ is a divisorial contraction of a minimal model program, the
  space $X_\lambda$ is again $\mathbb Q$-factorial, and the pair $\bigl(
  X_\lambda,\, \lambda_*D \bigr)$ is dlt. By $\mathbb Q$-factoriality, the
  pairs $(X, D_0)$, $(X, \emptyset)$ and $(X_\lambda, \emptyset)$ will
  likewise be dlt, \cite[Cor.~2.39]{KM98}.
\end{rem}

In order to prove Proposition~\ref{prop:extwopoleslem}($n+1$) and thus to
complete the proof of Theorem~\ref{thm:main}, we need to show that the natural
inclusion map
$$
H^0 \bigl( X,\, \Omega^{[p]}_X \bigr) \to H^0 \bigl( X, \Omega^{[p]}_X(\log
D_0) \bigr)
$$
is surjective for all numbers $1 < p \leq \dim X$. To this end, let $\sigma
\in H^0 \bigl( X,\, \Omega^{[p]}_X(\log D_0)\bigr)$ be any given reflexive
form on $X$. We show that the following holds.

\begin{claim}\label{claim:mos18A}
  The reflexive form $\sigma$ does not have any log poles, i.e., $\sigma \in
  H^0 \bigl( X,\, \Omega^{[p]}_X \bigr)$.
\end{claim}

We will prove Claim~\ref{claim:mos18A} in Sections~\ref{sec:S3.1} and
\ref{sec:S3.2}, considering the cases where $\dim \lambda(D_0) = 0$ and $\dim
\lambda(D_0) > 0$ separately. Before starting with the proof, we include
preparatory Sections~\ref{ssec:S3.A}--\ref{ssec:S3.4} where we recall facts
used in the proof, set up notation, and discuss the (non)existence of
reflexive relative differentials on $D_0$.

\subsection{Adjunction for the divisor $\pmb{D_0}$ in  $\pmb{X}$}
\label{ssec:S3.A}

By inversion of adjunction the support of the divisor $D_0$ is normal,
\cite[Cor.~5.52]{KM98}. A technical difficulty occurring in our reasoning will
be the fact that $D_0$ need not be Cartier, so that one cannot apply
adjunction na\"ively. It is generally not even true that $K_{D_0}$ or $K_{D_0}
+ (D-D_0)|_{D_0}$ are $\mathbb Q$-Cartier. In particular, it does not make
sense to say that $(D_0, (D-D_0)|_{D_0})$ is klt. However, a more elaborate
adjunction procedure, which involves a correction term $\Diff_{D_0}(0)$ that
accounts for the failure of $D_0$ to be Cartier, is known to give the
following.

\begin{lem}[Existence of a divisor making $D_0$ klt]
  There exists an effective $\mathbb Q$-Weil divisor $\Diff_{D_0}(0)$ on $D_0$
  such that the pair $\bigl(D_0, \Diff_{D_0}(0)\bigr)$ is klt.
\end{lem}
\begin{proof}
  The divisor $D_0$ being normal, it follows from the Adjunction Formula for Weil
  divisors on normal spaces, \cite[Chapt.~16 and Prop.~16.5]{FandA92} see also
  \cite[Sect.~3.9 and Glossary]{MR2352762}, that there exists an effective $\mathbb
  Q$-Weil divisor $\Diff_{D_0}(0)$ on $D_0$ such that $ K_{D_0} + \Diff_{D_0}(0)$ is
  $\mathbb Q$-Cartier and such that the following $\mathbb Q$-linear equivalence
  holds,
  $$
  K_{D_0} + \Diff_{D_0}(0) \sim_{\mathbb Q} \bigl(K_X + D_0\bigr)|_{D_0}.
  $$

  Better still, since $D_0$ is irreducible, it follows from
  \cite[Prop.~5.51]{KM98} that the pair $(X, D_0)$ is actually plt, and
  \cite[Thm.~17.6]{FandA92} then gives that the pair $\bigl(D_0,
  \Diff_{D_0}(0)\bigr)$ is klt, as claimed.
\end{proof}

\subsection{Simplifications and notation}
\label{ssec:S3.3}

Observe that Claim~\ref{claim:mos18A} may be checked locally on
$X_\lambda$. Better still, we may always replace $X_\lambda$ with an open
subset $X_\lambda^\circ \subseteq X_\lambda$, as long as $X_\lambda^\circ \cap
\pi(D_0) \not = \emptyset$. In complete analogy with the arguments of
Section~\ref{ssec:14c2}, we may therefore assume the following.

\begin{awlog}\label{awlogPoS3S1}
  The variety $X_\lambda$ is affine. The image $T := \lambda(D_0)$, taken with
  its reduced structure, is smooth and has a free sheaf of differentials,
  $\Omega^1_T \simeq \sO_T^{\oplus \dim T}$.
\end{awlog}

Note that, as in Section~\ref{sssect:projective}, Assumption~\ref{awlogPoS3S1}
allows to apply Noether normalisation to the affine variety $T$.  Shrinking
$X_\lambda$ further, and performing an étale base change, if necessary,
Proposition~\ref{prop:projection} thus allows to assume the following.

\begin{awlog}
  There exists a commutative diagram of surjective morphisms
  $$
  \xymatrix{ X \ar[rr]_{\lambda} \ar@/^.5cm/[rrrr]^{\psi} && X_\lambda
    \ar[rr]_{\phi} && T }
  $$
  where the restriction $\phi|_{\lambda(D_0)} : \lambda(D_0) \to T$ is
  isomorphic.
\end{awlog}

\begin{notation}
  If $t \in T$ is any point, we consider the scheme-theoretic fibres $X_t :=
  \psi^{-1}(t)$, $X_{\lambda,t} := \phi^{-1}(t)$ and $D_{0,t} :=
  (\psi|_{D_0})^{-1}(t)$.
\end{notation}

Shrinking $T$ ---and thereby $X_\lambda$--- yet further, if necessary, the
Cutting-Down Lemma~\ref{lem:cuttingDown2} allows to assume that the
appropriate fibre pairs are again dlt or klt.  More precisely, we may assume
that the following holds.

\begin{awlog}\label{awlog:phwfsT1}
  If $t \in T$ is any point, then $X_t$ and $X_{\lambda, t}$ are normal. The
  pairs $(X_t, D_{0,t})$ and $(X_{\lambda, t}, \emptyset)$ are dlt, and $\bigl
  (D_{0,t}, \Diff_{D_0}(0) \cap D_{0,t} \bigr)$ are klt.
\end{awlog}

Remark~\ref{rem:dlt} asserts that $(X, D_0)$ and $(X, \emptyset)$ are both
dlt.  Theorems~\ref{thm:relativedifferentialfiltration} and
\ref{thm:relativereflexiveresidue} therefore apply, showing the existence of a
filtration for relative reflexive differentials and the existence of a residue
map over a suitable open set of $T$. Shrinking $T$ again, we may thus assume
that the following holds.

\begin{awlog}\label{awlog:phwfsT2}
  The conclusions of Theorems~\ref{thm:relativedifferentialfiltration} and
  \ref{thm:relativereflexiveresidue} hold for the pairs $(X,D_0)$ and $(X,
  \emptyset)$ without further shrinking of $T$.
\end{awlog}

\subsection{Vanishing of relative reflexive differentials on $\pmb{D_0}$}
\label{ssec:S3.4}

As we have seen in Section~\ref{sec:extwopoles-intro}, the non-existence of
reflexive differentials on $D_0$ is an important ingredient in the proof of
Theorem~\ref{thm:main}. Unlike the setup of
Section~\ref{sec:extwopoles-intro}, we do not assume that $D_0$ maps to a
point, and a discussion of relative reflexive differentials is needed.

\begin{lem}[Vanishing of reflexive differentials on $D_{0,t}$]\label{lem:noreldiff1}
  If $t \in T$ is any point, then $H^0 \bigl( D_{0,t},\,
  \Omega^{[q]}_{D_{0,t}} \bigr) = 0$ for all $1 \leq q \leq n$.
\end{lem}
\begin{proof}
  Let $t \in T$ be any point and recall from \cite[Cor.~1.3(2)]{HMcK07} that
  $D_{0,t}$, which is a fibre of the map $\lambda|_{X_t} : X_t \to
  X_{\lambda,t}$, is rationally chain connected.  Since we argue under the
  inductive hypothesis that Proposition~\ref{prop:kltRCC-n}($n'$) holds for
  all numbers $n' \leq n$ and since the pair $\bigl(D_{0,t}, \Diff_{D_0}(0)
  \cap D_{0,t} \bigr)$ is klt by Assumption~\ref{awlog:phwfsT1}, we obtain the
  vanishing $H^0 \bigl( D_{0,t},\, \Omega^{[q]}_{D_{0,t}} \bigr) = 0$ for all
  $1 \leq q \leq n$, ending the proof.
\end{proof}

\begin{lem}[Vanishing of relative reflexive differentials on $D_0$]\label{lem:noreldiff2}
  We have $H^0 \bigl( D_0,\, \Omega^{[q]}_{D_0/T} \bigr) = 0$ for all $1 \leq
  q \leq n$.
\end{lem}
\begin{proof}
  We argue by contradiction and assume that there exists a non-zero section
  $\tau \in H^0 \bigl( D_0,\, \Omega^{[q]}_{D_0/T} \bigr)$. Let $D_0^\circ
  \subseteq D_0$ be the maximal open subset where the morphism $\psi|_{D_0}$
  is smooth, and let $Z := D_0 \setminus D_0^\circ$ be its complement. As
  before, set $D_{0,t}^\circ := D_{0,t} \cap D_0^\circ$ and $Z_t := D_{0,t}
  \cap Z$. Since $D_0$ is normal, it is clear that $\codim_{D_0} Z \geq 2$. If
  $t \in T$ is a general point, it is likewise clear that $\codim_{D_{0,t}}
  Z_t \geq 2$.

 If $t \in T$ is general, the restriction of the non-zero section $\tau$
  to $D_{0,t}^\circ$ does not vanish,
  \begin{equation}\label{eq:restT1}
  \tau|_{D_{0,t}^\circ} \in H^0 \bigl( D_{0,t}^\circ,\,
  \Omega^{[q]}_{D_0^\circ/T}|_{D_{0,t}^\circ} \bigr) \setminus \{0\}.
  \end{equation}
  However, since $\psi|_{D_0}$ is smooth along $D_{0,t}^\circ$, and since
  $\codim_{D_{0,t}} Z_t \geq 2$, we have isomorphisms
  \begin{equation}\label{eq:restT2}
    H^0 \bigl( D_{0,t}^\circ,\, \Omega^{[q]}_{D_0^\circ/T}|_{D_{0,t}^\circ}
    \bigr) \simeq H^0 \bigl( D_{0,t}^\circ,\, \Omega^{q}_{D_{0,t}^\circ} \bigr)
    \simeq H^0 \bigl( D_{0,t},\, \Omega^{[q]}_{D_{0,t}} \bigr).
  \end{equation}
  But Lemma~\ref{lem:noreldiff1} asserts that the right-hand side
  of~\eqref{eq:restT2} is zero, contradicting~\eqref{eq:restT1}. The
  assumption that there exists a non-zero section $\tau$ is thus absurd, and
  Lemma~\ref{lem:noreldiff2} follows.
\end{proof}

\subsection{Proof of Claim~\ref*{claim:mos18A} if $\pmb{\dim \pi(D_0) = 0}$}
\label{sec:S3.1}

Theorem~\ref{thm:relativereflexiveresidue} assert that a residue map
$$
\rho^{[p]} : \Omega^{[p]}_X(\log D_0) \to \Omega^{[p-1]}_{D_0}
$$
exists. Since $p > 1,$ Lemma~\ref{lem:noreldiff1} implies
$$
H^0 \bigl( D_0,\, \Omega^{[p-1]}_{D_0} \bigr) = 0,
$$
so that $\rho^{[p]}(\sigma) = 0$. As observed in Remark~\vref{rem:poletest},
this shows that $\sigma \in H^0 \bigl( X,\, \Omega^{[p]}_X \bigr)$, finishing
the proof of Proposition~\ref{prop:extwopoleslem} in case $\dim \pi(D_0) =
0$. \qed

\subsection{Proof of Claim~\ref*{claim:mos18A} if $\pmb{\dim \pi(D_0) > 0}$}
\label{sec:S3.2}

The proof of Claim~\ref{claim:mos18A} in case $\dim \pi(D_0) > 0$ is at its
core rather similar to the arguments of the preceding
Section~\ref{sec:S3.1}. However, rather than applying the residue sequence
directly to obtain a reflexive differential on $D_0$, we need to discuss the
filtrations induced by relative differentials. Dealing with reflexive sheaves
on singular spaces poses a few technical problems which will be discussed
---and eventually overcome--- in the following few subsections.

\subsubsection{Relating Claim~\ref*{claim:mos18A} to the reflexive restriction of $\sigma$}
\label{ssec:relclaimtrros}

To prove Claim~\ref{claim:mos18A}, we need to show that $\sigma \in H^0 \bigl(
X,\, \Omega^{[p]}_X \bigr)$. Since all sheaves in question are torsion-free,
this may be checked on any open subset of $X$ which intersects $D_0$
non-trivially. To be more specific, let $X^\circ \subseteq X$ be the maximal
open set where the pair $(X, D_0)$ is snc, and where the morphism $\psi$ is an
snc morphism both of $(X, \emptyset)$ and of $(X, D_0)$. To prove
Claim~\ref{claim:mos18A}, it will then suffice to show that $\sigma|_{X^\circ}
\in H^0 \bigl( X^\circ,\, \Omega^{p}_{X^\circ} \bigr)$.

We aim to study $\sigma$ by looking at its restriction $\sigma|_{D_0^\circ}$,
where $D_0^\circ := D_0 \cap X^\circ$. The restriction is governed by the
following commutative diagram, whose first row is the standard residue
sequence~\eqref{eq:stdResidue}. The second row is the obvious restriction to
$D_0^\circ$,
$$
\xymatrix{ %
  0 \ar[r] & %
  \Omega^{p}_{X^\circ} \ar[rr]^(.4){\gamma} \ar[d]_{\text{restriction}} && %
  \Omega^{p}_{X^\circ}(\log D_0^\circ) \ar[rr] \ar[d]_{\text{restriction}}
  && \Omega^{p-1}_{D_0^\circ} \ar[d]^{=} \ar[r] & 0\\
  & %
  \Omega^{p}_{X^\circ}|_{D_0^\circ} \ar[rr]^(.4){\gamma|_{D_0^\circ}} && %
  \Omega^p_{X^\circ}(\log D_0^\circ)|_{D_0^\circ} \ar[rr] &&
  \Omega^{p-1}_{D_0^\circ} \ar[r] & 0.}
$$
A quick diagram chase thus reveals that to show $\sigma|_{X^\circ} \in H^0
\bigl( X^\circ,\, \Omega^{p}_{X^\circ} \bigr)$, it suffices to show that the
restriction of $\sigma|_{D_0^\circ}$ comes from
$\Omega^{p}_{X^\circ}|_{D_0^\circ}$. More precisely, we see that to prove
Claim~\ref{claim:mos18A} it suffices to show that
\begin{equation}\label{eq:amen}
  \sigma|_{D_0^\circ} \in \Image \left[ \gamma|_{D_0^\circ} : H^0 \bigl(
    D_0^\circ,\, \Omega^{p}_{X^\circ}|_{D_0^\circ} \bigr) \to H^0 \bigl(
    D_0^\circ,\, \Omega^{p}_{X^\circ}(\log D_0^\circ)|_{D_0^\circ}\bigr)
  \right].
\end{equation}

Next, we aim to express the inclusion in~\eqref{eq:amen} in terms of reflexive
differentials which are globally defined along the divisor $D_0$, making the
statement more amenable to the methods developed in Part~\ref{part:2} of this
paper. To this end, observe that
$$
(\Omega^{[p]}_X|_{D_0}^{**})|_{D_0^\circ} \simeq
\Omega^p_{X^\circ}|_{D_0^\circ} \text{\,\, and \,\,} (\Omega^{[p]}_X(\log
D_0)|_{D_0}^{**})|_{D_0^\circ} \simeq \Omega^p_{X^\circ}(\log
D_0^\circ)|_{D_0^\circ}.
$$
Thus, if $\wtilde \sigma_{D_0} \in H^0 \bigl( D_0,\, \Omega^{[p]}_X(\log
D_0)|_{D_0}^{**}\bigr)$ denotes the image of $\sigma|_{D_0}$ in the reflexive
hull of $\Omega^{[p]}_X(\log D_0)|_{D_0}$, then the inclusion
in~\eqref{eq:amen} will hold if we show that
\begin{equation}\label{eq:S18D1inctilde}
  \wtilde \sigma_{D_0} \in \Image \left[ H^0 \bigl( D_0,\,
    \Omega^{[p]}_X|_{D_0}^{**}\bigr) \to H^0 \bigl( D_0,\,
    \Omega^{[p]}_X(\log D_0)|_{D_0}^{**}\bigr) \right].
\end{equation}
We will show more, namely, that $\wtilde \sigma_{D_0}$ is not only in the
image of the sheaf $\Omega^{[p]}_X|_{D_0}^{**}$, but that it is already in the
image of the subsheaf $\psi^*\Omega^r_T|_{D_0}$. The following
lemma will be useful in the formulation of that claim.

\begin{lem}\label{lem:S18D1}
  The natural inclusions $\psi^*\Omega^p_T \hookrightarrow
  \Omega^{[p]}_X \hookrightarrow \Omega^{[p]}_X(\log D_0)$ yield a diagram of
  sheaves as follows,
  \begin{equation}\label{eq:bonny}
    \xymatrix{
      \psi^*\Omega^p_T|_{D_0} \ar@/^6mm/[rrrr]^{\beta, \text{ injective}}
      \ar[rr] && %
      \Omega^{[p]}_{X}|_{D_0}^{**} \ar[rr] &&
      \Omega^{[p]}_{X}(\log D_0)|_{D_0}^{**}. }
  \end{equation}
\end{lem}
\begin{proof}
  Assumption~\ref{awlog:phwfsT2} allows to apply
  Theorem~\ref{thm:relativedifferentialfiltration} from
  Page~\pageref{thm:relativedifferentialfiltration} (existence of relative
  differential sequences) to the sheaves $\Omega^{[p]}_{X}$ and
  $\Omega^{[p]}_{X}(\log D_0)$, obtaining a commutative diagram of injective
  sheaf morphisms,
  \begin{equation}\label{eq:clyde}
    \xymatrix{ %
      {\protect{\phantom{(\log)}}}\sF^{[p]} = \psi^*\Omega^p_T \ar[rr]
      \ar@<10mm>@{=}[d]&& %
      \Omega^{[p]}_{X} \ar[d] \\ %
      \sF^{[p]}(\log) = \psi^*\Omega^p_T \ar[rr] && %
      \Omega^{[p]}_{X}(\log D_0).
    }
  \end{equation}
  The diagram~\eqref{eq:bonny} is obtained by restricting \eqref{eq:clyde} to
  $D_0$ and taking double duals. Injectivity of the map $\beta$ follows from a
  repeated application of Corollary~\ref{cor:rrdsbc} to the sheaf
  $\psi^*\Omega^p_T = \sF^{[p]}(\log)|_{D_0}^{**}$. This finishes the proof of
  Lemma~\ref{lem:S18D1}.
\end{proof}

Returning to the proof of Claim~\ref*{claim:mos18A}, observe that
Lemma~\ref{lem:S18D1} allows us to view $\psi^*\Omega^p_T|_{D_0}$ as a
subsheaf
$$
\psi^*\Omega^p_T|_{D_0} \subseteq \Image \left[
(\Omega^{[p]}_X|_{D_0})^{**} \to (\Omega^{[p]}_X(\log D_0)|_{D_0})^{**}
\right].
$$
With this notation, to prove the inclusion in~\eqref{eq:S18D1inctilde}, it is
thus sufficient to prove the following claim.

\begin{claim}[Proves \eqref{eq:S18D1inctilde} and hence Proposition~\ref{prop:extwopoleslem}($n+1$)]\label{claim:mos18B}
  The section $\wtilde \sigma_{D_0}$ comes from $T$. More precisely, we claim
  that we have inclusions
  $$
  \wtilde \sigma_{D_0} \in H^0 \bigl( D_0,\,
  \underbrace{\psi^*\Omega^p_T|_{D_0}}_{= \sF^{[p]}(\log)|_{D_0}^{**}} \bigr)
  \subseteq H^0 \bigl( D_0,\, (\Omega^{[p]}_X(\log D_0)|_{D_0})^{**}\bigr).
  $$
\end{claim}

\subsubsection{Filtrations induced by relative differentials and their inclusions}

Recall from Assumption~\ref{awlog:phwfsT2},
Theorem~\ref{thm:relativedifferentialfiltration} and
Corollary~\ref{cor:rrdsbc} that there exists a filtration of
$\Omega^{[p]}_{X}(\log D_0)$,
$$
  \Omega^{[p]}_{X}(\log D_0) = \sF^{[0]}(\log) \supseteq \sF^{[1]}(\log)
   \supseteq \cdots  \supseteq \sF^{[p]}(\log) \supseteq \{0\}
$$
giving rise to exact sequences
$$
0 \to \sF^{[r+1]}(\log)|_{D_0}^{**} \to \sF^{[r]}(\log)|_{D_0}^{**} \to
\psi^*\Omega_T^r \otimes \Omega_{X/T}^{[p-r]}(\log D_0)|_{D_0}^{**}.
$$
Since $\psi^*\Omega^p_T$ is a trivial vector bundle, we see that
to prove Claim~\ref{claim:mos18B} it is sufficient to prove the following.

\begin{claim}\label{claim:mos18C}
  For all numbers $q>0$, we have $H^0\bigl( D_0, \, \Omega_{X/T}^{[q]}(\log
  D_0)|_{D_0}^{**} \bigr) = 0$.
\end{claim}

\subsubsection{Proof of Claim~\ref*{claim:mos18C} in case $q=1$}
\label{ssec:donnerunddoria}

We argue by contradiction and assume that there exists a non-zero section
$\tau \in H^0\bigl( D_0, \, \Omega_{X/T}^{[1]}(\log D_0)|_{D_0}^{**} \bigr)$.

We maintain the notation introduced in Section~\ref{ssec:relclaimtrros}.  If
$t \in T$ is general, the section $\tau$ will then induce a non-zero section
\begin{equation}\label{eq:betatau0}
  \tau|_{D_{0,t}^\circ} \in H^0\bigl( D_{0,t}^\circ, \, \Omega_{X/T}^{[1]}(\log
  D_0)|_{D_{0,t}^\circ}^{**} \bigr) = H^0\bigl( D_{0,t}^\circ, \,
  \Omega_{X_t}^1(\log D_{0,t}^\circ)|_{D_{0,t}^\circ} \bigr).
\end{equation}
On the other hand, let $\beta$ be the composition of the following canonical
morphisms
\begin{multline*}
  H^0 \bigl( X,\, \Omega^{[1]}_{X/T}(\log D_0)|_{X_t}^{**} \bigr)
  \xrightarrow[\text{restr.~to }X^\circ]{\simeq} H^0 \bigl( X^\circ,\,
  \Omega^1_{X^\circ/T}(\log D_0^\circ)\bigr){\longrightarrow} \\
  \xrightarrow[\text{restr.~to }X_t^\circ]{} H^0 \bigl( X^\circ_t,\,
  \Omega^1_{X^\circ/T}(\log D_0^\circ)|_{X^\circ_t} \bigr)
  \xrightarrow[\psi|_{X^\circ}\text{ is snc}]{\simeq} H^0 \bigl(
  X^\circ_t,\, \Omega^1_{X_t^\circ}(\log D_{0,t}^\circ) \bigr)
  \longrightarrow\\
  \xrightarrow[\text{restr.~to }X_t^\circ]{\simeq} H^0 \bigl( X_t,\,
  \Omega^{[1]}_{X_t}(\log D_{0,t}) \bigr) \xrightarrow[\text{restr.~to
  }D_{0,t}]{} H^0 \bigl( D_{0,t},\, \Omega^{[1]}_{X_t}(\log
  D_{0,t})|_{D_{0,t}} \bigr)
  \longrightarrow \\
  \xrightarrow[\text{refl.~hull}]{} H^0 \bigl( D_{0,t},\,
  \Omega^{[1]}_{X_t}(\log D_{0,t})|_{D_{0,t}}^{**} \bigr).
\end{multline*}
Then a comparison with~\eqref{eq:betatau0} immediately shows that
$\beta(\tau)|_{D_{0,t}^\circ} \not = 0$. In particular, we obtain that
\begin{equation}\label{eq:hnnz}
  H^0 \bigl( D_{0,t},\, \Omega^{[1]}_{X_t}(\log D_{0,t})|_{D_{0,t}}^{**} \bigr)
  \not = 0.
\end{equation}
On the other hand, Theorem~\vref{thm:Chernclass} (description of Chern class
by residue sequence) shows that there exists a smooth open subset
$D_{0,t}^{\circ\circ} \subseteq D_{0,t}$ with small complement,
\begin{equation}\label{eq:cdoimd00}
  \codim_{D_{0,t}} \bigl( D_{0,t} \setminus D_{0,t}^{\circ\circ}\bigr) \geq 2,
\end{equation}
and an exact sequence,
\begin{multline}\label{ml:ccc}
  0 \to \underbrace{H^0\bigl( D_{0,t}^{\circ\circ},\,
    \Omega^1_{D_{0,t}^{\circ\circ}} \bigr)}_{=: \sf A} \to \underbrace{H^0
    \bigl( D_{0,t}^{\circ\circ},\, \Omega^{[1]}_{X_t}(\log
    D_{0,t})|_{D_{0,t}^{\circ\circ}}^{**} \bigr)}_{=: \sf B} \to \\
  \to \underbrace{H^0 \bigl( D_{0,t}^{\circ\circ},\,
    \sO_{D_{0,t}^{\circ\circ}} \bigr)}_{=: \sf C} \xrightarrow{\delta} H^1
  \bigl( D_{0,t}^{\circ\circ},\, \Omega^1_{D_{0,t}^{\circ\circ}} \bigr) \to
  \cdots,
\end{multline}
where $\delta(m \cdot \mathbf{1}) = c_1\bigl(\sO_{D_{0,t}^{\circ\circ}}(m\cdot
D_{0,t}^{\circ\circ})\bigr)$, for $m$ sufficiently large and
divisible. Observing that
\begin{align*}
  {\sf A} & \simeq H^0\bigl( D_{0,t},\, \Omega^{[1]}_{D_{0,t}} \bigr) = 0 &
  \text{\eqref{eq:cdoimd00} and \eqref{lem:noreldiff1}}\\
  {\sf B} & \simeq H^0 \bigl( D_{0,t},\, \Omega^{[1]}_{X_t}(\log
  D_{0,t})|_{D_{0,t}}^{**} \bigr) \not = 0 &
  \text{\eqref{eq:cdoimd00} and~\eqref{eq:hnnz}} \\
  {\sf C} & \simeq H^0 \bigl( D_{0,t},\, \sO_{D_{0,t}} \bigr) \simeq \mathbb C
  &\text{\eqref{eq:cdoimd00}}
\end{align*}
The sequence~\eqref{ml:ccc} immediately implies that
$c_1\bigl(\sO_{D_{0,t}^{\circ\circ}}(mD_{0,t}^{\circ\circ})\bigr) = 0$. That,
however, cannot be true, as the contraction $\lambda|_{X_t} : X_t \to
X_{\lambda,t}$ contracts the divisor $D_{0,t} \subset X_t$ to a point, so that
Assertion~(\ref{lem:excdiv}.\ref{el:excdiv2}) of the Negativity
Lemma~\ref{lem:excdiv} implies that $D_{0,t}$ is actually $\mathbb
Q$-anti-ample, relatively with respect to the contraction morphism
$\lambda|_{X_t}$. By the inequality \eqref{eq:cdoimd00}, it is then also clear
that $c_1\bigl(\sO_{D_{0,t}^{\circ\circ}}(m D_{0,t}^{\circ\circ})\bigr) \in
H^2\bigl( D_{0,t}^{\circ\circ},\, \mathbb R \bigr)$ cannot be zero. In fact,
choose a complete curve $C \subset D_{0,t}^{\circ\circ}$ and observe that the
restriction $\sO_{D_{0,t}^{\circ\circ}}(m D_{0,t}^{\circ\circ})|C$ is a
negative line bundle.  We obtain a contradiction which shows that the original
assumption about the existence of a non-zero section $\tau$ was absurd. This
completes the proof of Claim~\ref{claim:mos18C} in case $q=1$. \qed

\subsubsection{Proof of Claim~\ref*{claim:mos18C} in case $q>1$}

Using Assumption~\ref{awlog:phwfsT2} and applying the left-exact section functor
$\Gamma$ to the residue sequence
(\ref{thm:relativereflexiveresidue}.\ref{il:RelReflResidue}) constructed in
Theorem~\ref{thm:relativereflexiveresidue}, we obtain an exact sequence,
$$
0 \to \underbrace{H^0\bigl( D_0, \, \Omega^{[q]}_{D_0/T} \bigr)}_{= 0 \text{
    by Lemma~\ref{lem:noreldiff2}}} \to H^0\bigl( D_0, \,
\Omega_{X/T}^{[q]}(\log D_0)|_{D_0}^{**} \bigr) \to \underbrace{H^0\bigl( D_0,
  \, \Omega^{[q-1]}_{D_0/T} \bigr)}_{=0 \text{ by
    Lemma~\ref{lem:noreldiff2}}},
$$
and Claim~\ref*{claim:mos18C} follows immediately. This finishes the proof of
Proposition~\ref{prop:extwopoleslem} in case $\dim \pi(D_0) > 0$. \qed

\part{Appendix}

\appendix

\PreprintAndPublication{
\section{Effective linear combinations of exceptional divisors}\label{app:A}

The following Negativity Lemma is well-known to experts, and variants are
found in the literature. Since the Negativity Lemma is central to our
arguments, we reproduce a full proof here for the reader's convenience.

\begin{lem}[\protect{Negativity Lemma for exceptional divisors, cf.~\cite[Lem.~3.6.2]{BCHM06}}]\label{lem:excdiv}
  Let $\pi: \wtilde X \to X$ be a birational, projective and surjective
  morphism between irreducible and normal quasi-projective varieties.
  \begin{enumerate}

  \item\label{el:excdiv0} If $D$ is a $\pi$-exceptional $\mathbb Q$-divisor on
    $\wtilde X$ which is $\mathbb Q$-Cartier and $\pi$-anti-ample, then $D$ is
    effective, and $\supp(D) = E$.

  \item\label{el:excdiv2} If $X$ is $\mathbb Q$-factorial, then there exists
    an effective and $\pi$-anti-ample Cartier divisor $D$ on $\wtilde X$ with
    $\supp(D) = E$. In particular, the $\pi$-exceptional set is of pure
    codimension one in $\wtilde X$.

  \item\label{el:excdiv1} If $D \subset \wtilde X$ is any non-trivial
    effective $\mathbb Q$-Cartier divisor with $\supp(D) \subseteq E$, then
    $D$ is not $\pi$-nef.
      \end{enumerate}
\end{lem}
\begin{proof}[Proof of (\ref{lem:excdiv}.\ref{el:excdiv0})]
  Since (\ref{lem:excdiv}.\ref{el:excdiv0}) is local on $X$, we may assume
  without loss of generality that $X$ is affine, and that there exists a
  number $m \in \mathbb N$ such that the divisor $mD$ is integral,
  Cartier, and such that the linear system $|-mD|$ is relatively
  basepoint-free.

  Decompose $D= D_{\rm pos} - D_{\rm neg}$, where $D_{\rm pos}$ and $D_{\rm
    neg}$ are both effective, and do not share a common component. A section
  $\sigma \in H^0\bigl( \wtilde X,\, \sO_{\wtilde X}(-mD)\bigr)$ is then seen
  as a rational function $f_{\sigma}$ on $\wtilde X$ with prescribed zeros
  along $D_{\rm pos}$, and possibly with poles of bounded order along $D_{\rm
    neg}$. It is, however, clear that $f_\sigma$ cannot have any poles at all:
  the function $f_\sigma$, which is certainly regular away from $E$, defines a
  function $g_\sigma$ in $X \setminus \pi(E)$. Since $\codim_X \pi(E) \geq 2$,
  the function $g_\sigma$ will extend to a function which is regular all over
  $X$, and whose pull-back necessarily agrees with $f_\sigma$. In summary, we
  obtain that the linear system $|-mD|$ has basepoints along $D_{\rm neg}$. It
  follows that $D_{\rm neg}=0$.

  It remains to show that $\supp(D)=E$. That, however, follows from the fact
  that the $\pi$-anti-ample divisor $D$ intersects every curve in $C \subset
  E$ negatively if the \change{curve} is mapped to a point in $X$.
\end{proof}

\begin{proof}[Proof of (\ref{lem:excdiv}.\ref{el:excdiv2})]
  Let $D' \subset \wtilde X$ be any divisor which is $\pi$-anti-ample; $D'$
  exists because the morphism $\pi$ is assumed to be projective. By
  assumption, there exists a number $m$ such that $m$ times the
  cycle-theoretic image $\pi_*D'$ is Cartier.  The divisor $D := m D' -
  \pi^*(m\pi_*D')$ is then $\pi$-anti-ample and supported on $E$. Apply
  (\ref{lem:excdiv}.\ref{el:excdiv0}) to conclude that $D$ is effective and
  that $\supp(D) = E$.
\end{proof}

\begin{proof}[Proof of (\ref{lem:excdiv}.\ref{el:excdiv1})]
  Let $d := \dim \pi(E)$. Choose general hypersurfaces $H_1, \ldots, H_d
  \subset X$ and $H_{d+1}, \ldots, H_{\dim X-2} \subset \wtilde X$. Further,
  set
  $$
  H := \pi^{-1}(H_1) \cap \cdots \pi^{-1}(H_d) \cap H_{d+1} \cap \cdots \cap
  H_{\dim X-2} \subset \wtilde X.
  $$
  By Seidenberg's Theorem, \cite[Thm.~1.7.1]{BS95}, the intersection $H$ is
  then a normal surface. Further, it follows from the construction that the
  cycle-theoretic intersection $D_H := D \cap H$ is an effective,
  $\pi|_H$-exceptional $\mathbb Q$-Cartier divisor on $H$. The Hodge-Index
  theorem therefore asserts that $(D_H)^2 < 0$. It follows that there exists a
  curve $C \subseteq \supp D_H \subset H \subset X$ which is contained in the
  $\pi$-exceptional set and intersects $D$ negatively, $D.C < 0$. This
  completes the proof.
\end{proof}
}{}

\section{Finite group actions on coherent sheaves}\label{app:B}
\label{sec:groups}

Let $G$ be a finite group acting on a normal variety $X$. In this appendix, we
consider $G$-sheaves on $X$ and their associated push-forward sheaves on the
quotient space. Some results presented here are well-known to
experts. Lemma~\ref{lem:invariantsexact} is contained for example in the
unpublished preprint \cite{KollarGroupsActingOnSchemes}. However, since we
were not able to find published proofs of any of these result we decided to
include them here in order to keep our exposition as self-contained as
possible.

\begin{defn}[$G$-sheaf and morphisms of $G$-sheaves]\label{def:Gsheaf}
  Let $G$ be a finite group acting on a normal variety $X$. If $g \in G$ is
  any element, we denote the associated automorphism of $X$ by $\phi_g$. A
  $G$-sheaf $\sF$ on $X$ is a coherent sheaf of $\sO_X$-modules such that for
  any open set $U \subseteq X$ is any open set, there exist natural
  push-forward morphisms
  $$
  (\phi_g)_* : \sF(U) \to \sF\bigl( \phi_g(U) \bigr)
  $$
  that satisfy the usual compatibility conditions. A morphism $\alpha : \sF
  \to \sG$ of $G$-sheaves is a sheaf morphism such that for any open set $U$
  and any element $g \in G$, then there are commutative diagrams
  $$
  \xymatrix{%
    \sF(U) \ar[rr]^(.45){(\phi_g)_*} \ar[d]_{\alpha(U)} && \sF\bigl( \phi_g(U)
    \bigr) \ar[d]^{\alpha (\phi_g(U))} \\
    \sG(U) \ar[rr]^(.45){(\phi_g)_*} && \sG\bigl( \phi_g(U) \bigr).}
  $$
\end{defn}

\begin{defn}[Invariant sheaves]
  If $G$ acts trivially on $X$, and if $\sF$ is any $G$-sheaf, the associated
  sheaf of invariants, denoted $\sF^G$, is the sheaf associated to the
  complete presheaf
  $$
  \sF^G(U) := \bigl( \sF(U) \bigr)^G
  $$
  where $\bigl( \sF(U) \bigr)^G$ denotes the submodule of $G$-invariant
  elements of the $\sO_X(U)$-module $\sF(U)$.
\end{defn}

In the remainder of the present Section~\ref{sec:groups}, we consider the
setup where $G$ acts on $X$, with quotient morphism $q: X \to X/G$. Let $\sG$
be a coherent $G$-sheaf of $\sO_X$-modules. Then, the push-forward $q_*\sG$ is
a $G$-sheaf on $X/G$ for the trivial $G$-action on $X/G$, and its associated
sheaf of invariants will be denoted by $(q_*\sG)^G$.  The following lemmas
collect fundamental properties of the functor $q_*(\cdot)^G$.

\begin{lem}[Exactness Lemma]\label{lem:invariantsexact}
  Let $G$ be a finite group acting on a normal variety $X$, and let $q: X \to
  X/G$ be the quotient morphism. Let $\sG$ be a coherent $G$-sheaf of
  $\sO_X$-modules.  Then, the $G$-invariant push-forward $(q_* \sG)^G$ is a
  coherent sheaf of $\sO_{X/G}$-modules. Furthermore, if
  $$
  0 \to \sF \to\sG \to \sH \to 0
  $$
  is a $G$-equivariant exact sequence of $\sO_X$-modules, the induced sequence
  \begin{equation}\label{eq:exct}
    0 \to (q_*\sF)^G \to (q_* \sG)^G \to (q_* \sH)^G \to 0
  \end{equation}
  is likewise exact.
\end{lem}
\begin{proof}
  The sequence~\eqref{eq:exct} is clearly left-exact. For right-exactness, it
  follows from a classical result of Maschke \cite{Maschke} that any finite
  group in characteristic zero is linearly reductive. In other words, any
  finite-dimensional representation of $G$ splits as a direct sum of
  irreducible $G$-subrepresentations. It follows that for every
  $G$-representation $V$, there exists a Reynolds operator, i.e., a
  $G$-invariant projection $R: V \twoheadrightarrow V^G$, see for example,
  \cite[Sect.~V-2]{Fogarty}. It follows that $V^G$ is a direct summand of $V$.

  So, if $\sG$ is any coherent $G$-sheaf on $X$, it follows from the above
  that $(q_*\sG)^G$ is a direct summand of the coherent sheaf $q_*(\sG)$ on
  $X/G$. Consequently, $(q_*\sG)^G$ is likewise coherent.

  Another consequence of the existence of the Reynolds operator is that for
  every $G$-equivariant map $\varphi: V \to W$ between (not necessarily
  finite-dimensional) $G$-representations, the induced map $\varphi^G: V^G \to
  W^G$ between the subspaces of invariants is still surjective. This shows
  right-exactness of \eqref{eq:exct} and implies the claim.
\end{proof}

\begin{lem}[Reflexivity Lemma]\label{lemma:reflexivepushforward}
  Let $G$ be a finite group, $X$ a normal $G$-variety, and $\sG$ a reflexive
  coherent $G$-sheaf. Then, the $G$-invariant push-forward $(q_*\sG)^G$ is
  also reflexive.
\end{lem}
\begin{proof}
  We have to show that $(q_*\sG)^G$ is torsion-free and normal. Since $\sG$ is
  torsion-free, $q_*\sG$ is torsion-free, and hence $(q_*\sG)^G$ is torsion-free as a
  subsheaf of $q_*\sG$. To prove normality, let $U$ be an affine open subset of $X/G$
  and $Z \subset U$ a closed subvariety of codimension at least $2$. Let
  $$
  s \in H^0\bigl(U\setminus Z,\ (q_*\sG)^G
  \bigr)= H^0\bigl (q^{-1}(U)\setminus q^{-1}(Z),\ \sG \bigr)^G.
  $$
  Since $q$ is finite, $q^{-1}(Z)$ has codimension at least $2$ in
  $q^{-1}(U)$. Since $\sG$ is reflexive, hence normal, the section $s$ extends
  to a $G$-invariant section of $\sG$ over $q^{-1}(U)$.
\end{proof}

\begin{lem}[Splitting Lemma]\label{lem:splittings}
  Let $G$ be a finite group acting on a normal variety $X$ with quotient $q: X
  \to X/G$. Let
  \begin{equation}\label{eq:exactequivariant}
    0 \to \sH \to \sF \to \sG \to 0
  \end{equation}
  be a $G$-equivariant exact sequence of locally free $G$-sheaves on
  $Y$. Then, the induced exact sequence
  \begin{equation}
    0 \to (q_*\sH)^G \to (q_* \sF)^G \to (q_*\sG)^G \to 0
  \end{equation}
  is locally split in the analytic topology.
\end{lem}
\begin{proof}
  \PreprintAndPublication{
    \begin{figure}
      \centering

      \ \\

      $$
      \xymatrix{
        \begin{picture}(3.5, 3.2)(0,0)
          \put( 0.0, 1.7){open set $V$}
          \put( 1.0, 0.7){\includegraphics[width=1.5cm]{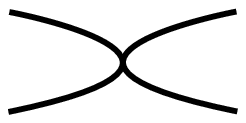}}
        \end{picture}
        \ar@<-7mm>[rr]^{\iota}_{\text{inclusion}}
        \ar[d]_(.7){q'}^(.7){\text{quotient map}} &&
        \quad
        \begin{picture}(3.5, 3.2)(0,0)
          \put( 0.0, 3.5){normal space $X$}
          \put( 0.0, 2.0){\includegraphics[width=3.5cm]{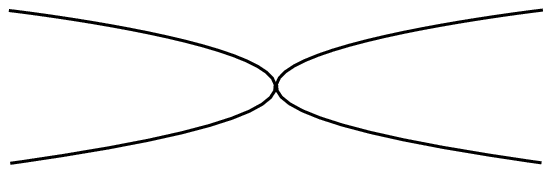}}
          \put( 0.0, 0.5){\includegraphics[width=3.5cm]{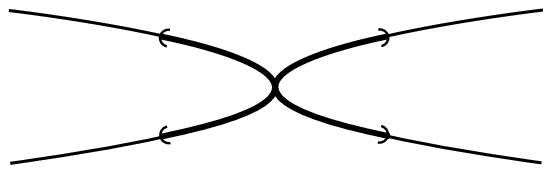}}
          \put( 1.67, 0.92){$\bullet$}
          \put( 1.9, 1.0){\scriptsize $x$}
          \put( 2.1, 1.4){\scriptsize $V$}
        \end{picture}
        \ar[d]_(.7){q}^(.7){\text{quotient map}} \\
        \begin{picture}(3.5, 0.7)(0,0)
          \put( 0.0, 0.3){$V/G_x$}
          \put( 1.0, 0.0){\includegraphics[width=1.5cm]{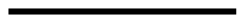}}
        \end{picture}
        \ar[rr]^{\bar \imath}_{\text{inclusion}}
        && \quad
        \begin{picture}(3.5, 0.7)(0,0)
          \put( 1.75, 0.2){\scriptsize $z$}
          \put( 1.75, 0.0){\scriptsize $\bullet$}
          \put( 2.25, 0.2){\scriptsize $U$}
          \put( 0.0, 0.3){$X/G$}
          \put( 0.0, 0.0){\includegraphics[width=3.5cm]{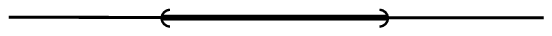}}
        \end{picture}
      }
      $$

      \caption{Setup in the proof of the Splitting Lemma~\ref*{lemma:reflexivepushforward}}
      \label{fig:quot}
    \end{figure}

    As shown in Figure~\vref{fig:quot}, let}{Let} $z \in X/G$ be any point and
  $x \in q^{-1}(z)$ any preimage point, with isotropy group $G_x$. By the
  holomorphic slice theorem, cf.~\cite[Hilfssatz~1]{Holmann} or
  \cite[Sect.~5.5]{HeinznerGIT}, there exists an open Stein neighbourhood $U =
  U(z) \subseteq X/G$ and an open $G_x$-invariant Stein neighbourhood $V=V(x)
  \subseteq X$ such that $q^{-1}(U)$ is $G$-equivariantly biholomorphic to the
  twisted product
  $$
  G \times_{G_x}V := (G \times V)/G_x,
  $$
  where $G_x$ acts on $G \times V$ as
  $$
  \begin{array}{ccc}
    G_x \times (G \times V) & \to & G \times V\\
    \bigl( h,\, (g, v) \bigr) & \mapsto & (gh^{-1}, h \cdot v).
  \end{array}
  $$
  Let $q':V \to V/G_x$ denote the quotient of $V$ by $G_x$. Observe then that
  the inclusion $\imath: V \hookrightarrow q^{-1}(U)$ induces a biholomorphic
  map
  $$
  \bar \imath: V/G_x \overset{\simeq}\longrightarrow U = q^{-1}(U)/G.
  $$

  Shrinking $U$, if necessary, we may assume that the
  sequence~\eqref{eq:exactequivariant} is split on $V$ with splitting $s: \sG
  |_V \to \sF|_V$.  By averaging $s$ over $G_x$ we obtain a sheaf morphism
  $\bar s: \bigl(q_*'(\sG|_V)\bigr)^{G_x} \to \bigl(q_*'(\sF|_V)\bigr)^{G_x}$
  that splits the exact sequence
  $$
  0 \to \bigl(q_*'(\sH|_V)\bigr)^{G_x} \to \bigl(q_*'(\sF|_V)\bigr)^{G_x} \to
  \bigl(q_*'(\sG|_V) \bigr)^{G_x} \to 0.
  $$
  Finally we notice that for any coherent $G$-sheaf $\sS$ on $q^{-1}(U)$,
  the inclusion $\imath$ induces a canonical isomorphism
  $$
  \phi_\sS: \bar \imath^*(q_* \sS)^G
  \overset{\simeq}{\longrightarrow}\bigl(q_*'(\sS|_V) \bigr)^{G_x}.
  $$
  Applying this observation to $\sF$ and $\sG$, we obtain a commutative
  diagram
  $$
  \xymatrix{%
    \bar \imath^*(q_* \sF)^G \ar[r] \ar[d]_{\phi_{\sF}}& \bar \imath^*(q_* \sG)^G
    \ar[d]^{\phi_{\sG}} \\
    \bigl(q_*'(\sF|_V) \bigr)^{G_x} \ar[r]& \ar@/_1pc/[l]_{\bar s}\bigl(q_*'(\sG|_V)
    \bigr)^{G_x} }
  $$
  The map $\phi_\sF \circ \bar s \circ \phi_\sG$ then is the desired splitting.
\end{proof}

\begin{lem}[Restriction Lemma]\label{lem:surjection}
  Let $G$ be a finite group, $X$ a normal $G$-variety, and $\sF$ locally free
  coherent $G$-sheaf on $X$. Let $q: X \to X/G$ be the quotient map, and let
  $\Delta$ be a normal $G$-invariant subvariety of $X$ with image $D =
  q(\Delta)$. Then, we have a canonical surjection
  $$
  ((q_*\sF)^G|_D)^{**} \twoheadrightarrow (q|_\Delta)_*(\sF|_\Delta)^G.
  $$
\end{lem}
\begin{proof}Let $\imath: \Delta \hookrightarrow X$ denote the inclusion.
  Clearly, the restriction morphism $\sF \twoheadrightarrow
  \imath_*(\sF|_\Delta)$ is $G$-equivariant. Since $q$ is finite, we obtain a
  surjection $q_*(\sF) \twoheadrightarrow q_*(\imath_*(\sF|_\Delta))$. The
  Exactness Lemma~\ref{lem:invariantsexact} implies that the induced map of
  invariants $(q_*\sF)^G \rightarrow q_*(\imath_*(\sF|_\Delta))^G$ is still
  surjective. This morphism stays surjective after restriction to $D$, i.e. we
  obtain a surjection
  $$
  \varphi: (q_*\sF)^G|_D \twoheadrightarrow q_*(\imath_*(\sF|_\Delta))^G|_D =
  (q|_\Delta)_*(\sF|_\Delta)^G.
  $$
  Since the restriction $\sF|_\Delta$ is locally free and $\Delta$ is normal
  by assumption, the Reflexivity Lemma~\ref{lemma:reflexivepushforward}
  implies that $(q|_\Delta)_*(\sF|_\Delta)^G$ is reflexive and hence
  torsion-free. As a consequence $\varphi$ factors over the natural map
  $(q_*\sF)^G|_D \to ((q_*\sF)^G|_D)^{**}$. This shows the claim.
\end{proof}

\PreprintAndPublication{
\section{The Four-Lemmas for vector spaces}

For the reader's convenience, we recall the elementary 4-Lemmas from
commutative algebra.

\begin{lem}[\protect{Four-Lemma for injectivity, \cite[XII~Lem.~3.1(i)]{McL75}}]\label{lem:4inj}
  Consider a commutative diagram of linear maps between vector spaces, as
  follows
  $$
  \xymatrix{
    A  \ar[r] \ar@{->>}[d]_a & B \ar[r] \ar@{^{(}->}[d]_b & C \ar[r] \ar[d]_c & D \ar@{^{(}->}[d]_d \\
    A' \ar[r] & B' \ar[r] & C' \ar[r] & D'.  }
  $$
  If the horizontal sequences are exact, then $c$ is injective. \qed
\end{lem}

\begin{lem}[\protect{Four-Lemma for surjectivity, \cite[XII~Lem.~3.1(ii)]{McL75}}]\label{lem:4sur}
  Consider a commutative diagram of linear maps between vector spaces, as
  follows
  $$
  \xymatrix{
    B \ar[r] \ar@{->>}[d]_b & C \ar[r] \ar[d]_c & D \ar[r] \ar@{->>}[d]_d & E \ar@{^{(}->}[d]_e \\
    B' \ar[r] & C' \ar[r] & D' \ar[r] & E'.  }
  $$
  If the horizontal sequences are exact, then $c$ is surjective. \qed
\end{lem}
}{}

\newcommand{\etalchar}[1]{$^{#1}$}
\providecommand{\bysame}{\leavevmode\hbox to3em{\hrulefill}\thinspace}
\providecommand{\MR}{\relax\ifhmode\unskip\space\fi MR}
\providecommand{\MRhref}[2]{%
  \href{http://www.ams.org/mathscinet-getitem?mr=#1}{#2}
}
\providecommand{\href}[2]{#2}

\PreprintAndPublication{
\listoffigures
\listoftables
}

\end{document}